\documentclass{amsart}

\usepackage{amssymb,amsfonts, amsmath, amsthm}
\usepackage[all,arc]{xy}
\usepackage{enumerate}
\usepackage{mathrsfs}
\usepackage{mathtools}
\usepackage[mathscr]{euscript}
\usepackage{array}
\usepackage{tikz}
\usepackage{tikz-cd}
\usepackage{textcomp}
\usetikzlibrary{decorations.pathmorphing,shapes}
\usepackage[colorlinks=true, linkcolor=blue, citecolor = red]{hyperref}
\usetikzlibrary{patterns}

\usepackage{geometry}

\newcommand{\E}{\mathscr{E}}

\newcommand{\Y}{\mathscr{Y}}
\newcommand{\G}{\mathscr{G}}
\renewcommand{\H}{\mathscr{H}}
\newcommand{\Lprime}{\mathscr{L}}

\newcommand{\s}{\mathscr{S}}
\newcommand{\I}{\mathscr{I}}
\newcommand{\M}{\mathscr{M}}
\renewcommand{\O}{\mathscr{O}}

\newcommand{\R}{\mathscr{R}}
\newcommand{\T}{\mathbb{T}}
\newcommand{\U}{\mathscr{U}}
\newcommand{\V}{\mathscr{V}}
\newcommand{\W}{\mathscr{W}}

\newcommand{\comma}{,}
\newcommand{\leb}{[}
\newcommand{\reb}{]}
\newcommand{\set}{\mathscr{S}\text{et}}
\newcommand{\cat}{\mathscr{C}\text{at}}

\newcommand{\Univ}{\mathscr{U}\text{niv}}
\newcommand{\Arr}{\mathscr{A}\text{rr}}

\newcommand{\colim}{\mathrm{colim}}
\newcommand{\ds}{\displaystyle}
\newcommand{\Fr}{\mathscr{F}\text{r}}
\newcommand{\Sub}{\mathscr{S}\text{ub}}

\newcommand{\Sq}{\mathscr{S}\text{q}}
\newcommand{\Sep}{\mathscr{S}\text{ep}}
\newcommand{\Path}{\mathscr{P}\text{ath}}
\newcommand{\Kan}{\mathscr{K}\text{an}}
\newcommand{\Diag}{\mathscr{D}\text{iag}}
\newcommand{\Uni}{\mathscr{U}\text{ni}}
\newcommand{\Los}{\L\text{o}\acute{s}}
\newcommand{\Adj}{\mathscr{A}\text{dj}}
\newcommand{\Tw}{\mathscr{T}\text{w}}
\newcommand{\Emb}{\mathscr{E}\text{mb}}

\newcommand{\comsq}[8]{
  \begin{tikzcd}[row sep=0.5in, column sep=0.5in]
    #1 \arrow[r, "#5"] \arrow[d, "#6"']
    \pgfmatrixnextcell #2 \arrow[d, "#7"] \\
    #3 \arrow[r, "#8"]
    \pgfmatrixnextcell #4
  \end{tikzcd}
}

\newcommand{\pbsq}[8]{
  \begin{tikzcd}[row sep=0.5in, column sep=0.5in]
    #1 \arrow[r, "#5"] \arrow[d, "#6"'] \arrow[dr, phantom, "\ulcorner", very near start]
    \pgfmatrixnextcell #2 \arrow[d, "#7"] \\
    #3 \arrow[r, "#8"']
    \pgfmatrixnextcell #4
  \end{tikzcd}
}

\newcommand{\liftsq}[8]{
  \begin{tikzcd}[row sep=0.5in, column sep=0.5in]
    #1 \arrow[r, "#5"] \arrow[d, "#6"']
    \pgfmatrixnextcell #2 \arrow[d, "#7"] \\
    #3 \arrow[r, "#8"] \arrow[ur, dashed]
    \pgfmatrixnextcell #4
  \end{tikzcd}
}

\newcommand{\simpset}[7]{
 \begin{tikzcd}[row sep=0.5in, column sep=0.5in]
   #1 \arrow[r, shorten >=1ex,shorten <=1ex]
   \pgfmatrixnextcell #2 
   \arrow[l, shift left=1.2, "#5"] \arrow[l, shift right=1.2, "#4"'] 
   \arrow[r, shift right, shorten >=1ex,shorten <=1ex ] \arrow[r, shift left, shorten >=1ex,shorten <=1ex] 
   \pgfmatrixnextcell #3 
   \arrow[l] \arrow[l, shift left=2, "#7"] \arrow[l, shift right=2, "#6 "'] 
   \arrow[r, shorten >=1ex,shorten <=1ex] \arrow[r, shift left=2, shorten >=1ex,shorten <=1ex] 
   \arrow[r, shift right=2, shorten >=1ex,shorten <=1ex]
   \pgfmatrixnextcell \cdots 
   \arrow[l, shift right=1] \arrow[l, shift left=1] \arrow[l, shift right=3] \arrow[l, shift left=3] 
 \end{tikzcd}
}

\newcommand{\adjun}[4]{
 \tikzcdset{row sep/normal=0.5in}
 \tikzcdset{column sep/normal=0.5in}
\begin{tikzcd}
 #1  \arrow[r, shift left=1.4, "#3"] \pgfmatrixnextcell
 #2 \arrow[l, shift left=1, "#4"] 
\end{tikzcd}
}

\newtheorem{theone}{Theorem}[section]
\newtheorem*{thetwo}{Theorem}
\newtheorem{lemone}[theone]{Lemma}
\newtheorem{propone}[theone]{Proposition}
\newtheorem{corone}[theone]{Corollary}
\newtheorem*{cortwo}{Corollary}

\theoremstyle{definition}
\newtheorem{defone}[theone]{Definition}
\newtheorem{exone}[theone]{Example}

\newtheorem{notone}[theone]{Notation}

\theoremstyle{remark}
\newtheorem{remone}[theone]{Remark}
\newtheorem{intone}[theone]{Intuition}
\newtheorem{queone}[theone]{Question}

\makeatletter
\def\@seccntformat#1{%
  \expandafter\ifx\csname c@#1\endcsname\c@section\else
  \csname the#1\endcsname\quad
  \fi}
\makeatother

\title{An Elementary Approach to Truncations}

\author{Nima Rasekh}

\date{January 2020}

\address{{\'E}cole Polytechnique F{\'e}d{\'e}rale de Lausanne, SV BMI UPHESS, Station 8, CH-1015 Lausanne, Switzerland}

\email{nima.rasekh@epfl.ch}

\begin{document}

\begin{abstract}

We study truncated objects using elementary methods. 
Concretely, we use universes and the resulting natural number object to define internal truncation levels
and prove they behave similar to standard truncated objects.
\par 
Moreover, we take an elementary approach to localizations, giving various equivalent conditions that characterize localizations 
and constructing a localization out of a sub-universe of local objects via an internal right Kan extension.
We then use this general approach, as well as an inductive approach, to construct truncation functors. 
\par 
We use the resulting truncation functors to prove classical results about truncations, such as Blakers-Massey theorem, 
in the elementary setting. We also give examples of non-presentable $(\infty,1)$-categories where the elementary approach can be used 
to define and compute truncations.
\par 
Finally, we turn around and use truncations to study elementary $(\infty,1)$-toposes and show how they can help us better understand 
subobject classifiers and universes.
\end{abstract}

\maketitle
\addtocontents{toc}{\protect\setcounter{tocdepth}{1}}

\tableofcontents

 \section{Introduction}\label{Introduction}
 
 \subsection{Motivation}\label{Motivation}
 One powerful tool in algebraic topology are truncations. They allow us to simplify the structure of a space 
 and help us focus on certain aspects of a space by giving us a filtration.
 For that reason, truncations have been generalized to many other settings and in particular 
 to presentable $(\infty,1)$-categories \cite{SY19} and Grothendieck $(\infty,1)$-toposes \cite{Re05}, \cite{Lu09}.
 However, there has not been an extensive effort to generalize these tools for categories that are not presentable 
 and/or don't have infinite colimits. 
 \par 
 The goal of this work is to take an important step in that direction. We take an {\it elementary} 
 approach to truncations: Define and study all aspects of truncations without ever resorting to any infinite construction 
 which would require us to develop a theory of sets.
 \par 
 We achieve this primarily by using {\it universes} (also known as object classifiers) and {\it natural number objects}.
 A natural number object allows us to use induction and in particular gives us an internal grading, whereas a universe 
 allows us to argue about collections of object internally, without using infinite colimits. 
 We call a category which has such a structure an {\it elementary $(\infty,1)$-topos} \cite{Ra18b} in analogy to elementary toposes, 
 which have been studied in logic for quite some time.
 \par 
 This paper is part of a larger project whose aim it is to expand the application and frontiers of homotopy theory
 and prove everything using elementary methods.
 We already give many applications of such elementary methods in this paper: 
 We use universes to construct a localization functor as an {\it internal right Kan extension} 
 (Subsection \ref{Subsec:Constructing Localizations via Universes}).
 We also present an elementary proof of the Blakers-Massey theorem (Section \ref{Sec Blakers Massey Theorem for Modalities}).
 Moreover, we construct and compute truncations in a non-presentable $(\infty,1)$-category which does not 
 have infinite colimits and observe how it has non-standard truncation levels that cannot be found in 
 presentable $(\infty,1)$-categories (Subsection \ref{Subsec:Non Standard Truncations}).
 
 \subsection{Main Results} \label{Main Results}
  We will break down the major results of this work into different groups:
  \begin{enumerate}
   \item Constructing Truncations 
   \item Results about Truncations
   \item Applications of the Existence of Truncations
  \end{enumerate}
 
  {\bf Constructing Truncations:}
  We give two different methods of constructing $(-1)$-trunctions, each using different initial assumptions.
  
  \begin{thetwo} (Theorem \ref{The Neg Trunc from Join})
   Let $A$ be any object and let $\{ inl: A^{\ast n} \to A^{\ast n+1} \}_{n:\mathbb{N}}$ be the sequential diagram described in 
   Remark \ref{Rem Join Truncation Map}. 
   Then the sequential colimit $A^{\ast \infty}$ is the $(-1)$-truncation of $A$.
  \end{thetwo}
 
 \begin{cortwo} (Corollary \ref{Cor:Neg One Trunc via SOC})
  Let $\E$ be an $(\infty,1)$-category satisfying the condition of Remark \ref{Rem:Intro Condition for Localization Section} 
  which has a subobject classifier $\Omega$. 
  Then the functor 
  $$\T_\Omega: \E \to \tau_{-1}\E$$ 
  is the $(-1)$-truncation functor.
 \end{cortwo}
 
 We can generalize this the previous construction to an $A$-localization.
 
 \begin{cortwo} (Corollary \ref{Cor:A Localization Construction})
  Let $\E$ be an $(\infty,1)$-category satisfying the condition of Remark \ref{Rem:Intro Condition for Localization Section}
  with locally Cartesian closed universe $\U$ that classifies $S$.
  Let $A$ be an object and $\U^{loc_A}$ be the sub-universe of $A$-local objects.
  Then $\T_{\U^{loc_A}}: \E^S \to \E^S$ induces a localization functor 
  $$\T_{\U^{loc_A}}: \E^S \to (\E^{loc_A})^S$$
 \end{cortwo}
  
 This specializes to $n$-truncations.
 
 \begin{cortwo} (Corollary \ref{Cor:Truncation Functor Construction})
  Let $\E$ be an $(\infty,1)$-category satisfying the condition of Remark \ref{Rem:Intro Condition for Localization Section}
  with locally Cartesian closed universe $\U$ that classifies $S$.
  Then $\T_{\U^{\leq n}}: \E^S \to \E^S$ induces a localization functor 
  $$\T_{\U^{\leq n }}: \E^S \to \tau_n\E^S$$
 \end{cortwo}
 
 However, we can also define $n$-truncations inductively:
 
 \begin{thetwo} (Theorem \ref{The:Constructing Trunctions Inductively})
   Let $n$ be a natural number. We can define the $(n+1)$-truncation functor 
   \begin{center}
    \adjun{\E}{\tau_{n+1}\E}{\tau_{n+1}}{i}
   \end{center}
   inductively by $n$-truncating the loop objects.
  \end{thetwo}
  
 {\bf Results about Truncations:}
 We can use our understanding of universes to give elementary proofs for results which had originally  been 
 proven using spaces before (\cite{Re05}, \cite{SY19}). 
 
 \begin{cortwo} (Corollary \ref{Cor:Elementary Equiv})
  Let $X$ an object and $n < m$. Then we have an equivalence of categories 
  $$ \tau_n (\eta_X)^*: \tau_n(\E_{/\tau_m X}) \to \tau_n(\E_{/X})$$
 \end{cortwo}
 
 Having a well-defined $n$-truncation functor we can use elementary methods to prove the Blakers-Massey theorem.
 
 \begin{cortwo} (Corollary \ref{Cor:BMT Truncations})
  (Classical Blakers-Massey Theorem)
  Let us assume we have a pushout square such that $f$ is $m$-connected and $g$ 
  is $n$-connected.
 \begin{center}
  \begin{tikzcd}[row sep=0.5in, column sep=0.5in]
    Z \arrow[d, "f"] \arrow[r, "g"] & Y \arrow[d,"k"] \\
    X \arrow[r, "h"] & W \arrow[ul, phantom, "\ulcorner", very near start] 
  \end{tikzcd}
 \end{center}
 Then, the gap map $(f,g): Z \to X \times_W Y$ is $(m+n)$-connected.
 \end{cortwo}
 
 We can also use it to study $n$-truncations in non-presentable $(\infty,1)$-categories such as filter quotients.
 
 \begin{cortwo} (Corollary \ref{Cor:Trunc Obj filter product})
  Let $\E$ be an elementary $(\infty,1)$-topos such that $1$ has two subobjects.
  Moreover, let $I$ be a set and $\Phi$ a filter on $P(I)$ and  
   fix a natural number $(a_i)_{i \in I}$ in $(\mathbb{N})_{i \in I}$.
  Then an object $(X_i)_{i \in I}$ is $(n_i)_{i \in I}$-truncated if and only if
  $$\{ i: X_i \text{ is } n_i-\text{truncated} \} \in \Phi$$
 \end{cortwo}
 
 {\bf Applications of the Existence of Truncations:}
 We will use $(-1)$-truncations in our study of elementary $(\infty,1)$-toposes.
 First, we use it to prove that we can construct free universes.
 
 \begin{thetwo} (Theorem \ref{The:Smallest Universe EHT})
  Let $\E$ satisfy the conditions of Remark \ref{Rem:E Conditions Free Universe Section}.
  Then the following equivalent:
  \begin{enumerate}
   \item $\E$ has sufficient universes.
   \item There exists a functor 
   $$\Fr_\U: \O_\E \xrightarrow{ \ \simeq \ } \E_{/ \Univ}$$
   that assigns to each morphism its free universe.
  \end{enumerate}
 \end{thetwo}
 
 Finally, we use $(-1)$-truncations to give an alternative description of subobject classifiers.
 
 \begin{thetwo} 
  (Theorem \ref{The:EHT SOC egal Prop Resize})
  Let $\E$ be $(\infty,1)$-category, which satisfies all conditions of Definition \ref{Def:EHT} except for the existence of 
  subobject classifiers. Then the following are equivalent:
  \begin{enumerate}
   \item $\E$ is an elementary $(\infty,1)$-topos.
   \item There exists a universe $\U$ that satisfies propositional resizing.
  \end{enumerate}
 \end{thetwo}

 \subsection{Outline}\label{Outline}
  As some section can be read quite independently from each other, we present a diagram of the dependencies of various sections, 
  as a guide for a reader who is only interested in a specific section or result.
  
  \begin{itemize}
   \item $\Rightarrow$: Strong dependency
   \item $\rightarrow$: weak dependency
  \end{itemize}
  
  
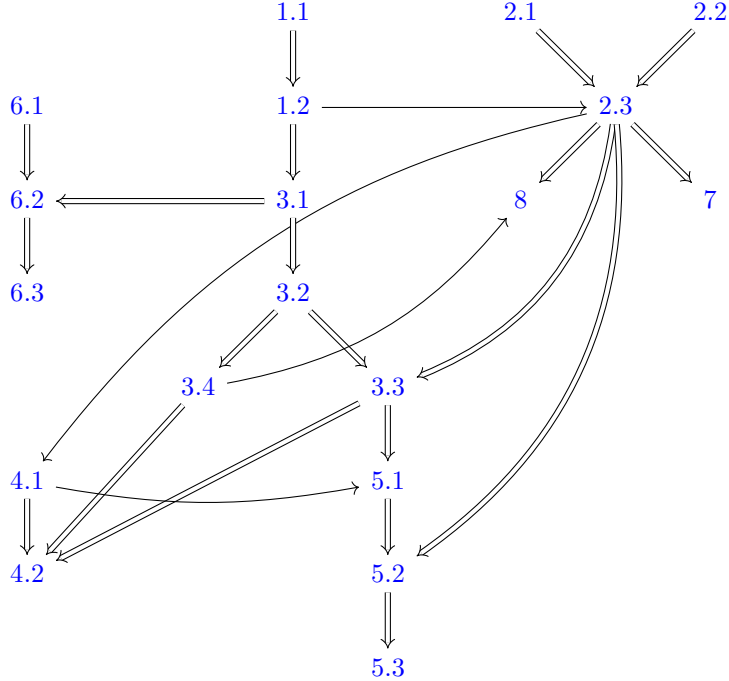
\begin{center}
   \begin{tikzcd}[row sep=0.3in, column sep=0.2in]
    & & & & \ref{Subsec:Truncated Objects} \arrow[d, Rightarrow] & & & \ref{Subsec:Def of Join} \arrow[dr, Rightarrow] & & \ref{Subsec:Sequential Colimits via NNO} \arrow[dl, Rightarrow] \\
    \ref{Subsec:Review Filter Hypercomp} \arrow[d, Rightarrow] &  & & & \ref{Subsec:Connected Objects} \arrow[d, Rightarrow] \arrow[rrrr] & & & & \ref{Subsec:Join and Trunc}  \arrow[dr, Rightarrow] \arrow[dl, Rightarrow] \arrow[dddlll, Rightarrow, bend left=30] \arrow[dddddlll, Rightarrow, bend left=30] \arrow[ddddllllllll, bend right=20]& \\
    \ref{Subsec:Los Truncation} \arrow[d, Rightarrow] & & & & \ref{Subsec:NNO and Truncations} \arrow[llll, Rightarrow] \arrow[d, Rightarrow] & & &  \ref{Sec:Propositional Resizing} & &  \ref{Sec:Free Universes}\\
    \ref{Subsec:Non Standard Truncations} & & & & \ref{Subsec:Prop of Trunc Functors} \arrow[dr, Rightarrow] \arrow[dl, Rightarrow] & & & & & \\
    & & & \ref{Subsec:Universe of Truncated Objects} \arrow[ddlll, Rightarrow] \arrow[uurrrr, bend right=20] & & \ref{Subsec:Neg One Conn} \arrow[d, Rightarrow] \arrow[ddlllll, Rightarrow] & & & &\\ 
    \ref{Subsec:Constructing Localizations via Universes} \arrow[d, Rightarrow] \arrow[rrrrr, bend right=10] & & & & & \ref{Subsec:Modalities} \arrow[d, Rightarrow] & & & & \\
    \ref{Subsec:Inductive Construction of Truncations} & & & & & \ref{Subsec:BMT for Modalities}\arrow[d, Rightarrow] & & & & \\
     & & & & & \ref{Subsec:BMT for Truncations} & & & &
   \end{tikzcd}
  \end{center}
  
  We now give an overview of the various sections:
  \par 
  {\bf Section \ref{Sec Truncated and Connected Objects} :}
  This section is a general introduction to the concept of truncated and connected objects in locally Cartesian closed $(\infty,1)$-categories.
  Subsection \ref{Subsec:Truncated Objects} focuses on truncated objects, whereas Subsection \ref{Subsec:Connected Objects} 
  focuses on connected objects. The goal is to see what kind of results we can prove about truncated objects with minimal assumptions on our 
  $(\infty,1)$-category. The work in this section is heavily motivated by the study of truncated objects in \cite{Re05}.
  \par 
  {\bf Section \ref{Sec:The Join Construction} :}
  The goal of this section is to give an explicit construction of $(-1)$-truncations via an internal sequential colimit of joins. 
  For that reason, we first review the join construction in Subsection \ref{Subsec:Def of Join}. 
  Then we review the notion of internal sequential colimits via natural number objects. That is the goal of Subsection 
  \ref{Subsec:Sequential Colimits via NNO}. Finally, we combine those two results to construct $(-1)$-truncations in 
  Subsection \ref{Subsec:Join and Trunc}.
  This section is motivated by the work of Egbert Rijke in homotopy type theory \cite{Ri17}.
  \par 
  {\bf Section \ref{Sec Truncation Functors}:}
  The goal of this section is to study a general internal truncation functor in an $(\infty,1)$-category. 
  As a first step we redefine truncated objects internally via a natural number object, which is what we do in 
  Subsection \ref{Subsec:NNO and Truncations}.
  In Subsection \ref{Subsec:Prop of Trunc Functors}, 
  we then assume the existence of a truncation functor and use it to further study truncated and connected objects.
  We then apply our new understanding of such truncations. First, we use it 
  in Subsection \ref{Subsec:Neg One Conn} to study $(-1)$-connected maps.
  Then we use it in Subsection \ref{Subsec:Universe of Truncated Objects} to study the subuniverse of $n$-truncated objects.
  \par 
  {\bf Section \ref{Sec:Constructing Truncations} :}
  In this section we prove the existence of truncations. In Subsection \ref{Subsec:Constructing Localizations via Universes}
  we take a very general approach and give several equivalent characterizations of a localization functor and, in particular, 
  give a general recipe for constructing localizations via an internal right Kan extension,
  which then also applies to truncations. In Subsection \ref{Subsec:Inductive Construction of Truncations} we give a more familiar 
  argument and construct truncations inductively as a lower-level truncations of the loop objects, 
  which is motivated by the work of Rijke \cite{Ri17}.
  \par 
  {\bf Section \ref{Sec Blakers Massey Theorem for Modalities} :}
  In this section we adapt the proof in \cite{ABFJ17} to show that every modality satisfies the Blakers-Massey theorem.
  We start in Subsection \ref{Subsec:Modalities} by reviewing basic properties, and in particular locality, of modalities.
  In Subsection \ref{Subsec:BMT for Modalities} we then confirm that the proof given in \cite{ABFJ17} still holds in our setting. 
  Finally, in Subsection \ref{Subsec:BMT for Truncations} we apply this result to the modality of trunctions.
  \par 
  {\bf Section \ref{Sec:Filter Quotients and Truncations} :}
  In this section we apply our results about truncations to a special class of elementary $(\infty,1)$-toposes: filter quotients.
  In Subsection \ref{Subsec:Review Filter Hypercomp} we review the filter quotient construction. 
  In Subsection \ref{Subsec:Los Truncation} we give a classification result for truncated objects in a filter quotient, 
  which is motivated by $\Los$'s theorem for ultraproducts. Finally, in Subsection \ref{Subsec:Non Standard Truncations} 
  we apply these results to one specific $(\infty,1)$-category and illustrate how the existence of non-standard natural numbers 
  leads to non-standard truncation levels, which are essential in its study.
  \par 
  {\bf Section \ref{Sec:Free Universes} :}
  This section is an application of the existence of $(-1)$-truncations in elementary $(\infty,1)$-topos theory.
  We apply the $(-1)$-truncations to prove that we can always functorially assign each map 
  to the ``smallest" universe in an elementary $(\infty,1)$-topos.
  \par 
  {\bf Section \ref{Sec:Propositional Resizing} :}
  This section is another application of $(-1)$-truncations to elementary $(\infty,1)$-toposes. 
  We prove that the existence of a subobject classifier is equivalent to a property on a universe in the topos.
  \par 
  {\bf Appendix \ref{Sec:HTT}:}
  In this appendix we give a quick overview of the main $(\infty,1)$-topos theoretic ideas we use throughout.
  For more details see \cite{Ra18b}.
  \par 
  {\bf Appendix \ref{Sec:Appendix NNO}:}
  In this appendix we give a quick review of natural number objects and their basic properties.
  For more details see \cite{Ra18c}.

 \subsection{Background} \label{Background}
 We work in the context of $(\infty,1)$-categories and so use the language of $(\infty,1)$-categories freely throughout.
 The preferred model of the author is complete Segal spaces \cite{Re01}, but all results here hold in any model of $(\infty,1)$-categories
 that has a theory of limits, adjunctions and Cartesian fibrations, such as quasi-categories, or any other $\infty$-cosmos \cite{RV16}. 
 The only exception is Section \ref{Sec:Filter Quotients and Truncations}, where we also mention Kan enriched categories.
 \par 
 We primarily use the theory of elementary $(\infty,1)$-toposes and their universes as introduced in \cite{Ra18b}.
 We have summarized some of the key definitions and results in Appendix \ref{Sec:HTT}.
 \par 
 Moreover, we make extensive use of natural number objects and the fact that they exist in an elementary $(\infty,1)$-topos \cite{Ra18c}.
 We have summarized the main definitions about natural number objects in Appendix \ref{Sec:Appendix NNO}.
 \par 
 Many statements and theorem in Sections \ref{Sec Truncated and Connected Objects} and \ref{Sec Truncation Functors}
 are inspired by \cite{Re05} and thus it is a helpful source for better understanding the material here.
 \par 
 Section \ref{Sec:The Join Construction} is directly motivated by the analogous result in \cite{Ri17}.
 The same holds for the theorem and part of the proof in Subsection \ref{Subsec:Inductive Construction of Truncations}.
 \par 
 Section \ref{Sec Blakers Massey Theorem for Modalities} is a direct adaptation of work in \cite{ABFJ17} 
 and can only be read in tandem with their excellent work.
 \par 
 Section \ref{Sec:Filter Quotients and Truncations} depends on work in \cite{Ra20}. In particular the example in 
 Subsection \ref{Subsec:Non Standard Truncations} has been studied there in far greater detail.
 \par 
 There are other sources that we use in a more limited way. 
 For example we use the hypercompletions (Section \ref{Sec:Filter Quotients and Truncations}) 
 as studied in \cite{Lu09} and propositional resizing (Section \ref{Sec:Propositional Resizing}) as introduced in \cite{UF13}.

 \subsection{Notation} \label{Notation}
  In order to make the results as broad as possible, we will specify the necessary assumptions on $\E$ at the beginning of each section. 
 We denote the final object by $1$ and initial object by $\emptyset$. 
 We denote the $(\infty,1)$-category of spaces by $\s$, but the Kan enriched category of spaces by $\Kan$. 
 In order to avoid any confusion, we denote the final object in $\s$ by $*$.
 For more details on topos theoretic notation see Appendix \ref{Sec:HTT}.
  
 \subsection{Acknowledgements} \label{Subsec Acknowledgements}
  This work benefitted from many interesting conversations. 
  I want thank Charles Rezk for helping me better understand truncations and in particular 
  giving me a new perspective on his result \cite[Lemma 8.6]{Re05}, which resulted in 
  an elementary proof for Theorem \ref{The:U Leq n is n plus trunc}.
  I also want to thank Georg Biedermann for helpful conversations about their work in \cite{ABFJ17}.  
  I want to thank Lior Yanovski for pointing me to the work by him and Tomer Schlank on truncations \cite{SY19}.
  I also would like to thank Jonas Frey for making me aware of a cool method to construct $(-1)$-truncations in type theory,
  which was my motivation for the work in Subsection \ref{Subsec:Constructing Localizations via Universes}.
  I also want to thank Raffael Stenzel for telling me about the connection between $(-1)$-truncations and universes,
  which led to the construction of free universes in Section \ref{Sec:Free Universes}.
  I also want to thank Asaf Horev for giving me insight into my own result and explaining to me the connection between my ideas 
  and right Kan extensions, which I have tried to capture in Intuition \ref{Int:Internal Right Kan Extension}.
  I want to also thank Peter Lumsdaine for suggesting a ``$\Los$'s theorem for truncations" (Theorem \ref{The:Trunc Obj filter quotient}).
  Finally, I would like to thank the Max-Planck-Institut f{\"u}r Mathematik for its hospitality and financial support.
 
 \section{Truncated and Connected Objects} \label{Sec Truncated and Connected Objects}
 This section is a gentle introduction to the concept of truncated and connected objects in a $(\infty,1)$-category.
 We want to study the properties of truncated objects that can be proven with very few assumptions on our category.
 \par
 In Subsection \ref{Subsec:Truncated Objects} we discuss truncated objects and in Subsection \ref{Subsec:Connected Objects} 
 we discuss their counter-parts, namely connected objects.
 
 \begin{remone} \label{Rem:Trunc and Conn Obj Cat Conditions}
  Throughout this section $\E$ has following properties:
  \begin{enumerate}
   \item Finite limits and colimits
   \item Locally Cartesian closed (Definition \ref{Def:LCCC})
  \end{enumerate}
 This implies that 
 \begin{enumerate}
  \item[(3)] Colimits are universal (Definition \ref{Def:Colimits Universal},
   Lemma \ref{Lemma:LCCC Colimits Universal}).
 \end{enumerate}
 \end{remone}
 
 \begin{notone}
  As we make many analogies to spaces in this section we denote the $n$-sphere in $\s$ as $S^n$ and the $n$-sphere in $\E$ 
  as $S^n_\E$ (Definition \ref{Def Sphere in E}).
 \end{notone}

 \subsection{Truncated Objects} \label{Subsec:Truncated Objects}
 In this subsection we define and study truncated objects. 
 Let us first review truncated spaces.
 
 \begin{defone}
  Let $S^{-1} = \emptyset$, the empty space, and $S^n$ ($n \geq 0$) be defined inductively by the pushout square
  \begin{center}
   \begin{tikzcd}[row sep=0.5in, column sep=0.5in]
    S^{n-1} \arrow[d] \arrow[r] & * \arrow[d]  \\
    * \arrow[r] & S^n \arrow[ul, phantom, "\ulcorner", very near start]
   \end{tikzcd}
  \end{center}
 \end{defone}

 \begin{defone}
  Let $n \geq -2$. A space $X$ is $n$-{\it truncated} if 
  $$Map(*,X) \to Map(S^{n+1},X)$$
  is an equivalence of spaces.
 \end{defone}
 
 \begin{exone}
  A space is $(-2)$-truncated if it is contractible. Moreover, it is $(-1)$-truncated if it is empty or contractible.
 \end{exone}
 
 There is a relative version of this definition.
 
 \begin{defone}
  A map of spaces $f: Y \to X$ is $n$-{\it truncated} if every fiber is $n$-truncated.
 \end{defone}

 The goal is to generalize this concept to an arbitrary $(\infty,1)$-category.
 
 \begin{defone}
  An object $X$ in $\E$ is $n$-truncated if for every object $Y$ the mapping object 
  $Map_\E(Y,X)$ is an $n$-truncated space.
 \end{defone}
 
 \begin{exone}
  An object $X$ is $(-2)$-truncated if $Map(Y,X)$ is contractible. 
  This is means $X$ is equivalent to the final object $1$.
 \end{exone}
 
 The definition has a relative version.
 
 \begin{defone}
  A map $f: Y \to X$ is $n$-truncated if it is $n$-truncated as an object in $\E_{/X}$.
 \end{defone}
 
 We have following alternative characterization of $n$-truncated maps.
 
 \begin{lemone} \label{Lemma Yoneda Truncation Cond}
  A map $f: Y \to X$ is $n$-truncated if for every object $Z$ the map of spaces 
  $f_*: Map(Z, Y) \to Map(Z, X)$ is $n$-truncated.
 \end{lemone}

 \begin{exone}
  A map $f$ is $(-2)$-truncated if and only if it is an equivalence in $\E$.
 \end{exone}

 The two definitions are related by the following lemma.
 
 \begin{lemone}
  An object $X$ is $n$-truncated if and only if $X \to 1$ is $n$-truncated.
 \end{lemone}

 Notice this definition also implies following important result.
 
 \begin{lemone} \label{Lemma Trunc Base Change}
  $n$-truncated maps are closed under base change.
 \end{lemone}
 
 \begin{remone} 
  This does not hold for cobase change. In the pushout diagram of spaces
  \begin{center}
   \begin{tikzcd}[row sep=0.5in, column sep=0.5in]
     S^1 \arrow[r] \arrow[d] & 1 \arrow[d] \\
     1 \arrow[r] & S^2 \arrow[ul, phantom, "\ulcorner", very near start]
   \end{tikzcd}
  \end{center}
  the map $S^1 \to 1$ is $1$-truncated. But the map $1 \to S^2$ is clearly not $1$-truncated 
  as the homotopy fiber over each point in $S^2$ is $\Omega S^2$, which is not $1$-truncated \cite{Gr69}.
 \end{remone}
 
 Our next goal is to give an alternative characterization of $n$-truncated maps.

 \begin{defone} \label{Def Sphere in E}
  Let $S^{-1}_\E = \emptyset_\E$, the initial object, and $S^n_\E$ ($n \geq 0$) be defined inductively by the pushout square
  \begin{center}
   \begin{tikzcd}[row sep=0.5in, column sep=0.5in]
    S^{n-1}_\E \arrow[d] \arrow[r] & 1 \arrow[d]  \\
    1 \arrow[r] & S^n_\E \arrow[ul, phantom, "\ulcorner", very near start]
   \end{tikzcd}
  \end{center}
  We call these objects the {\it $n$-spheres in $\E$}. Notice they come with a well-defined basepoint $1 \to S^n_\E$.
 \end{defone}

 \begin{lemone} \label{Lemma:Ext Sphere vs Int Spheres}
  Let $Z$ and $X$ be two objects in $\E$ and $n \geq -1$. Then we have an equivalence 
  $$Map_\E(Z \times S^n_\E, X) \simeq Map_{\s}(S^n, Map_\E(Z,X))$$
 \end{lemone}
 
 \begin{proof}
  We will prove it by induction. If $n = -1$ then the result is clear. Let us assume it holds for $n$.
  As colimits in $\E$ are universal the object $Z \times S^{n+1}$ is the pushout of the following diagram
  \begin{center}
   \begin{tikzcd}[row sep=0.5in, column sep=0.5in]
    Z \times S^{n}_\E \arrow[d] \arrow[r] & Z \arrow[d]  \\
    Z \arrow[r] & Z \times S^{n+1}_\E \arrow[ul, phantom, "\ulcorner", very near start]
   \end{tikzcd}
  \end{center}
  Taking mapping spaces gives us following pullback square.
  \begin{center}
   \pbsq{Map_\E(Z \times S^{n+1}_\E, X)}{Map_\E(Z,X)}{Map_\E(Z,X)}{Map_\E(Z \times S^n_\E, X)}{}{}{}{}
  \end{center}
  Using our induction assumption we can rewrite this pullback square as follows.
  \begin{center}
   \pbsq{Map_\E(Z \times S^{n+1}_\E, X)}{Map_{\s}(*,Map_\E(Z,X))}{Map_{\s}(*,Map_\E(Z,X))}{Map_{\s}(S^n, Map_\E(Z, X))}{}{}{}{}
  \end{center}
  This implies that 
  $$Map_\E(Z \times S^{n+1}_\E, X) \xrightarrow{ \ \ \simeq \ \ } Map_{\s}( \Sigma S^n ,  Map_\E(Z, X)) 
  \xrightarrow{ \ \ \simeq \ \ } Map_{\s}(S^{n+1}, Map_\E(Z, X))$$
  This finishes our induction step.
 \end{proof}

 \begin{lemone} \label{Lemma Ntrun with Sn condition}
  An object $X$ in $\E$ is $n$-truncated if and only if for every object $Z$ the induced map 
  $$ Map_\E( Z \times S^{n+1}_\E , X) \to Map_\E(Z , X)$$
  is an equivalence of spaces.
 \end{lemone}
 
 \begin{proof}
  By the previous lemma we have a diagram where the vertical map is always an equivalence.
  \begin{center}
  \begin{tikzcd}[row sep=0.3in, column sep=0.5in]
   Map_{\s}(*,Map_\E( Z \times S^{n+1}_\E , X)) \arrow[dd, "\simeq"] \arrow[dr] & \\
   & Map_{\s}(*,Map_\E(Z , X)) \\
   Map_{\s}(S^{n+1}, Map_\E(Z, X)) \arrow[ur] &
  \end{tikzcd}
  \end{center}
  So, the top map is an equivalence if and only if the bottom map is an equivalence.
  However, the bottom map is an equivalence if and only if $Map_\E(Z,X)$ is $n$-truncated.
 \end{proof}

 \begin{corone}
  An object $X$ is $n$-truncated if and only if the map $X^{S^{n+1}_\E} \to X$ is an equivalence in $\E$.
 \end{corone}

 \begin{exone}
  $X$ is $(-1)$-truncated if and only if the diagonal map $\Delta: X \to X^{S^0_\E} \simeq X \times X$ is an equivalence.
 \end{exone}

 This example generalizes in the appropriate way.
 
 \begin{propone} \label{Prop Diag Trunc}
  $X$ is $n$-truncated if and only if $\Delta: X \to X \times X$ is $(n-1)$-truncated.
 \end{propone}
 
 \begin{proof}
  Let $Z$ be an arbitrary object. $\Delta: X \to X \times X$ is $(n-1)$-truncated if and only if the map 
  $$Map_{X \times X}(Z, X) \to Map_{X \times X}(Z \times S^n_\E,X)$$
  is an equivalence. 
  This is equivalent to 
  $$Map(Z, X) \times_{Map(Z,X \times X)} * \to Map(Z \times S^n_\E,X) \times_{Map(Z \times S^n_\E,X \times X)} *$$
  which is equivalent to the diagram 
  \begin{center}
   \pbsq{Map(Z,X)}{Map(Z,X \times X)}{Map(Z \times S^n_\E, X)}{Map(Z \times S^n_\E, X \times X)}{}{}{}{}
  \end{center}
  being a pullback square.
  By adjoining the right hand column this is equivalent to 
  \begin{center}
   \pbsq{Map(Z,X)}{Map(Z \coprod Z,X )}{Map(Z \times S^n_\E, X)}{Map(Z \times S^n_\E \coprod Z \times S^n_\E, X)}{}{}{}{}
  \end{center}
  being a pullback square.
  Using the fact that $Map(-,X)$ takes colimits to limits, this is equivalent to 
  $$Map(Z,X) \to Map(Z \times S^n_\E \coprod_{Z \times S^n_\E \coprod Z \times S^n_\E} (Z \coprod Z), X)$$
  being an equivalence. 
  Using the fact that colimits are universal, this is equivalent to 
  $$Map(Z,X) \to Map(Z \times (S^n_\E \coprod_{S^0_\E \times S^n_\E} S^0_\E), X)$$
  being an equivalence. However, $S^n_\E \coprod_{S^0_\E \times S^n_\E} S^0_\E \simeq S^{n+1}_\E$ and so this is equivalent to 
  $$Map(Z,X) \to Map(Z \times S^{n+1}_\E, X)$$
  being an equivalence, which is the definition of $X$ being $n$-truncated.
 \end{proof}

 These results have relative versions which can be proven analogously.
 
 \begin{lemone} \label{Lemma Ntrun with Sn condition Relative}
  A map $Y \to X$ is $n$-truncated if and only if for every map $ p:Z \to X$ the induced map 
  $$ Map_{/X}( Z \times S_\E^{n+1}, Y) \to Map_{/X}(Z, Y)$$
  is an equivalence of spaces. Here we take the map  $p \pi_1 : Z \times S^{n+1}_\E \to X$. 
 \end{lemone}

 \begin{corone}
  A map $Y \to X$ is $n$-truncated if and only if the map  
   $$[Y \to X]^{[X \times S^{n+1}_\E \to X]} \to [Y \to X]$$
   is an equivalence in $\E_{/X}$. Here the left hand side is the 
   internal mapping object in $\E_{/X}$.
 \end{corone}

 \begin{lemone}
  A map $A \to B$ is $n$-truncated if and only if the map $\Delta: A \to A \times_B A$ is $(n-1)$-truncated.
  In particular, a map $A \to B$ is $(-1)$-truncated if and only if $A \to A \times_B A$ is an equivalence.
 \end{lemone}

 One implication of these results is the comparison between $(-1)$-truncated maps and monos.
 
 \begin{propone} \label{Prop Conn Eq Mono}
  A map $f: A \to B$ is mono if and only if it is $(-1)$-truncated.
 \end{propone}
 
 \begin{proof}
  Assume $f$ is $(-1)$-truncated and fix an object $C$. Then we have following pullback diagram 
  \begin{center}
   \pbsq{Map(C,A)}{Map(C,A)}{Map(C,A)}{Map(C,B)}{}{}{}{}
  \end{center}
  Let us now fix a map $f: C \to B$ and the denote the fiber over $g$ by $Fib_g$. Then the pullback gives us an equivalence 
  $$Fib_g \xrightarrow{ \ \ \simeq \ \ } Fib_g \times Fib_g$$
  This implies that $Fib_g$ is $(-1)$-truncated which exactly implies that $f$ is mono.
  \par 
  On the other hand, if $f$ is mono, then for any map $g:C \to B$ the space $Fib_g$ is either empty or contractible 
  which implies that that the square above is a pullback square.
 \end{proof}

 \subsection{Connected Objects} \label{Subsec:Connected Objects}
 In this subsection we define and study connected objects in $\E$ (Remark \ref{Rem:Trunc and Conn Obj Cat Conditions}).
 The goal is to show how they complement truncated objects.
 First, we again review the concept of connected spaces.
 
 \begin{defone}
  A space $X$ is $n$-{\it connected} if its $n$-truncation is contractible.
  This means that $\pi_k(X)$ is trivial for $k \leq n$.
 \end{defone}

 This definition also has a relative version
 
 \begin{defone}
  A map of spaces $Y \to X$ is $n$-{\it connected} if the homotopy fiber over each point $x \in X$ is 
  $n$-connected.
 \end{defone}
 
 We cannot generalize this definition to an arbitrary locally Cartesian category as there is no way to deduce 
 the existence of a truncation functor. 
 However, in Section \ref{Sec Truncation Functors} we will focus elementary $(\infty,1)$-toposes, 
 where we can construct truncation functors and in particular prove the analogue to the definition above  
 (Proposition \ref{Prop Conn iff Tau Final Object}).
 Thus, we will now give another formulation of connected maps of spaces.
 
 \begin{lemone}
  A map of spaces $Y \to X$ is $n$-connected if and only if for every $n$-truncated map $Z \to X$
  the induced map 
  $$Map_{/X}(X,Z) \to Map_{/X}(Y,Z)$$
  is an equivalence of spaces.
 \end{lemone}
 
 This equivalent definition can be easily generalized to a definition for connected maps 
 in $\E$.
 
 \begin{defone} \label{Def NConn Map}
  A map $Y \to X$ in  $\E$ is $n$-{\it connected} if for every $n$-truncated map $ Z \to X$
  the induced map 
    $$Map_{/X}(X,Z) \to Map_{/X}(Y,Z)$$
  is an equivalence of spaces.
 \end{defone}
 
 The definition has following special case.
 
 \begin{defone}
  An object $Y$ is $n$-connected if for every $n$-truncated object $Z$ the 
  induced map 
   $$Map(1 ,Z) \to Map(Y,Z)$$
  is an equivalence of spaces.
 \end{defone}
 
 \begin{exone}
  Lemma \ref{Lemma Ntrun with Sn condition} tells us that the object $S^{n+1}_\E$ is $n$-connected.
 \end{exone}
 
 In particular we can generalize Lemma \ref{Lemma Ntrun with Sn condition Relative} using $n$-connected maps.
 
 \begin{lemone}
  A map $Y \to X$ is $n$-truncated if for every map $Z \to X$ and $n$-connected object $W$ the induced map
  $$Map_{/X}(Z \times W, Y) \to Map_{/X}(Z,Y)$$
  is an equivalence.
 \end{lemone}
 
 The condition of being truncated gets progressively more relaxed, in other words being $(-2)$-truncated is 
 the strictest possible truncation conditions. On the other being connected get progressively more restrictive. 
 Thus, we can ask ourselves what happens in the infinite case.
 
 \begin{defone} \label{Def:Infinite Connected Map}
  We say a map $f$ is $\infty$-{\it connected} if it is $n$-connected for all $n$. 
 \end{defone}

 Notice every equivalence is trivially $n$-connected, but the opposite does not hold. 
 For an example for $\infty$-connected maps that are not equivalences see \cite[11.3]{Re05}. 
 
 We showed that $n$-truncated maps are preserved by base change. Similarly we can show that $n$-connected maps 
 are preserved by cobase change.

 \begin{propone} \label{Prop n Conn coBase Change}
  Let $X \to Y$ be $n$-connected. Then for any map $X \to A$ the induced map 
  $A \to Y \coprod_X A$ is also $n$-connected.
 \end{propone}
 
 \begin{proof}
  Let $Z \to A \coprod_X Y$ be an $n$-truncated map. A pushout diagram in $\E$ is still a pushout diagram in the over category. 
  This gives us following equivalence 
  $$Map_{/A \coprod_X Y}(A \coprod_X Y, Z) \simeq Map_{/A \coprod_X Y}(A,Z) \underset{Map_{/A \coprod_X Y}(X,Z)}{\times} Map_{/A \coprod_X Y}(Y,Z)$$
  Now we have following commutative square.
  \begin{center}
   \begin{tikzcd}[row sep=0.5in, column sep=0.5in]
    Map_{/A \coprod_X Y}(Y,Z) \arrow[r, "\simeq"] \arrow[d] & Map_Y(Y, Z \underset{A \coprod_X Y}{\times} Y) \arrow[d, "\simeq"] \\  
    Map_{/A \coprod_X Y}(X,Z) \arrow[r, "\simeq"] & Map_Y(X,Z \underset{A \coprod_X Y}{\times} Y )
   \end{tikzcd}
  \end{center}
  The horizontal maps are equivalences because of the adjunction. The right vertical map is an equivalence because $X \to Y $ is $n$-connected and 
  $Z \times_{A \coprod_X Y} Y$ is $n$-truncated, by Lemma \ref{Lemma Trunc Base Change}. This implies that the left hand map is also an equivalence.
  Finally, we have following pullback square.
  \begin{center}
   \pbsq{Map_{/A \coprod_X Y}(A,Z) \underset{Map_{/A \coprod_X Y}(X,Z)}{\times} Map_{/A \coprod_X Y}(Y,Z)}{Map_{/A \coprod_X Y}(Y,Z)}{
   Map_{/A \coprod_X Y}(A,Z)}{Map_{/A \coprod_X Y}(X,Z)}{}{}{\simeq}{}
  \end{center}
  The fact that the right hand vertical map is an equivalence implies that the left hand vertical map is also an equivalence, which is exactly 
  the result we wanted.
 \end{proof}

 \begin{remone}
  In fact $n$-connected maps are also preserved by base change. We will prove it in Corollary \ref{Cor N Conn Base Change},
  but before that we need to develop more theory and define truncation functors.
 \end{remone}
 
 Connected maps also behave well with respect to composition.
 
 \begin{propone} \label{Prop Conn fg}
  Let $f: Y \to X$ and $g: Z \to Y$.
  \begin{enumerate}
   \item If $f$ and $g$ are $n$-connected then $fg$ is $n$-connected.
   \item If $g$ and $fg$ are $n$-connected then $f$ is $n$-connected.
  \end{enumerate}
 \end{propone}
 
 \begin{proof}
  Let $W$ be a $n$-truncated object over $X$. We have following diagram
  \begin{center}
   \begin{tikzcd}[row sep=0.3in, column sep=0.1in]
    & & Map_{/X}(X,W) \arrow[dl, "f^*"] \arrow[dr, "(fg)^*"] & & \\
    & Map_{/X}(Y,W) \arrow[dl, "\simeq"] & & Map_{/X}(Z,W) \arrow[dr, "\simeq"] & \\
    Map_{/Y}(Y, W \underset{X}{\times} Y) \arrow[rrrr, "g^*", "\simeq"']  & & & & Map_{/Y}(Z, W \underset{X}{\times} Y) 
   \end{tikzcd}
  \end{center}
   If we assume that $f$ and $g$ are $n$-connected then $f^*$ and $g^*$ are equivalences which implies that
   $(fg)^*$ is an equivalence which means that $fg$ is $n$-connected.
   On the other hand, if we assume that $g$ and $fg$ are $n$-connected then $g^*$ and $(fg)^*$ are equivalences 
   which implies that $f^*$ is an equivalence which means that $f$ is $n$-connected.
 \end{proof}

 Truncated and connected objects do not intersect in a non-trivial way.
 
 \begin{lemone} \label{Lemma:N Trunc Conn Contr}
  Let $Y$ be an object that is $n$-truncated and $n$-connected. Then $Y$ is the final object.
 \end{lemone}
 
 \begin{proof}
  As $Y$ is $n$-connected we know that we have an equivalence 
  $$Map(1,Z) \simeq Map(Y,Z)$$
  for every $n$-truncated object $Z$. But $Y$ itself is $n$-truncated, which means we have an equivalence.
  $$Map(1,Y) \simeq Map(Y,Y)$$
  This means that every map $Y \to Y$ will factor through the final map $Y \to 1$. If we apply this to the identity map we get
  a retract 
  $$Y \to 1 \to Y$$
  which proves that $Y$ is equivalent to the final object.
 \end{proof}

 This result also has an interesting relative version.
 
 \begin{corone} \label{Cor:N Trunc Conn Equiv}
  Let $f: Y \to X$ be a map that is $n$-truncated and $n$-connected. Then $f$ is an equivalence.
 \end{corone}
 
 \begin{remone}
  For the case $n=-1$ this generalizes the classical result for elementary $1$-toposes, where every map that is 
  epi and mono is an isomorphism. That is known as being a {\it balanced category}.
 \end{remone}

 Before we finish this section, we state an alternative characterization of $n$-connected maps. 
 The original proof is stated with presentable categories, however, it it is not used in the proof 
 and so we will refrain from the repeating the proof and refer the reader to the right source.
 
 \begin{propone} \label{Prop Trunc Conn Lift}
  \cite[Proposition 8.10]{Re05}
  Let $f: X \to Y$ be a map. The following are equivalent.
  \begin{enumerate}
   \item For all $n$-truncated maps $g: W \to Z$ the diagram 
   \begin{center}
    \pbsq{Map(Y,W)}{Map(Y,Z)}{Map(X,W)}{Map(X,Z)}{g_*}{f^*}{f^*}{g_*}
   \end{center}
   is a homotopy pullback diagram of spaces. 
   \item $f$ is $n$-connected.
  \end{enumerate}
  \end{propone}
 
  \begin{remone}
   The proposition is basically saying that if $g$ is $n$-truncated
   then $f$ is $n$-connected if and only if the space of lifts of the square below is contractible.
   \begin{center}
    \liftsq{X}{W}{Y}{Z}{}{f}{g}{}
   \end{center}
  \end{remone}

 We can use this lemma to prove some moderate stability for connectivity that we will use later on 
 (see Proposition \ref{Prop Pb of epimono is epimono}).
 
 \begin{lemone} \label{Lemma Product by A is conn}
  Let $f:X \to Y$ be $n$-connected and $A$ an object. Then the map $A \times X \to A \times Y$ is also $n$-connected.
 \end{lemone}

 \begin{proof}
  Let $g: W \to Z$ be $n$-truncated. We need to prove that the square 
  \begin{center}
    \pbsq{Map(A \times Y,W)}{Map(A \times Y,Z)}{Map(A \times X,W)}{Map(A \times X,Z)}{g_*}{f^*}{f^*}{g_*}
   \end{center}
   is a pullback square, which by adjunction is equivalent to 
   \begin{center}
    \pbsq{Map(Y,W^A)}{Map(Y,Z^A)}{Map(X,W^A)}{Map(X,Z^A)}{(g^A)_*}{f^*}{f^*}{(g^A)_*}
   \end{center}
   for which it suffices to prove that $g^A: W^A \to Z^A$ is $n$-truncated.
   \par 
   By Lemma \ref{Lemma Ntrun with Sn condition Relative} it suffices to prove that for any map $B \to Z^A$, the map 
   $$(g^A)_*: map_{Z^A}(B \times S^{n+1}_\E,W^A) \to map_{Z^A}(B,W^A)$$
   is an equivalence.
   Using adjunction this is equivalent to 
   $$g_*: map_Z(B \times A \times S^{n+1}_\E, W) \to map_Z(B \times A, W)$$
   being an equivalence, which immediately follows from the fact that $g$ is $n$-truncated.
 \end{proof}

 In order to further study truncated and connected maps we need to be able to truncate objects internally.
 In a presentable $(\infty,1)$-category we can construct truncated objects via the existence of infinite 
 colimits and the small object argument.
 However, we are not in the presentable setting and thus cannot use such techniques. 
 \par 
 Fortunately, there is a way to get around that. For that we need to make a technical digression and study the 
 join construction.

  \section{The Join Construction} \label{Sec:The Join Construction}
  In order to further our study of truncations, we need to be able to construct truncated objects via a 
  truncation functor. Usually truncations are defined via the small object argument, but we cannot use such arguments 
  because the categories in question are not presentable.
  \par 
  For that reason, we have to always give explicit constructions of the trunction functors. 
  In this section we will take a first important step by giving an explicit construction of the $(-1)$-truncation functor. 
  \par 
  There is a concrete way to construct $(-1)$-truncations in a Grothendieck $(\infty,1)$-topos which 
  is motivated by the $\check{C}$ech nerve.
  For a given object $X$ we form the simplicial object 
  \begin{center}
   \simpset{X}{X \times X}{X \times X \times X}{\pi_1}{\pi_2}{}{}
  \end{center}
  and take the colimit of the simplicial diagram \cite[Proposition 7.8]{Re05} \cite[Proposition 6.2.3.4]{Lu09}.
  However, this is an infinite colimit which does not exist in our setting.
  \par 
  On the other hand, we can prove that we have natural number objects in our setting (see Appendix \ref{Sec:Appendix NNO} for more details).
  Combining natural number objects with the existence of universes, we can define and compute internal sequential colimits \cite[Subsection 4.4]{Ra18c}. 
  The goal is thus to find a way to replace the simplicial diagram 
  by a sequential diagram. This has been done successfully by Egbert Rijke in the context of homotopy type theory where he replaces 
  a simplicial diagram of products with a sequential diagram of joins \cite{Ri17}.
  \par 
  Thus, our goal is to adapt his argument to our setting of $(\infty,1)$-categories. For that reason, we first have to study the join. 
  This will be done in Subsection \ref{Subsec:Def of Join}. We will then review sequential diagrams and internal sequential colimits in 
  an elementary $(\infty,1)$-topos, which we do in Subsection \ref{Subsec:Sequential Colimits via NNO}. 
  Finally, we use a sequential colimit of joins to construct the $(-1)$-truncation, 
  which is the goal of Subsection \ref{Subsec:Join and Trunc}.
  
  \begin{remone}
   In this section $\E$ is an $(\infty,1)$-category that satisfies following conditions:
   \begin{enumerate}
    \item It has finite limits and colimits.
    \item It is locally Cartesian closed (Definition \ref{Def:LCCC}).
    \item It has a subobject classifier (Definition \ref{Def:SOC}).
    \item It has sufficient universes $\U$ (Definition \ref{Def:Sufficient Universes})
    that are closed under finite limits and colimits (Definition \ref{Def:Closed Universes}).
   \end{enumerate}
    those conditions also imply
   \begin{enumerate}
    \item[(5)] Colimits are universal (Definition \ref{Def:Colimits Universal}, Lemma \ref{Lemma:LCCC Colimits Universal})
    \item[(6)] It has a natural number object (Theorem \ref{The:EHT has NNO}).
   \end{enumerate}
  \end{remone}
 
  \begin{remone}
   The material in this section is independent from the previous section, and we will only make use of the definition of a 
   $(-1)$-truncated object, namely an object $X$ such that the diagonal 
   $$\Delta_X: X \to X \times X$$
   is an equivalence.
  \end{remone}

  \subsection{The Definition of Join} \label{Subsec:Def of Join}
   In this subsection we define the join of two objects and give some basic properties.
   
   \begin{defone} \label{Def Join}
     Let $A$ and $B$ be two objects in $\E$. We define the {\it join} $A \ast B$ in $\E$ by the following diagram.
   \begin{center}
    \begin{tikzcd}[row sep=0.5in, column sep=0.5in]
     A \times B \arrow[d] \arrow[r] & B \arrow[d] \\ 
     A \arrow[r] &  \ds A \ast B \arrow[ul, phantom, "\ulcorner", very near start]
    \end{tikzcd}
   \end{center}
   \end{defone}
 
  \begin{notone}
   Notice this definition directly generalizes to  $\E_{/X}$ where we can define the join 
   of two maps $f \ast g$ for two maps $f: A \to X$ and $g: B \to X$.
  \end{notone}

  We now want to establish some basic results about the join construction that will be necessary later on.
  
  \begin{theone}
   Let $p: Y \to X$ be an object in $\E_{/X}$ and assume we have two maps $f: A \to X$ and $g: B \to X$. Then we have an equivalence
    $$p^*f \ast p^*g \simeq p^*(f \ast g) .$$
  \end{theone}
  
  \begin{proof}
   If we pullback the diagram above by the map $p: Y \to X$ we get
   \begin{center}
    \begin{tikzcd}[row sep=0.5in, column sep=0.5in]
     p^*A \underset{Y}{\times} p^*B \arrow[d] \arrow[r] & p^*B \arrow[d] \arrow[rdd, bend left = 20, "p^*g"] & \\
     p^*A \arrow[r] \arrow[rrd, bend right = 20, "p^*f"'] &  \ds p^*A \underset{Y}{\ast} p^*B \arrow[dr, dashed, "p^*(f \ast g)"'] \arrow[ul, phantom, "\ulcorner", very near start] & \\
     & & Y
    \end{tikzcd}
   \end{center}
   As colimits are universal, taking pullback preserves pushout diagrams which means that 
   $p^*(f \ast g)$ is the join of $p^*f$ and $p^*g$ giving us the desired result.
  \end{proof}

  \begin{lemone}
   The join of two $(-1)$-truncated objects is $(-1)$-truncated.
  \end{lemone}
  
  \begin{proof}
   We know that $A \times B$ is also $(-1)$-truncated,
   and so the result follows from the fact that subobjects of $1$ are closed under pushouts.
  \end{proof}
  
  In fact we have a certain inverse.
  
  \begin{lemone} \label{Lemma Idempotent Join}
   An object $A$ is $(-1)$-truncated if and only if the map $A \to A \ast A$ is an equivalence.
  \end{lemone}
  
  \begin{proof}
   If $A$ is $(-1)$-truncated then $A \times A \simeq A$ and so the result follows. 
   On the other hand, if $A \ast A \simeq A$ then we have pushout diagram 
   \begin{center}
    \begin{tikzcd}[row sep=0.5in, column sep=0.5in]
     A \times A \arrow[d, "\pi_1"] \arrow[r, "\pi_2"] & A \arrow[d, "id_A"] \\
     A \arrow[r, "id_A"] & A \arrow[ul, phantom, "\ulcorner", very near start]
    \end{tikzcd}
   \end{center}
   First, this implies that $\pi_1 \simeq \pi_2$. Now for any object $B$ we thus have following commutative diagram
   (which is not necessarily a pushout diagram):
   \begin{center}
    \begin{tikzcd}[row sep=0.5in, column sep=0.5in]
      Map(B,A) \times Map(B,A) \arrow[d, "\pi_1"] \arrow[r, "\pi_2"] & Map(B,A) \arrow[d, "id"] \\
      Map(B,A) \arrow[r, "id"] & Map(B,A)
    \end{tikzcd}
   \end{center}
   If $Map(B,A)$ is empty then there is nothing to prove. However, if not then let $f \in Map(B,A)$ be any map.
   The fact that $\pi_1 \simeq \pi_2$ implies that the following is a commutative diagram 
   \begin{center}
    \begin{tikzcd}[row sep=0.5in, column sep=0.5in]
     Map(B,A) \arrow[dr, "id \times f"] \arrow[drr, "\{ f \}", bend left=20] \arrow[ddr, "id"', bend right=20] & &  \\
     &  Map(B,A) \times Map(B,A) \arrow[d, "\pi_1"] \arrow[r, "\pi_2"] & Map(B,A) \arrow[d, "id"] \\
      & Map(B,A) \arrow[r, "id"] & Map(B,A) 
    \end{tikzcd}
   \end{center}
   where $\{ f \}: Map(B,A) \to Map(B,A)$ is the map that takes everything to $f \in Map(B,A)$. The fact that this diagram commutes implies that the 
   identity is homotopic to the constant map which means $Map(B,A)$ is contractible.
  \end{proof}

  Finally, notice that colimits are universal with respect to joins 
  (the same way they are universal with respect to pullbacks).
  
  \begin{lemone} \label{Lemma Join Descent}
   Let $f:B \to C$ be any map. Then $Coeq(f \ast id_A ,id) \simeq Coeq(f,id) \ast A$. 
  \end{lemone}

  \begin{proof}
   As colimits are universal, we have $Coeq(f \times id_A ,id) \simeq Coeq(f,id) \times A$. 
   The result then follows by taking the pushout.
  \end{proof}

  \subsection{Internal Sequential Colimits via Natural Number Objects} \label{Subsec:Sequential Colimits via NNO}
  Before we proceed to our main goal of constructing $(-1)$-truncations, 
  we need a better understanding of internal sequential colimits. The material here is a review  
  and the reader can find far more details (and proofs) in \cite[Subsection 4.4]{Ra18c}.
  
   \begin{defone} \label{Def Natural Number Object}
   A {\it natural number object} is an object $\mathbb{N}$ in $\E$ along with two maps 
   $$1 \xrightarrow{ \ \ o \ \ } \mathbb{N} \xrightarrow{  \ \ s \ \ } \mathbb{N}$$
   such that $(\mathbb{N} , o , s)$ is initial.
  \end{defone}
  
  For more details see Definition \ref{Def:NNO}. We will now use natural number objects to define 
  sequences, sequential diagrams and sequential colimits.
  
  \begin{defone} \label{Def:Sequence of Objects}
   A {\it sequence of objects} $\{ A_n \}_{n: \mathbb{N}}$ is a map $\{ A_n \}_{n: \mathbb{N}} : \mathbb{N} \to \U$.
  \end{defone}

  It is usually not possible to define such a map directly. Rather the best way is to make use of the initiality condition,
  which we illustrate in the next example.
  
  \begin{exone}
   Let $X$ be an object in $\E$. Using the fact that we have coproducts we get a functor 
   $- \coprod X: \E \to \E$, and as $\U$ is closed under colimits we thus get a map $- \coprod X: \U \to \U$.
   A map $1 \to \U$ corresponds to a choice of object, which can be $X$ itself. 
   As we have chosen a triple $(\U, 1 \to \U, - \coprod X: \U \to \U)$, by the initiality we get a unique map 
   $$\mathbb{N} \to \U$$
   which takes $n$ to the $n$-fold coproduct $\coprod_n X$.
  \end{exone}
  
  Given that we usually build sequences via endomorphism $p: \U \to \U$, we adapt following notation convention.
  
  \begin{notone}
   Let $X: 1 \to \U$ be an object, $p: \U \to \U$, then we depict the resulting sequence $\mathbb{N} \to \U$ as 
   $$X, p(X), p^2(X), ... $$
  \end{notone}

   The next goal is to use natural number objects to construct sequential colimits. For that we need internal coproducts.
  
  \begin{defone} \label{Def:Internal Coproduct}
  Let $A_n$ be a sequence of objects. We define the {\it internal coproduct}, $\ds \sum_{n: \mathbb{N}} A_n$ as the pullback
  \begin{center}
   \pbsq{\ds \sum_{n: \mathbb{N}} A_n }{\U_*}{\mathbb{N}}{\U}{}{p_A}{}{}
  \end{center}
 \end{defone}
 
 \begin{defone}
  A {\it sequential diagram} $\{ f_n: A_n \to A_{n+1} \}_{n: \mathbb{N}}$ is a sequence of objects 
  $\{ A_n \}_{n: \mathbb{N}}: \mathbb{N} \to \U$ as well as a choice of map 
  \begin{center}
   \begin{tikzcd}[row sep=0.5in, column sep=0.5in]
    \ds \sum_{n: \mathbb{N}} A_n \arrow[rr, "\{ f_n \}_{n:\mathbb{N}}"] \arrow[dr] & & \ds \sum_{n: \mathbb{N}} A_{n+1} \arrow[dl] \\ 
     & \mathbb{N} & 
   \end{tikzcd}
  \end{center}
 \end{defone}

 \begin{notone}
  For any $n:\mathbb{N}$ we get a map $f_n: A_n \to A_{n+1}$.
  Thus, we will use following notation for a sequential diagram
  $$A_0 \xrightarrow{ \ \ f_0 \ \ } A_1 \xrightarrow{ \ \ f_1 \ \ } A_2 \xrightarrow{ \ \ f_2 \ \ } \cdots \ .$$
 \end{notone}
 
  \begin{remone}
  It is often quite difficult to directly construct such a sequential diagram. 
  Thus, the goal is to use the universal property of natural numbers again. Let $\U_1$, the morphism classifier of $\U$ 
  (see Theorem \ref{The:CSU Exist} for a detailed description). Then a map $1 \to \U_1$ is uniquely determined by a morphism 
  in $\E$ and a map $\U_1 \to \U_1$ can be determined by a functor on the arrow category $\Arr(\E) \to \Arr(\E)$.
 \end{remone}

 \begin{exone} \label{Ex:Sequential Diagram Coproducts}
  Let $[\emptyset \to X]: 1 \to \U_1$ be the map that chooses the morphism $\emptyset \to X$ and $[- \coprod id_X]: \U_1 \to \U_1$.
  be the map that takes a morphism $f: A \to B$ to $f \coprod id_X: A \coprod X \to B \coprod X$. 
  Then we get a sequential diagram of the form 
  $$X \xrightarrow{ \ \iota_0 \ } X \coprod X \xrightarrow{ \ \iota_0 \ } X \coprod (X \coprod X) ...$$
 \end{exone}
 
 \begin{defone}
  Let $\{ f_n \}_{n: \mathbb{N}}$ be a sequential diagram of the sequence of objects $A$. Then the {\it sequential colimit} of $f$ 
  is the coequalizer 
  \begin{center}
    \begin{tikzcd}[row sep=0.5in, column sep=0.5in]
    \ds \sum_{n: \mathbb{N}} A_n \arrow[r, shift left=0.05in, "f"] 
      \arrow[r, shift right=0.05in, "id_{ \sum_{n:\mathbb{N}} A_n }"'] & 
    \ds \sum_{n: \mathbb{N}} A_n \arrow[r]  & 
    A_{\infty}
   \end{tikzcd}
  \end{center}
 \end{defone}
  
  We will use these techniques in the next subsection to construct a sequential diagram of joins, whose 
  sequential colimit is the $(-1)$-truncation.
  
  \subsection{Join Sequence and Truncations} \label{Subsec:Join and Trunc}
  Having established some basic facts about joins and sequential colimits, we will now use them to construct $(-1)$-truncations.
  
  \begin{defone} \label{Def:Neg One Truncation}
   The $(-1)$-{\it truncation} is an adjunction 
   \begin{center}
    \adjun{\E}{\tau_{-1}\E}{\tau_{-1}}{i}
   \end{center}
   where $\tau_{-1}\E$ is the full subcategory of $(-1)$-truncated objects.
  \end{defone}

  \begin{remone}
   Concretely, the $(-1)$-truncation of an object $X$ is a map $i: X \to Y$
   such that following two conditions hold:
   \begin{enumerate}
    \item $Y$ is $(-1)$-truncated.
    \item For any other $(-1)$-truncated object $Z$ the map $Map(Y,Z) \to Map(X,Z)$ is an equivalence.
   \end{enumerate}
  \end{remone}

  It is valuable to see the relative definition which gives us the {\it image}.
  
   \begin{defone} \label{Def:Image}
    Let $f: Y \to X$ be a map in $\E$. The {\it image} is a diagram of the following form
    \begin{center}
     \begin{tikzcd}[row sep=0.5in, column sep=0.5in]
      Y \arrow[dr, "f"] \arrow[r, "q_f"] & Im(f) \arrow[d, "i_f", hookrightarrow] \arrow[r, dashed] & Z \arrow[dl, "i", hookrightarrow] \\
      & X & 
     \end{tikzcd}
    \end{center}
    where $i_f$ is $(-1)$-truncated and with the following universal property. For any $(-1)$-truncated map $i: Z \to X$, the induced map 
    $$(q_f)^*: Map_{/X}(Im(f), Z) \to Map_{/X}(Y, Z)$$
    is an equivalence.
   \end{defone}
  The existence of $(-1)$-truncations in $\E$ will imply that every map has an image, as it will simply be the $(-1)$-truncation of an 
  object $Y \to X$ in $\E_{/X}$. 
  \par 
  The goal is to show that we can build a sequential diagram of joins such that the 
  sequential colimit is the $(-1)$-truncation. 
  A review of the basics of internal sequential colimits can be found in Subsection \ref{Subsec:Sequential Colimits via NNO}.
  
 We want to construct a sequential diagram using joins, analogous to Example \ref{Ex:Sequential Diagram Coproducts}. 
 
 \begin{remone} \label{Rem Join Truncation Map}
  Let $\U_1$ be the morphism classifier. 
  Then the data of the map $[\emptyset \to A]: 1 \to \U_1$ and $[- \ast id_A]: \U_1 \to \U_1$ 
  gives us the commutative diagram
  \begin{center}
     \begin{tikzcd}[row sep=0.5in, column sep=0.5in]
       \ds\sum_{n:\mathbb{N}} A^{\ast n} \arrow[dr] \arrow[rr, "inl"] & & \ds\sum_{n:\mathbb{N}} A^{\ast n+1} \arrow[dl] \\
       & \mathbb{N} &
     \end{tikzcd}
    \end{center}
 \end{remone}
 
 \begin{notone}
  Following our previous notation convention, we will depict this sequential diagram as 
  $$ A \to A \ast A \to (A \ast A) \ast A \to ... $$
 \end{notone}
  
  The goal is to prove that the sequential colimit of this sequence is the actual $(-1)$-truncation of $A$, which is the content of 
  Theorem \ref{The Neg Trunc from Join}.
  However, before that we first need to prove several lemmas. 
  
  \begin{lemone}
   Let $A$ and $A'$ be two objects and $i: A \to A'$ be a map such that for every $(-1)$-truncated object $B $ the induced map 
   $$ i^*: Map(A',B) \to Map(A,B) $$
   is an equivalence.
   Then the induced map 
   $$\iota_0^* : Map(A \ast A', B) \to Map(A,B)$$
   is an equivalence as well.
  \end{lemone}

  \begin{proof}
   As $B$ is $(-1)$-truncated, the space $Map(A,B)$ is either empty or contractible. If it is empty the result is obvious. If it is 
   contractible then we have to prove that $Map(A \ast A', B)$ is non-empty as well. As we know it is $(-1)$-truncated this implies that 
   $Map(A \ast A', B)$ contractible which will give us the desired result.
   \par 
   Let $j: A \to B$ be a map. Then by assumption that $i^*$ is an equivalence we thus have a map $h: A' \to B$ such that $hi \simeq j$. 
   This in particular implies that $j \pi_1 \simeq k \pi_2$ which induces a map 
   $j \ast h: A \ast A' \to B$.
  \end{proof}
  
  \begin{remone}
    The statement of the previous lemma and the proof is an adaptation of \cite[Lemma 3.1]{Ri17}.
  \end{remone}
  
  \begin{lemone}
   Let $(A_n,f_n)$ be a sequential diagram. Moreover, let $i_n: A \to A_n$ be a map between the sequences,
   where we take $A$ to be the constant sequence.
   
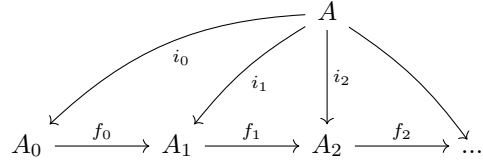
\begin{center}
    \begin{tikzcd}[row sep=0.5in, column sep=0.5in]
      & & A \arrow[dll, "i_0", bend right=20] \arrow[dl, "i_1", bend right=10] \arrow[d, "i_2"] \arrow[dr, bend left=10] & \\
     A_0 \arrow[r, "f_0"] & A_1 \arrow[r, "f_1"] & A_2 \arrow[r, "f_2"] & ...
    \end{tikzcd}
   \end{center}
  Moreover, for every $(-1)$-truncated object $B$ we have an equivalence 
  $$i_n^*: Map(A_n,B) \to Map(A,B)$$
  Then we have an equivalence $i_\infty^*: Map(A_\infty,B) \to Map(A,B)$.
  \end{lemone}
 
 \begin{proof}
  The space $Map(A,B)$ is $(-1)$-truncated. If it is empty the result follows trivially. Let us assume $Map(A,B)$
  is contractible. Then by the assumption $Map(A_n,B)$ is also contractible, which means there exist maps
  $h_n: A_n \to B$. This gives us a diagram where we take $B$ to be the constant sequence.
   
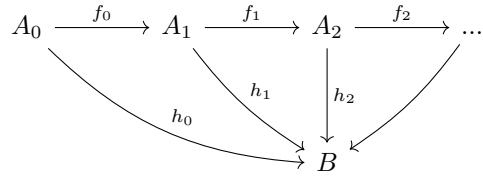
\begin{center}
    \begin{tikzcd}[row sep=0.5in, column sep=0.5in]
    
     A_0 \arrow[r, "f_0"] \arrow[drr, "h_0", bend right=20]  & A_1 \arrow[r, "f_1"]  \arrow[dr, "h_1", bend right=10] & 
     A_2 \arrow[r, "f_2"] \arrow[d, "h_2"] & ... \arrow[dl, bend left=10] \\
       & & B   &
    \end{tikzcd}
   \end{center}
  This map induces a map $h_\infty : A_\infty \to B$, which implies $Map(A_\infty, B)$ is non-empty giving us the desired result.
 \end{proof}
 
 \begin{remone}
  This statement and proof is an adaptation of \cite[Lemma 3.2]{Ri17} from homotopy type theory.
 \end{remone}
  
  Finally, we also need to understand how the join operation affects a sequential colimit.
  
  \begin{lemone} \label{Lemma Join Sequence with constant}
   Let $(A_n, f_n)$ be a sequence. Let $B$ be another object, which gives us a sequence $(A_n \ast B, f_n \ast id_B)$. 
   Then we have an equivalence $A_\infty \ast B \simeq (A_n \ast B)_\infty$.
  \end{lemone}
  
  \begin{proof}
   This follows from the fact that the join operation commutes with coequalizers, by Lemma \ref{Lemma Join Descent}.
  \end{proof}
  
  \begin{theone} \label{The Neg Trunc from Join}
   Let $A$ be any object and let $\{ inl: A^{\ast n} \to A^{\ast n+1} \}_{n:\mathbb{N}}$ be the sequential diagram described in Remark \ref{Rem Join Truncation Map}. 
   Then the sequential colimit $A^{\ast \infty}$ is the $(-1)$-truncation of $A$.
  \end{theone}

  \begin{proof}
   The two lemmas above already imply that the colimit satisfies the universal property.
   Thus, all that is left is to prove that $A^{\ast \infty}$ is $(-1)$-truncated.
   According to Lemma \ref{Lemma Idempotent Join} all we need is that the map $A^{\ast \infty} \to A^{\ast \infty} \ast A^{\ast \infty}$ 
   is an equivalence. 
   For that we need several steps.
   \par
   First, take the sequence
   $$ A \ast A \to (A \ast A) \ast A \to ... .$$
   We can think of this sequence in two different ways. On the one side it is $id_A \ast inl$. Thus, by Lemma \ref{Lemma Join Sequence with constant}
   the sequential colimit of this sequence is $A \ast A^{\ast \infty}$. 
   On the other side it is the sequence $inl(s)$. Thus, by \cite[Theorem 4.30]{Ra18c} it has sequential colimit $A^{\ast \infty}$. This implies that 
   the map $A \to A \ast A^{\ast \infty}$ is an equivalence. 
   Using a similar argument we deduce that $A \to A^{\ast n} \ast A^{\ast \infty}$ is an equivalence. 
   \par 
   Now we have following maps of sequences 
   \begin{center}
    \begin{tikzcd}[row sep=0.5in, column sep=0.5in]
     A \arrow[d, "\simeq"] \arrow[r] & A \ast A \arrow[d, "\simeq"] \arrow[r] & A \ast A \ast A \arrow[d, "\simeq"] \arrow[r] & ... \arrow[r] & A^{\ast \infty} \arrow[d, "\simeq"] \\
     A \ast A^{\ast \infty} \arrow[r] & A \ast A \ast A^{\ast \infty} \arrow[r] & A \ast A \ast A \ast A^{\ast \infty} \arrow[r] & ...  \arrow[r] & A^{\ast \infty} \ast A^{\ast \infty}
    \end{tikzcd}
   \end{center}
   The first row has sequential colimit $A^{\ast \infty}$, whereas by Lemma  \ref{Lemma Join Descent} the bottom row has sequential colimit $A^{\ast \infty} \ast A^{\ast \infty}$.
   However, as the sequences are equivalent it follows that $A^{\ast \infty} \to A^{\ast \infty} \ast A^{\ast \infty}$ is an equivalence, 
   which proves that $A^{\ast \infty}$ is $(-1)$-truncated.
  \end{proof}
  
  Notice the construction is a colimit construction and thus functorial, which means we have following theorem.
  
  \begin{notone}
   Let $\tau_{-1} \E$ be the subcategory of $(-1)$-truncated objects in $\E$.
  \end{notone}

  \begin{theone} \label{The Neg One Adj}
  There is an adjunction 
  \begin{center}
   \adjun{\E}{\tau_{-1} \E}{\tau_{-1}}{i}
  \end{center}
  Where $i: \tau_{-1}\E \to \E$ is the inclusion map. We call $\tau_{-1}$ the truncation functor.
 \end{theone}
 
 \begin{proof}
  The functoriality of the colimit construction implies that sending $A$ to $A^{\ast \infty}$ is functorial. 
  We have already proven that $A^{\ast \infty}$ is $(-1)$-truncated and that there is an adjunction.
 \end{proof}
  
 \begin{notone}
  Henceforth we will denote the $(-1)$-truncation of $A$ by $\tau_{-1}(A)$ and the sequential colimit map 
  $ inl^{\circ \infty}: A \to A^{\ast \infty} = \tau_{-1}(A)$ by $\eta_A: A \to \tau_{-1}(A)$ and we notice that $\eta_A$ is
  the unit map of the adjunction.
 \end{notone}
 
 There is also a relative version, which we will state for the sake of notation.
  
 \begin{notone}
  For a map $Y \to X$ we denote the $(-1)$-truncation by $\tau^X_{-1}(Y) \to X$. 
 \end{notone}
 
 \begin{remone} \label{Rem:Cover gives Neg One Conn}
  We can think of the truncation $\tau_{-1}A$ as the biggest subobject of $1$ that is ``covered"
  by $A$. Thus, in particular, if there exists a map $1 \to A$ then $\tau_{-1}A = 1$.
  For example, let $\Omega$ be the subobject classifier and $\U$ be a universe. 
  Then we have $\tau_{-1} \Omega = \tau_{-1} \U = 1$.
 \end{remone}
 
 The opposite to the remark above is not true.
 
 \begin{exone}
  Let $D$ be the diagram $1 \leftarrow 0 \rightarrow 2$. Then $Fun(D^{op}, \s)$ is a presheaf topos. 
  An object is a diagram of spaces $X_1 \rightarrow X_0 \leftarrow X_2$. Now let $P$ be the object 
  $$ \{ 0 \} \rightarrow \{ 0 , 1 \} \leftarrow \{ 1 \} $$
  where the maps are given by the evident inclusions.
  Clearly, the map to the final object is surjective and so $\tau_{-1}(P)$ is the final object. 
  However, there is no map from the final object as any such map would give us a diagram 
  \begin{center}
   \begin{tikzcd}[row sep=0.5in, column sep=0.5in]
     \{ * \} \arrow[r] \arrow[d] & \{ * \} \arrow[d] & \{ * \} \arrow[l] \arrow[d] \\
    \{ 0 \} \arrow[r, hookrightarrow] & \{ 0 \comma 1 \} \arrow[r, hookleftarrow] & \{ 1 \}
   \end{tikzcd}
  \end{center}
  which is impossible, as there is no map $\{* \} \to \{0,1\}$ that makes both squares commute at the same time.
 \end{exone}

\section{An Elementary Approach to Truncations} \label{Sec Truncation Functors}
 In Section \ref{Sec Truncated and Connected Objects} we used the notion of truncated spaces to define 
 truncated objects {\it externally}: An object $X$ is truncated if the mapping space $Map(-,X)$ is truncated.
 In this section we want to give an {\it internal} characterization of truncation that makes 
 use of natural number objects. 
 \par 
 If the natural number object is standard (Definition \ref{Def:Standard NNO}) then the internal and external notions 
 coincide. However, if not, then we can have internal truncation levels that have no external manifestation
 (Subsection \ref{Subsec:Non Standard Truncations}).
 \par 
 The goal of this section is to define and study internal truncation levels.
 We first define internal truncation levels, via internal spheres, and then we observe that many results of 
 Section \ref{Sec Truncated and Connected Objects} still hold in this internal setting. This is the goal of Subsection 
 \ref{Subsec:NNO and Truncations}. Then, in Subsection \ref{Subsec:Prop of Trunc Functors}, we will assume 
 there is a truncation functor 
 \begin{center}
  \adjun{\E}{\tau_n\E}{\tau_n}{i}
 \end{center}
 and study many interesting properties. 
 \par 
 In the next two subsections we look at some implications of the existence of truncations.
 First, in Subsection \ref{Subsec:Neg One Conn} we study $(-1)$-connected maps and prove that weak equivalences are local
 (Proposition \ref{Prop Equiv Local}).
 Then, in Subsection \ref{Subsec:Universe of Truncated Objects} we study the sub-universe of truncated objects
 and in particular prove it is itself (internally) $(n+1)$-truncated (Theorem \ref{The:U Leq n is n plus trunc}).
 
 \begin{remone}
  This section relies fundamentally on Section \ref{Sec Truncated and Connected Objects}, however, has only 
  minimal reliance on Section \ref{Sec:The Join Construction}. 
  Concretely, we only need it for Subsection \ref{Subsec:Neg One Conn} 
  and in particular Proposition \ref{Prop Equiv Local} and Lemma \ref{Lemma Constant Loop Space Constant}.
 \end{remone}

  \begin{remone}
   In this section $\E$ is an $(\infty,1)$-category that satisfies following conditions:
   \begin{enumerate}
    \item It has finite limits and colimits.
    \item It is locally Cartesian closed (Definition \ref{Def:LCCC}).
    \item It has a subobject classifier (Definition \ref{Def:SOC}).
    \item It has sufficient universes $\U$ (Definition \ref{Def:Sufficient Universes})
    that are closed under finite limits and colimits (Definition \ref{Def:Closed Universes}).
   \end{enumerate}
    those conditions also imply
   \begin{enumerate}
    \item[(5)] It has a natural number object (Theorem \ref{The:EHT has NNO}).
   \end{enumerate}
  \end{remone}
  
 \begin{notone} \label{Not:Truncations Neg Two}
  Notice that usually the minimal object in $\mathbb{N}$ is denoted by $0$, but for historical reasons truncation levels usually start 
  from $-2$. In order to reconcile this discrepancy, we will denote the minimal object in $\mathbb{N}$ by $-2$ as well.
  \par 
  We will also need the subobject of $\mathbb{N}$ starting from $-1$, which we denote by $s\mathbb{N}$ (as it is just the image of $s$)
  and note that $(s\mathbb{N},so,s)$ is also a natural number object with the same universal property.
 \end{notone}

 \subsection{Natural Number Objects and Truncations} \label{Subsec:NNO and Truncations}
 In Section \ref{Sec Truncated and Connected Objects} we defined an {\it external} notion of truncated objects. In order words, 
 we started with a notion of truncation for spaces and then generalized it to an $(\infty,1)$-category using the fact that it is 
 enriched over spaces. Thus, we are just making use of the degrees that come from the (shifted) natural numbers
 $\{-2,-1,0,...\}$ in spaces. 
 However, if an $(\infty,1)$-category has its own natural number object then it can have non-standard natural numbers, 
 which give us new truncation levels. 
 \par 
 The goal of this subsection is to define truncated objects in an $(\infty,1)$-category with a natural number object.
 The key change is that we now need to use natural number objects to define everything, but the resulting 
 universal properties are the same. 
 Thus, we adjust the definitions accordingly, but state the relevant results without proofs, instead referring the reader 
 to the relevant proof in Section \ref{Sec Truncated and Connected Objects}.
 
 \begin{defone} \label{Def:Sigma for U}
  Let $\sum: \U \to \U$ be the map of universes induced by the functor $\sum: \E \to \E$, which takes an object $X$ 
  to $1 \coprod_X 1$.
 \end{defone}

 \begin{defone} \label{Def:Sn NNO}
  Let $S^n: s\mathbb{N} \to \U$ be the unique map given by the triple $(\U, \emptyset: 1 \to \U, \sum: \U \to \U)$.
 \end{defone}
 
 \begin{exone}
  Let $n$ be a standard natural number object, then $S^n$ is just the standard $n$-sphere. In particular, $S^{-1} = \emptyset$ 
  and $S^0= 1 \coprod 1$.
 \end{exone}

 \begin{remone}
  By definition of $S^n$ there is always a map $1 \to S^n$.
 \end{remone}

 \begin{defone} \label{Def:Truncation NNO}
  Let $n$ be a natural number. An object $X$ is $n$-{\it truncated} if the map $X^{S^{n+1}} \to X$ is an equivalence.
 \end{defone}
 
 \begin{exone}
  If $n$ is a standard natural number object then this definition agrees with the previous one. In particular 
  a $(-2)$-truncated object is equivalent to the final object and a $(-1)$-truncted object is a subobject of the final object.
 \end{exone}

 Using the definition of $X^{S^{n+1}}$ 
 we immediately recover some important results from Section \ref{Sec Truncated and Connected Objects}.
 
 \begin{lemone}[Lemma \ref{Lemma Ntrun with Sn condition Relative}] \label{Lemma Ntrun with Sn condition Relative NNO}
  $X$ is $n$-truncated if and only if for every object $Z$ the map 
  $$map(Z \times S^{n+1},X) \to map(Z,X)$$
  is an equivalence of spaces.
 \end{lemone}

 \begin{lemone}[Lemma \ref{Lemma Trunc Base Change}] \label{Lemma Trunc Base Change NNO}
  $n$-truncated maps are closed under base change.
 \end{lemone}

  \begin{propone}[Proposition \ref{Prop Diag Trunc}] \label{Prop Diag Trunc NNO}
  $X$ is $n$-truncated if and only if $\Delta: X \to X \times X$ is $(n-1)$-truncated.
 \end{propone}
 
 \begin{remone}
  Notice, a result such as Lemma \ref{Lemma:Ext Sphere vs Int Spheres} does not hold anymore in this setting.
  This might seem surprising given that we proved this lemma using induction and we can use induction with natural number objects 
  (Theorem \ref{The:Peano NNO}).
  \par 
  The key observation is that the equivalence
  $$Map_\E(Z \times S^n_\E, X) \simeq Map_{\s}(S^n, Map_\E(Z,X))$$
  can only be defined {\it externally}, but cannot be expressed {\it internally}. In other words, from the perspective of $\E$ 
  the notion of an external sphere doesn't even make any sense!
  \par 
  Thus we can only define and prove results using internal spheres, such as Lemma \ref{Lemma Ntrun with Sn condition Relative NNO}.
  For a more detailed discussion on natural number objects and internal language see \cite{Jo03} and in particular 
  \cite[Lemma D5.2.1]{Jo03}.
 \end{remone}

 We can take a similar approach to $n$-connected objects. 

 \begin{defone} \label{Def:Connected Maps NNO}
  Let $n$ be a natural number. 
  A map $Y \to X$ is $n$-{\it connected} if for every $n$-truncated map $Z \to X$
  the induced map 
  $$Map_{/X}(X,Z) \to Map_{/X}(Y,Z)$$
  is an equivalence of spaces.
 \end{defone}

 \begin{defone} \label{Def:Infinite Connected Map NNO}
  A map $f$ is $\infty$-{\it connected} if it is $n$-connected for every natural number $n$.
 \end{defone}

 Again, we get the same results as in Section \ref{Sec Truncated and Connected Objects}.
 
  \begin{propone}[Proposition \ref{Prop n Conn coBase Change}] \label{Prop n Conn coBase Change NNO}
  Let $X \to Y$ be $n$-connected. Then for any map $X \to A$ the induced map 
  $A \to Y \coprod_X A$ is also $n$-connected.
 \end{propone}
 
 \begin{propone}[Proposition \ref{Prop Conn fg}] \label{Prop Conn fg NNO}
  Let $f: Y \to X$ and $g: Z \to Y$.
  \begin{enumerate}
   \item If $f$ and $g$ are $n$-connected then $fg$ is $n$-connected.
   \item If $g$ and $fg$ are $n$-connected then $f$ is $n$-connected.
  \end{enumerate}
 \end{propone}

 \begin{corone}[Corollary \ref{Cor:N Trunc Conn Equiv}] \label{Cor:N Trunc Conn Equiv NNO}
  Let $f: Y \to X$ be a map that is $n$-truncated and $n$-connected. Then $f$ is an equivalence.
 \end{corone}
 
  \begin{propone}[Proposition \ref{Prop Trunc Conn Lift}] \label{Prop Trunc Conn Lift NNO}
  Let $f: X \to Y$ be a map. The following are equivalent.
  \begin{enumerate}
   \item For all $n$-truncated maps $g: W \to Z$ the diagram 
   \begin{center}
    \pbsq{Map(Y,W)}{Map(Y,Z)}{Map(X,W)}{Map(X,Z)}{g_*}{f^*}{f^*}{g_*}
   \end{center}
   is a homotopy pullback diagram of spaces. 
   \item $f$ is $n$-connected.
  \end{enumerate}
  \end{propone}
 
  \begin{lemone}[Lemma \ref{Lemma Product by A is conn}] \label{Lemma Product by A is conn NNO}
  Let $f:X \to Y$ be $n$-connected and $A$ an object. Then the map $A \times X \to A \times Y$ is also $n$-connected.
 \end{lemone}
  
  \begin{remone}
   Henceforth,  whenever we need an elementary observations about truncated and connected objects,
    we will refer to the statements in this subsection rather than Section \ref{Sec Truncated and Connected Objects},
    in order to ensure they hold for all natural numbers in a natural number object and not just the standard ones.
  \end{remone}

\subsection{Properties of Trunction Functors} \label{Subsec:Prop of Trunc Functors}
 For this subsection we fix a natural number $n$ (which does not have to be standard).
 The goal of this subsection is to assume that we have an adjunction

 \begin{center}
  \adjun{\E}{\tau_n \E}{\tau_n}{i}
 \end{center} 
 where $\tau_n \E$ is the subcategory of $n$-truncated objects.
 Our goal is to use the existence of $\tau_n$ to prove many interesting results about truncated and connected objects.
 We will use some to these observations in the coming sections.
 In particular, we will use it in Subsection \ref{Subsec:Inductive Construction of Truncations} to give an inductive 
 construction of truncation functors.
 
 \begin{remone}
  The adjunction gives us an equivalence 
  $$Map(\tau_n X, \tau_n X) \xrightarrow{ \ \ \simeq \ \ } Map(X, \tau_n X)$$
  This gives us a canonical map $\eta_X: X \to \tau_n X$ (the image of the identity), which is the unit of the adjunction.
 \end{remone}

 It is worth pointing out the one case that is trivial.
  
 \begin{remone} \label{Rem Neg Two Trunc}
  The only $(-2)$-truncated object in $\E$ is the final object and so the $(-2)$-truncation of every object is the final object.
  Thus, $\tau_{-2}$ trivially exists.
 \end{remone}
 
 \begin{remone}
  Whenever we have an $(\infty,1)$-category and $n$ is a standard natural number then we can define an 
  $(n+1,1)$-category by externally truncating the mapping spaces.
  However, this external truncation does not generally coincide with an internal truncation.
  For an example notice that the following spaces are not equivalent
  $$\tau_0(Map(S^1,S^1)) \not\simeq Map(S^1,\tau_0(S^1)).$$
  Here we focus exclusively on the internal truncations as discussed before and not the external truncation.
 \end{remone}

 For the remainder of this subsection we study truncated and connected maps with the additional strength of having 
 a truncation functor. Concretely, the most important applications that we are going to prove in this Subsection are:
 
 \begin{enumerate}
  \item The truncation functors uniquely factors a map $f: X \to Y$ into a $n$-connected map followed by a $n$-truncated map 
  (Corollary \ref{Cor Conn is Trunc}).
  \item Connected maps are stable under base change (Corollary \ref{Cor N Conn Base Change}).
 \end{enumerate}

 \begin{lemone} \label{Lemma:Tau Equiv if mapped in Trunc}
  Let $f: X \to Y$ be a map in $\E$. Then $\tau_n f$ is an equivalence if and only if for all $n$-truncated $Z$ the map
  $$Map(Y,Z) \to Map(X,Z)$$
  is an equivalence of spaces.
 \end{lemone}
 
 \begin{proof}
  By the Yoneda lemma, the map $\tau_n f: \tau_n X \to \tau_n Y$ in $\tau_n \E$ is an equivalence if and only if
  the map $(\tau_n f)^*: Map(\tau_n Y, Z) \to Map(\tau_n X, Z)$ is an equivalence of spaces for all $n$-truncated $Z$.
  \par 
  Now we have following commutative diagram
  \begin{center}
   \begin{tikzcd}[row sep=0.5in, column sep=0.5in]
    Map(\tau_n Y, Z) \arrow[r] \arrow[d, "\simeq"] & Map(\tau_n X, Z) \arrow[d, "\simeq"] \\
    Map(Y, Z) \arrow[r] & Map(X, Z)
   \end{tikzcd}
  \end{center}
  By adjunction the vertical maps are equivalences. Thus, the top map is an equivalence if and only if the bottom map is an 
  equivalence, which gives us the desired result.
 \end{proof}
 
 \begin{remone}
  We will prove a generalized version for all localizations in Lemma \ref{Lemma:Localized equivalences}.
 \end{remone}

  \begin{lemone} \label{Lemma F n Conn then Tau N F equiv}
  If $f: X \to Y$ is $n$-connected, then $\tau_n f: \tau_n X \to \tau_n Y$ is an equivalence.
 \end{lemone}
 
 \begin{proof}
  By the previous lemma we need to show that for every $n$-truncated object $Z$, the map 
  $$f^*: map(Y,Z) \to map(X,Z)$$
  is an equivalence.
  However, this follows immediately from using Proposition \ref{Prop Trunc Conn Lift NNO} 
  with the $n$-connected map $f: X \to Y$ and $n$-truncated map $Z \to 1$.
 \end{proof}
 
 \begin{remone}
  We will prove a general form of this result for all localizations in Lemma \ref{Lemma:Image of L have Lf triv}.
 \end{remone}

 \begin{remone} 
  Note the reverse does not hold. For a partial reverse argument see Theorem \ref{The:Tau n F Equiv then f N Min One Conn}.
 \end{remone}
 
 \begin{lemone} \label{Lemma:Unit Connected}
  \cite[Proposition 8.5]{Re05}
  Let $X$ be an object in $\E$. Then the unit of the adjunction $\eta_X: X \to \tau_n(X)$ is $n$-connected. 
 \end{lemone}
 
 \begin{proof}
  We have to prove that for every $n$-truncated map $g:Z \to \tau_n(X)$ the induced map 
  $$ Map_{/\tau_n(X)}( \tau_n X,Z) \to Map_{/\tau_n(X)}(X,Z) $$
  is an equivalence.
  We have following commutative square 
  \begin{center}
   \comsq{Map(\tau_n X, Z)}{Map(X, Z)}{Map(\tau_n X, \tau_n X)}{Map(X, \tau_n X)}{\simeq}{g_*}{g_*}{\simeq}
  \end{center}
  As $\tau_n X$ is $n$-truncated and $Z \to \tau_n X$ is also $n$-truncated, $Z$ is also $n$-truncated.
  Thus, by the previous lemma the horizontal maps are equivalences of spaces. Take the point 
  $id: * \to Map(\tau_n X, \tau_n X)$ and $\eta_X: * \to Map(X, \tau_n X)$. Then we can pullback the square above to the 
  following square: 
  \begin{center}
   \comsq{Map_{/\tau_n(X)}(\tau_n X, Z)}{Map_{/\tau_n(X)}(X, Z)}{*}{*}{\simeq}{}{}{}
  \end{center}
  As the pullback of an equivalence is still an equivalence, the top map is still an equivalence and this gives us the desired result.
 \end{proof}
  
 This lemma has following interesting corollary.
 
 \begin{corone} \label{Cor Conn is Trunc}
  Let $X$ be an object with an $n$-connected map $g:X \to Y$ such that $Y$ is $n$-truncated.
  Then $Y \simeq \tau_n(X)$.
 \end{corone}
 
 This corollary also gives us a stability of truncations
 
 \begin{propone} \label{Prop Pb of epimono is epimono}
  Let $f:Y \to X$ and $g: Z \to X$ be two maps. Moreover, let $Y \to \tau^X(f) \to X$ be the factorization of $f$.
  Then the pullback of the factorization along $g$, as depicted in the diagram below, is the factorization of $g^*f$.
  \begin{center}
   \begin{tikzcd}[row sep=0.5in, column sep=0.5in]
    g^*Y \arrow[r] \arrow[d] \arrow[rr, bend right = 20, "g^*f"'] \arrow[dr, phantom, "\ulcorner", very near start] & 
    g^*(\tau^X_n(Y)) \arrow[r] \arrow[d] \arrow[dr, phantom, "\ulcorner", very near start] & Z \arrow[d, "g"] \\
    Y \arrow[r] \arrow[rr, bend right = 20, "f"'] & \tau^X_n(Y) \arrow[r] & X
   \end{tikzcd}
  \end{center}
 \end{propone}

 \begin{proof}
  It suffices to prove the result for the case $X= 1$. In that case we need to prove that 
  $$Z \times Y \to Z \times \tau_n(Y) \to Z$$
  is the $n$-truncation factorization. By the previous corollary it suffices to show that 
  $Z \times Y \to  Z \times \tau_n(Y)$ is $n$-connected and $Z \times \tau_n(Y) \to Z$ is $n$-truncated.
  The first statement follows from Lemma \ref{Lemma Product by A is conn NNO} as 
  $Y \to \tau_n(Y)$ is $n$-connected. 
  The second statement holds immediately as $n$-truncated maps are stable under base change (Lemma \ref{Lemma Trunc Base Change NNO}).
 \end{proof}
 
 It also implies following proposition.
 
 \begin{propone} \label{Prop Conn iff Tau Final Object}
  An object $X$ in $\E$ is $n$-connected if and only if $\tau_n(X)$ is equivalent to the final object.
 \end{propone}
 
 \begin{proof}
  Let us assume $\tau_n(X)$ is equivalent to the final object.
  By the proposition above $X \to \tau_n(X) \simeq 1$ is $n$-connected, which implies that $X$ is 
  $n$-connected. 
  \par 
  On the other hand, assume $X$ is $n$-connected. Then the map $X \to 1$ is $n$-connected and 
  the map $1 \to 1$ is $n$-truncated. Thus, by the corollary $1 \simeq \tau_n(X)$.
 \end{proof}
 
  It is valuable to see this manifests for $\infty$-connected maps.
 
 \begin{corone}
  A map $f$ is $\infty$-connected  (Definition \ref{Def:Infinite Connected Map NNO}) 
  if and only if $\tau_n(f)$ is the final object for each natural number $n$. 
 \end{corone}

 We can also use it to show that truncations commute with pullbacks.
 
 \begin{propone}
  Let 
  \begin{center}
   \pbsq{W}{Z}{Y}{X}{}{f'}{f}{p}
  \end{center}
   be a pullback square. Then 
  \begin{center}
   \pbsq{\tau^Y_n W}{\tau^X_n Z}{Y}{X}{}{\tau_n^Y f'}{\tau_n^X f}{p}
  \end{center}
  is also a pullback square.
 \end{propone}

 \begin{proof}
  We know that $X \to \tau^X_n Z \to Z$ is the factorization of $f$. This means that $p^*X \to p^* \tau^X_n Z \to Y$ is a factorization 
  of $p^*f$, according to Proposition \ref{Prop Pb of epimono is epimono}. However, by assumption $p^*f \simeq f'$ and $p^*X \simeq W$. 
  This implies that $W \to p^* \tau^X_n Z \to Y$ is the factorization of $f': W \to Y$ finishing the proof.
 \end{proof}

 The result above also has following immediate corollary.
 
 \begin{corone} \label{Cor N Conn Base Change}
  Let 
  \begin{center}
   \pbsq{W}{Z}{Y}{X}{}{f'}{f}{p}
  \end{center}
   be a pullback square.Then if $f$ is $n$-connected, $f'$ is also $n$-connected.
 \end{corone}
 
 \begin{proof}
  If $f$ is $n$-connected then $\tau^X_n f$ is an equivalence (by Proposition \ref{Prop Conn iff Tau Final Object}).
  By the previous proposition, the pullback, namely $\tau^Y_n f'$, is an equivalence as well. 
  This, again by Proposition \ref{Prop Conn iff Tau Final Object} implies that $f'$ is $n$-connected.
 \end{proof}
  
 \begin{remone}
  Proposition \ref{Prop n Conn coBase Change NNO} combined with Corollary \ref{Cor N Conn Base Change} implies that 
  $n$-connected maps are stable under base change and cobase change. This will be quite important later on, when we study modalities
  in a topos (Proposition \ref{Prop Conn Trunc Modality}).
 \end{remone}
 
 We can use our new understanding of truncation functors to recover further results about truncated maps. 
 Notice these were proven initially in \cite{Re05} in the presentable setting. However, the proofs there just 
 use basic facts about truncated and connected maps. Thus, we will not repeat the proofs, but rather give the relevant references.
   
 \begin{propone}
 \cite[Proposition 8.11]{Re05}
  A map $f: Y \to X$ is $n$-connected if and only if for every $n$-truncated
  map $g: Z \to X$ the map of internal mapping objects
  $$[g:Z \to X]^{[id_X: X \to X]} \to [g: Z \to X]^{[f: Y \to X]} $$
  is an equivalence in $\E_{/X}$.
 \end{propone}
 
 \begin{remone}
  This proposition is just the internal version of the original Definition \ref{Def:Connected Maps NNO}.
 \end{remone}
 
 \begin{defone}
  Let $f: X \to Y$ and $g: W \to Z$ be two maps. We define the {\it gap map} $u_{f,g}$ as the following map.
  \begin{center}
   \begin{tikzcd}[row sep=0.5in, column sep=0.5in]
    W^Y \arrow[drr, bend left = 20, "f^*"] \arrow[ddr, bend right = 20, "g_*"] \arrow[dr, dashed, "u_{f,g}"] & & \\
    & Z^Y \underset{Z^X}{\times} W^X \arrow[d] \arrow[r] \arrow[dr, phantom, "\ulcorner", very near start] & W^X \arrow[d, "g_*"] \\
    & Z^Y \arrow[r, "f^*"] & Z^X 
   \end{tikzcd}
  \end{center}
 \end{defone}

 \begin{propone}
  \cite[Proposition 8.13]{Re05}
  Let $f: X \to Y$ be a map and $n \geq -2$.
  Then the following are equivalent:
  \begin{enumerate}
   \item For all $m \geq -2$ and all $m$-truncated maps $g: W \to Z$ 
   the induced map $u_{f,g}$ is $(m-n-2)$-truncated if $m \geq n$ and an 
   equivalence if $m \leq n$.
   \item For all $n$-truncated maps $g: W \to Z$ the map $u_{f,g}$ is an
   equivalence.
   \item $f$ is $n$-connected.
  \end{enumerate}
 \end{propone}
 
 We will now use the results proven in this subsection about $n$-truncated and $n$-connected maps to 
 further our understanding of $(-1)$-connected maps (Subsection \ref{Subsec:Neg One Conn})
 and the subcategory of $n$-truncated objects (Subsection \ref{Subsec:Universe of Truncated Objects}).
  
 \subsection{(-1)-Connected Maps} \label{Subsec:Neg One Conn}
 As we observed in Section \ref{Sec Truncated and Connected Objects} $(-1)$-truncated morphism are special, 
 in particular they just give us the mono maps in our $(\infty,1)$-category. 
 We want to show that $(-1)$-connected maps are special too. 
 \par 
 Concretely, we prove a simple way to characterize $(-1)$-connected maps (Lemma \ref{Lemma Sub Cover}), 
 which gives an easy way to construct many $(-1)$-connected maps. 
 Then we use $(-1)$-connectedness to generalize the following result from classical topology:
 A non-empty space is $n$-connected if and only if each loop space is $n-1$-connected.
 This will be proven in Proposition \ref{Prop Loop n Conn N One Conn}.
 
 \begin{remone}
  A $(-1)$-connected map is also called an {\it effective epimorphism} \cite{Lu09} \cite{Re05} or even {\it cover} \cite{ABFJ17}.
  However, in order to avoid additional notation, we will simply call them $(-1)$-connected maps.
 \end{remone}
 
 First of all we present a nice way to characterize $(-1)$-connected maps.
 It was proven for Grothendieck $(\infty,1)$-toposes in \cite[Lemma 7.9]{Re05} and \cite[Proposition 6.2.3.10]{Lu09}.
 We generalize it here to elementary $(\infty,1)$-toposes.
 
 \begin{lemone} \label{Lemma Sub Cover}
  A map $f:Y \to X$ is $(-1)$-connected if and only if the induced map 
  on the set of subobjects, $f^*: Sub(X) \to Sub(Y)$, is injective.
 \end{lemone}
 
 \begin{proof}
  First, notice that $f^*$ is injective if and only if $f^*(A) = f^*(B)$ implies $A = B$ 
  for all $A \leq B$. That is because $f^*$ preserves the meet of two subobjects.
  Thus, we will always assume $A \leq B$ throughout the proof.
  \par 
  In order to prove the result we first have to gain a better understanding of $(-1)$-connected
  maps. By definition $Y \to X$ is $(-1)$-connected if we have an equivalence 
  $$Map_{/X}(X,Z) \to Map_{/X}(Y, Z)$$
  where $Z \to X$ is $(-1)$-truncated. However, the fact that $Z \to X$ is $(-1)$-truncated means
  that $Z$ is a subobject of $X$. Thus, the space $Map_{/X}(X,Z)$ is contractible if and only if 
  $Z = X$, otherwise it is empty. The equivalence then implies that $Map_{/X}(Y,Z)$
  has to be empty except if $Z = X$.
  \par 
  We can summarize this analysis by saying that in the diagram below: 
  
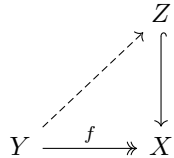
\begin{center}
   \begin{tikzcd}[row sep=0.5in, column sep=0.5in]
    & Z \arrow[d, hookrightarrow] \\
    Y \arrow[r, "f", twoheadrightarrow] \arrow[ur, dashed] & X
   \end{tikzcd}
  \end{center}
  $f$ being $(-1)$-connected is equivalent to a lift existing if and only if $Z$ is the maximal subobject.
  We will use this statement to prove the lemma.
  \par 
  First, assume that $f$ is $(-1)$-connected. We want to prove that $f^*: Sub(X) \to Sub(Y)$ is injective.
  Let $A$ and $B$ be two subobjects of $X$ such that $f^*(A) = f^*(B)$ and $B \leq A$ (using the realization 
  from the first paragraph). Then this gives us following diagram
  \begin{center}
   \begin{tikzcd}[row sep=0.5in, column sep=0.5in]
    f^*(B) \arrow[r, twoheadrightarrow] \arrow[d, hookrightarrow, "\simeq"] & B \arrow[d, hookrightarrow] \\
    f^*(A) \arrow[r, twoheadrightarrow] \arrow[d, hookrightarrow] & A \arrow[d, hookrightarrow] \\
    Y \arrow[r, "f", twoheadrightarrow] & X
   \end{tikzcd}
  \end{center}
  We already know that $(-1)$-connected maps are stable under base change by Corollary \ref{Cor N Conn Base Change}. This implies that the map 
  $f^*(A) \to A$ is also $(-1)$-connected. By the previous paragraph then $f^*(A)$ has a lift to $B$
  if and only if $A = B$.
  \par 
  On the other side assume that $f^*: Sub(X) \to Sub(Y)$ is injective. We will prove that 
  $f:Y \to X$ is $(-1)$-connected.
  By the previous paragraphs it suffices to prove that $f$ lifts to a subobject $Z$ of $X$ if and only if 
  $Z=X$. Thus, let $Z \hookrightarrow X$ be a subobject and assume a lift $Y \to Z$ exists. 
  Thus, we have the pullback diagram
  \begin{center}
   \pbsq{Y}{Z}{Y}{X}{}{}{}{f}
  \end{center}
  which means that $f^*(Z) = f^*(X) = Y$. By injectivity of $f^*$ we get $Z \simeq X$, which finishes the proof.
 \end{proof}

 \begin{exone}
  Using this lemma we can immediately deduce that a $(-1)$-connected map of spaces is a map that is surjective on 
  path components.
 \end{exone}
 
 We can use the previous lemma to find many interesting basic results about $(-1)$-connected maps.

 \begin{lemone} \label{Lemma Pushout Cover}
  $f: C \to A$ and $g: C \to B$ be two maps. Then the natural map 
  $$A \coprod B \to A \coprod_C B$$
  is $(-1)$-connected.
 \end{lemone}
 
 \begin{proof}
  We will show the induced map on subobjects is injective. First notice that $Sub$ is represented by the subobject classifier $\Omega$. 
  This gives us isomorphisms
  $$Sub(A \coprod_C B) \cong Hom(A \coprod_C B, \Omega) \cong Hom(A, \Omega) \underset{Hom(C, \Omega)}{\times} Hom(B,\Omega)$$
  $$Sub(A \coprod B) \cong Hom(A \coprod B, \Omega) \cong Hom(A, \Omega) \times Hom(B,\Omega)$$
  Thus, we need to prove that the map 
  $$Hom(A, \Omega) \underset{Hom(C, \Omega)}{\times} Hom(B,\Omega) \to Hom(A, \Omega) \times Hom(B,\Omega)$$
  is injective map of sets, which is clearly true by the definition of pullback of sets.
 \end{proof}

 \begin{remone}
  In fact, if we assume that all colimits exist, then this lemma also holds in a more general setting.
  For a proof in a higher Grothendieck topos (which by definition has all colimits) see \cite[Lemma 6.2.3.13]{Lu09}
 \end{remone}

 \begin{remone}
  Note the analogous statement for pullbacks and $(-1)$-truncated maps does not hold. In other words the map 
  $$A \underset{C}{\times} B \to A \times B$$
  is generally not in any way truncated. For an example in the category of spaces let $A = B = *$, the one point space, 
  and $C$ be any space. Then the diagram above is the map $\Omega C \to *$ from the loop space of $C$, which is generally not 
  in any way truncated (just take $C = S^2$).
 \end{remone}
 
 Notice epi and $(-1)$-connected are not related the same way that $(-1)$-truncated maps and monos are related (Proposition \ref{Prop Conn Eq Mono}).
 
 \begin{defone}
  A map $f: X \to Y$ is an {\it epimorphism} if for any object $Z$ the induced map 
  $$Map(Y,Z) \to Map(X,Z)$$
  is a $(-1)$-truncated map of spaces.
 \end{defone}

 \begin{exone}
  The map of spaces $S^1 \to *$ is $(-1)$-connected. However, it is clearly not epi. Indeed, if  
  the map $Map(*,Z) \to Map(S^1,Z)$ is mono for any space $Z$ then by the pullback below 
  \begin{center}
   \pbsq{Map(S^2,Z)}{Map(*,Z)}{Map(*,Z)}{Map(S^1,Z)}{}{}{}{}
  \end{center}
   $Map(S^2,Z)$ would be equivalent to $Map(*,Z)$.
 \end{exone}

 \begin{propone} \label{Prop Equiv Local}
  In the pullback diagram below where $p$ is a $(-1)$-connected map
   \begin{center}
    \begin{tikzcd}[row sep=0.5in, column sep=0.5in]
     Y' \arrow[r, twoheadrightarrow, "g'"] \arrow[d, "f'"] \arrow[dr, phantom, "\ulcorner", very near start] & Y \arrow[d, "f"]  \\
     X' \arrow[r, twoheadrightarrow, "g"] & X 
    \end{tikzcd}
   \end{center}
    we have the following. $f$ is a weak equivalence/$n$-truncated map/$n$-connected map 
    if and only if $f'$ is a weak equivalence/$n$-truncated map/$n$-connected map.
 \end{propone}
 \begin{proof}
   Clearly if $f$ satisfies any of those conditions then $f'$ does as well as they are stable under pullback.
   Thus, we will assume $f'$ satisfies the conditions and show that $f$ satisfies them as well.
   \par 
   First we assume that $f'$ is an equivalence and we will prove that $f$ is an equivalence.
   Here we make explicit use of the construction of the $(-1)$-truncation as given in Section \ref{Sec:The Join Construction}. 
   The commutative diagram above then results in a diagram 
   \begin{center}
    \begin{tikzcd}[row sep=0.5in, column sep=0.5in]
     Y' \arrow[r, "inl"] \arrow[d, "f'", "\simeq"'] & Y' \ast Y' \arrow[d, "f' \ast f'", "\simeq"'] \arrow[r, "inl"] & 
     (Y' \ast Y') \ast Y' \arrow[d, "(f' \ast f') \ast f'", "\simeq"'] \arrow[r] & ... \arrow[r] & Y \arrow[d, "f"] \\
     X' \arrow[r, "inl"] & X' \ast X' \arrow[r, "inl"] & (X' \ast X') \ast X' \arrow[r]& ... \arrow[r] & X
    \end{tikzcd}
   \end{center}
   The horizontal maps are sequential diagrams. By Theorem \ref{The Neg Trunc from Join} the sequential colimits of those diagrams construct the $(-1)$-truncations of $g$ and $g'$.
   However, the maps $f$ and $f'$ are $(-1)$-connected and so by Corollary \ref{Cor Conn is Trunc} they are the actual 
   truncations which means the horizontal sequential diagram are sequential colimit diagrams. 
   Moreover, the vertical maps $f'$, $f' \ast f'$, $(f' \ast f') \ast f'$... are all equivalences and so by the homotopy invariance of sequential colimits, 
   the colimit of the maps, namely $f$, is also an equivalence.
   \par 
   We now want to prove that if $f'$ is $n$-truncated or $n$-connected then $f$ is $n$-truncated or $n$-connected. Notice the pullback diagram above gives us following 
   pullback diagram. Using $n$-truncation functor we get following diagram
   \begin{center}
    \begin{tikzcd}[row sep=0.5in, column sep=0.5in]
     Y' \arrow[r, twoheadrightarrow, "g'"] \arrow[d, "i'"] \arrow[dr, phantom, "\ulcorner", very near start] & Y \arrow[d, "i"]  \\
     \tau_n^{X'}(f') \arrow[r, twoheadrightarrow] \arrow[d, "p'"] & \tau_n^X(f) \arrow[d, "p"] \\
     X'\arrow[r, twoheadrightarrow, "g"] & X 
    \end{tikzcd}
   \end{center}
   First, notice the map $\tau_n^{X'}(f') \to \tau_n^{X}(f)$ is also $(-1)$-connected, as $(-1)$-connected maps are stable under base change 
   (Corollary \ref{Cor N Conn Base Change}). 
   Assume, $f'$ is $n$-connected. Then $p'$ is an equivalence, which by the previous paragraph implies that $p$ is an equivalence, 
   which means that $f$ is $n$-connected. If $f'$ is $n$-truncated then we use the fact that $i'$ is an equivalence 
   to similarly deduce that $f$ is $n$-truncated as well.
 \end{proof}
 
 \begin{remone}
  This proof is a special case of a more general proof about modalities and locality. 
  The general case will be covered in Proposition \ref{Prop LR local}.
 \end{remone}

  $K$ be a space. If we know that $K$ is non-empty, then $K$ is $n$-connected 
  if and only if each loop space $\Omega_k K$ is $(n-1)$-connected. 
  Using this basic observation we can prove connectivity results using inducation.
  \par 
  We want to show that something similar holds in this setting (Proposition \ref{Prop Loop n Conn N One Conn}), 
  where $(-1)$-connectedness plays the role of being non-empty.
  In order to prove it, we first have to prove several lemmas.
 
 \begin{lemone} \label{Lemma Equiv Neg One Conn}
  Let $f,g : A \to B$ be two maps. Then the following are equivalent.
  \begin{enumerate}
   \item $f \simeq g$.
   \item The universal map $E \to A$ is $(-1)$-connected, where $E$ is the equalizer of $f,g: A \to B$. 
   \item The map $(f,g): A \to B \times B$ lifts to a map $A \to \tau^{B \times B}_{-1}(\Delta_B)$.
  \end{enumerate}
 \end{lemone}
 
 \begin{proof}
  {\it (1) $\Rightarrow$ (2)}
  If $f \simeq g$ then we have following diagram 
  \begin{center}
   \begin{tikzcd}[row sep=0.5in, column sep=0.5in]
    A \arrow[dr, "id_A"] \arrow[d, dashed] & & \\
    E \arrow[r, "p"] & A \arrow[r, shift left=0.06in, "f"] \arrow[r, shift right=0.06in, "g"'] & B 
   \end{tikzcd}
  \end{center}
  where the dashed arrow exists because of the universal property of the equalizer. 
  This implies that $E \to A$ is $(-1)$-connected, as explained in Remark \ref{Rem:Cover gives Neg One Conn}. 
  
  {\it (2) $\Rightarrow$ (3)}
  If $E \to A$ is $(-1)$-connected then the following square has a lift as the right hand map is 
  always $(-1)$-truncated.
  \begin{center}
   \liftsq{E}{\tau^{B \times B}_{-1}(\Delta_B)}{A}{B \times B}{}{}{}{( f \comma g )}
  \end{center}
  
  {\it (3) $\Rightarrow$ (1)}
  Let us assume the map $(f,g)$ has a lift as described above. Then applying the $map(D,-)$ for an 
  object $D$ gives us a diagram of spaces
  \begin{center}
   \begin{tikzcd}[row sep=0.5in, column sep=0.5in]
    & map(D, \tau^{B \times B}_{-1}(\Delta_B)) \arrow[d, hookrightarrow] \\
    map(D,A) \arrow[r, "(f \comma g)_*"] \arrow[ur, dashed] & map(D,B) \times map(D,B) 
   \end{tikzcd}
  \end{center}
  The space $map(D, \tau^{B \times B}_{-1}(\Delta_B))$ is by definition the subspace of $map(D,B)$ consisting of all points 
  $(x,y)$ such that $ x \simeq y$. Thus, we have an equivalence $f_* \simeq g_*$. As this holds for every object $D$ 
  we get an equivalence $f \simeq g$.
 \end{proof}
 
 \begin{lemone} \label{Lemma Constant Loop Space Constant}
  Assume that in the following diagram  
  \begin{center}
    \liftsq{A}{Z}{A \times A}{Z \times Z}{g}{\Delta}{\Delta}{g \times g}
   \end{center}
   a lift exists. Then $g: A \to Z$ factor through $\tau_{-1}(A)$.
 \end{lemone}

 \begin{proof}
  First notice, $\pi_1, \pi_2: A \times A \to A$ are both retracts of $\Delta_A: A \to A \times A$.
  Thus, the existence of the lift $A \times A \to Z$ implies that $g \pi_1 \simeq g \pi_2$. 
  This means we have a diagram 
  \begin{center}
   \begin{tikzcd}[row sep=0.5in, column sep=0.5in]
    A \times A \arrow[r, "\pi_1", shift left =0.06in] \arrow[r, "\pi_2"', shift right=0.06in] & 
    A \arrow[r] \arrow[d, "g"] & A \ast A \arrow[dl, "g_1", dashed] \\
    & Z & 
   \end{tikzcd}
  \end{center}
  where the map $g_1: A \ast A \to Z$ exists because of the universal property. 
  \par 
  Using a similar argument we get maps $g_n: A^{\ast n} \to Z$, which gives us a map $g_{\infty}: \tau_{-1}(A) \to Z$.
 \end{proof}
 
 \begin{propone} \label{Prop Loop n Conn N One Conn}
  Let $A$ be $(-1)$-connected and $A \to A \times A$ be $n$-connected.
  Then $A$ is $(n+1)$-connected.
 \end{propone}

 \begin{proof}
  Fix an $(n+1)$-truncated object $Z$. We have to prove that the map 
  $$ \alpha: Map(1,Z) \to Map(A,Z)$$ 
  is an equivalence. First we show that the map 
  $$ \pi_0(Map(1,Z)) \to \pi_0(Map(A,Z))$$
  is a bijection of sets.
  From this square we can get the diagram 
   \begin{center}
    \pbsq{Map(A \times A,Z)}{Map(A \times A,Z \times Z)}{Map(A,Z)}{Map(A,Z \times Z)}{\Delta_*}{\Delta^*}{\Delta^*}{\Delta_*}
   \end{center}
   corresponding to the lifting problem
      \begin{center}
    \liftsq{A}{Z}{A \times A}{Z \times Z}{}{\Delta}{\Delta}{}
   \end{center}
  Given that $Z \to Z \times Z $ is $n$-truncated (Proposition \ref{Prop Diag Trunc NNO}) 
  and $A \to A \times A$ is $n$-connected, this square is a homotopy pullback square. 
  By the previous lemma this implies that every map $A \to Z$ will factor through a map $\tau_{-1}(A) = 1 \to Z$, giving us a map 
  $$ \beta: Map(A,Z) \to Map(1,Z)$$ 
  We need to prove that this map is the inverse of the map above. By definition $\alpha \circ \beta$ is the identity.
  On the other hand, let $z: 1 \to Z$ be any map. 
  We need to show that $\beta \circ \alpha (z) \simeq z$. By construction we have a diagram 
  
  \begin{center}
   \begin{tikzcd}[row sep=0.5in, column sep=0.5in]
    A \arrow[dr] \arrow[d, dashed] & & \\
    E \arrow[r, twoheadrightarrow] & 1 \arrow[r, "z", shift left=0.06in] \arrow[r, "z'", shift right=0.06in] & Z 
   \end{tikzcd}
  \end{center}
  
  As the map $A \to 1$ is $(-1)$-connected, $E \to 1$ is also $(-1)$-connected, which by Lemma \ref{Lemma Equiv Neg One Conn} implies $z \simeq z'$.
  \par 
  We will now generalize this argument to higher homotopical levels. 
  If $Map(1,Z)= \emptyset$ then we are done. Otherwise, fix a point $z: 1 \to Z$. We need to prove that for any $m \geq 0$
  $$ \pi_0(\Omega^m_z Map(1,Z)) \to \pi_0(\Omega^m_{z!}Map(A,Z))$$
  is a bijection of sets.
  However, the map above is just equivalent to 
  $$ \pi_0(Map(1, \Omega^m_z Z)) \to \pi_0(Map(1, \Omega^m_z Z))$$
  and this map is an equivalence by the argument in the previous paragraph, as $\Omega^m_z Z$ is also $n$ truncated.
 \end{proof}

 \subsection{Universe of Truncated Objects} \label{Subsec:Universe of Truncated Objects}
 In this final subsection we want to use the truncation functor to study the sub-universe of $n$-truncated objects. 
 The key observation about it is that it classifies the $n$-truncation functor
 (Proposition \ref{Prop:U leq classifies trunc}) and is itself (internally) 
 $(n+1)$-truncated (Theorem \ref{The:U Leq n is n plus trunc}).
 \par 
 We can then use the fact that the subuniverse is $(n+1)$-truncated to prove some invariance result for $n$-connected maps,
 such as Lemma \ref{Lemma Tau n f equiv}, which itself leads to further ways of determining $n$-connected maps
 (Theorem \ref{The:Tau n F Equiv then f N Min One Conn}).
 \par 
 We can expand our universe $\U$ into a complete Segal universe $\U_\bullet$ (Definition \ref{Def:CSU}, Theorem \ref{The:CSU Exist}).
 
 \begin{defone} \label{Def:U leq n}
  The truncation functor and embedding functor 
  $$i\tau_n: \E \to \E$$
  induces an internal map of complete Segal universes 
  $$i\tau_n: \U_\bullet \to \U_\bullet$$
  We define the {\it complete Segal sub-universe of $n$-truncated objects} $\U^{\leq n}_\bullet$ as the image (Definition \ref{Def:Image})
  $$\U^{\leq n}_\bullet = \tau^{\U_\bullet}_{-1}(i\tau: \U_\bullet \to \U_\bullet)$$ 
 \end{defone}
 
 This object has following important universal property.
 
 \begin{lemone} \label{Lemma:U leq n classifies n trunc}
  Let $S$ be the class of maps classified by $\U_\bullet$. For each object $Z$ we have an equivalence 
  $$\tau_n((\E_{/Z})^S) \simeq Map(Z, \U_{\bullet}^{\leq n})$$
 \end{lemone}

 \begin{proof}
  We can factor every map of categories uniquely into an essential surjection followed by an embedding. 
  We can apply this factorization to the map $i \tau_n : \E \to \E$. 
  However, we know that $i$ is an embedding and $\tau_n$ is essentially surjective and so by uniqueness $(i, \tau_n)$ is 
  the unique factorization. 
  Thus, we get a diagram
  \begin{center}
   \begin{tikzcd}[row sep=0.5in, column sep=0.5in]
    (\E_{/Z})^S \arrow[r, "\tau_n", twoheadrightarrow] \arrow[d, "\simeq"] & \tau_n((\E_{/Z})^S) \arrow[r, "i", hookrightarrow] \arrow[d] & 
    (\E_{/Z})^S \arrow[d, "\simeq"] \\
    Map(Z, \U_{\bullet}) \arrow[r, twoheadrightarrow] & Map(Z, \U_{\bullet}^{\leq n}) \arrow[r, hookrightarrow] & Map(Z, \U_{\bullet})
   \end{tikzcd}
  \end{center}
  By the uniqueness of epi mono factorizations we thus get the desired equivalence. 
 \end{proof}

 \begin{remone}
  Notice that any universe $\U$ is itself an object in $\E$ and thus is classified by a higher universe. 
  This means that it also has a truncation $\tau_n(\U)$ using the larger universe. However, this differs from the object $\U^{\leq n}$. 
 \end{remone}
 
 \begin{exone}
  Let $\E$ be the category of (not necessarily small) spaces. In that case the core of the subcategory of small spaces $\U$ 
  is a universe in $\E$. Then $\U^{\leq n}$ is the core of the subcategory of $n$-truncated small spaces.
  However, on the other hand $\tau_n(\U)$ is just the truncation. So, for example $\tau_{-1}(\U) = *$ and 
  $\tau_0(\U)$ is the set of homotopy equivalence classes of small spaces. 
  On the other hand, $\U^{\leq -1} = \{ \emptyset, * \}$ and $\U^{\leq 0}$ is the groupoid of small sets.
 \end{exone}

 The truncation of the universe gives us an alternative way to construct truncations.
 
 \begin{propone} \label{Prop:U leq classifies trunc}
  Let $p: Y \to X$ be a map classified by $\ulcorner p\urcorner : X \to \U$. Then the composition 
  $\tau_n \circ \ulcorner p \urcorner : X \to \U^{\leq n}$ classifies the truncation $\tau^X(Y) \to X$
 \end{propone}
 
 \begin{proof}
  This follows immediately from the commutative diagram 
  \begin{center}
   \pbsq{\E_{/X}}{\tau_n(\E_{/X})}{Map(X, \U)}{Map(X, \U^{\leq n})}{\tau_n}{\simeq}{\simeq}{\tau_n}
  \end{center}
 \end{proof}
 
  If $n$ is a standard natural number, then $\tau_n \E$ is a $(n+1,1)$-category. In particular, 
   $(\tau_n \E)^{core}$ is an $(n+1)$-truncated space as every loop space is 
   $n$-truncated. We want to show that the result generalizes appropriately to internal truncation levels

 \begin{theone} \label{The:U Leq n is n plus trunc}
  Let $\U_{\bullet}$ be a complete Segal universe. Then for each $k$ the object $\U^{\leq n}_k$ is $(n+1)$-truncated.
 \end{theone}
 
 \begin{proof}
  By the Segal condition and the fact that truncated objects are closed under pullbacks it suffices to prove the desired 
  result for $\U^{\leq n}_0$ and $\U^{\leq n}_1$. First, notice we have an equivalence (Theorem \ref{The:CSU Exist})
  $$(s,t): \U^{\leq n}_1 \to \U^{\leq n}_0 \times \U^{\leq n}_0 = [\U^{\leq n}_0 \times (\U^{\leq n}_0)_* \to \U^{\leq n}_0 \times \U^{\leq n}_0]^{[
  (\U^{\leq n}_0)_* \times \U^{\leq n}_0 \to \U^{\leq n}_0 \times \U^{\leq n}_0]}$$
  Thus, the map $(s,t): \U^{\leq n}_1 \to \U^{\leq n}_0 \times \U^{\leq n}_0$ is $(n+1)$-truncated. 
  By the completeness condition, we have an inclusion 
  \begin{center}
   \begin{tikzcd}[row sep=0.5in, column sep=0.5in]
    \U^{\leq n}_0 \arrow[rr, hookrightarrow] \arrow[dr, "\Delta"'] & & \U^{\leq n}_1 \arrow[dl, "(s \comma t)"] \\
    & \U^{\leq n}_0 \times \U^{\leq n}_0
   \end{tikzcd}
  \end{center}
  This proves that the diagonal map $\Delta$ is $n$-truncated, which proves that $\U^{\leq n}_0$ is $n+1$-truncated 
  (Proposition \ref{Prop Diag Trunc NNO}). As the map $(s,t)$ is $n$-truncated and so also $(n+1)$-truncated, it 
  follows that $\U^{\leq n}_1$ is also $(n+1)$-truncated, which gives us the desired result.
 \end{proof}

 We will use the fact that $\U^{\leq n}_\bullet$ is $(n+1)$-truncated to prove some interesting new results about truncation functors.
 
 \begin{lemone} \label{Lemma Tau n f equiv}
  Let $f:Y \to X$ be $m$ connected and $n < m$. Then the induced pullback map on the full subcategories 
  $$ \tau_n f^*:  \tau_n(\E_{/X}) \to \tau_n(\E_{/Y})$$
  is an equivalence.
 \end{lemone}

 \begin{proof} 
  By the proposition above, we have an equivalence $\tau_n(\E_{/X}) \simeq map(X,\U^{\leq n}_\bullet)$ and we also 
  know that $\U^{\leq n}_\bullet$ is $n+1$-truncated 
  So, by Lemma \ref{Lemma F n Conn then Tau N F equiv}, we get the desired equivalence 
  $$ map(X,\U^{\leq n}_\bullet) \simeq map(Y, \U^{\leq n}_\bullet)$$
 \end{proof}
 
 This lemma has one very important application.
 
 \begin{corone} \label{Cor:Elementary Equiv}
  Let $X$ an object and $n < m$. Then we have an equivalence of categories 
  $$ \tau_n (\eta_X)^*: \tau_n(\E_{/\tau_m X}) \to \tau_n(\E_{/X})$$
 \end{corone}
 
 \begin{remone}
  I want to thank Charles Rezk for a helpful discussion and for giving me a new perspective on his Lemma 
  which led to the formulation and proof given in Lemma \ref{Lemma Tau n f equiv}
 \end{remone}

 \begin{remone}
  This corollary is usually expressed in the following form: Let $f: Y \to X$ be $m$-truncated and $n > m$. Then 
   \begin{center}
   \pbsq{Y}{\tau_n Y}{X}{\tau_n X}{}{f}{\tau_n f}{}
  \end{center}
  is a pullback square. 
 \end{remone}
 
 \begin{remone}
  Lemma \ref{Lemma Tau n f equiv} has been proven in the context of a Grothendieck $(\infty,1)$-topos by various authors, such as 
  \cite[Lemma 8.6]{Re05}, \cite[Lemma 4.3.3]{SY19}. However, the proofs there rely on generalizing the proof from spaces, 
  using the fact that every Grothendieck $(\infty,1)$-topos is a localization of a presheaf topos, rather than relying on intrinsic 
  properties. 
  \par 
  This is an interesting example where finding a proof in the context of an elementary $(\infty,1)$-topos also gives us a new proof for 
  a Grothendieck $(\infty,1)$-topos that does not depend on any specific presentation. 
 \end{remone}
  
  \begin{remone}
   This result was also proven independently in the context of homotopy type theory \cite[Proposition 2.18]{DORS18}.   
  \end{remone}

 We can use this result to prove new interesting results about connected maps.
 
  \begin{theone} \label{The:Tau n F Equiv then f N Min One Conn}
   Let $f: Y \to X$ be a map. If $\tau_n f$ is a weak equivalence, then $f$ is $(n-1)$-connected.
  \end{theone}
  
  \begin{proof}
   Let $Z \to X$ be a $(n-1)$-truncated map. We need to prove that 
   $$map_{/X}(X, Z) \to map_{/X}(Y, Z)$$
   is an equivalence. In order to prove it we need several equivalences. 
   \par 
   First, by Corollary \ref{Cor:Elementary Equiv}, we have an equivalence
   $$Z \xrightarrow{ \ \simeq \ } X \underset{\tau_nX}{\times} \tau_n Z$$
   and therefore it suffices to prove that we have an equivalence 
   $$map_{/X}(X, X \underset{\tau_nX}{\times} \tau_n Z) \to map_{/X}(Y, X \underset{\tau_nX}{\times} \tau_n Z).$$
   \par 
   Using the adjunction 
   \begin{center}
    \adjun{\E_{/X}}{\E_{/ \tau_nX}}{\eta_!}{\eta^*}
   \end{center}
    this is equivalent to proving that the map 
    $$map_{/\tau_nX}(X, \tau_n Z) \to map_{/\tau_nX}(Y,  \tau_n Z)$$
    is an equivalence.
   \par 
   Now, in the commutative diagram
   \begin{center}
    \begin{tikzcd}[row sep=0.5in, column sep=0.5in]
     map_{/\tau_n X}(X, \tau_n Z) \arrow[r, "f^*"]  & map_{/\tau_n X}(Y, \tau_n Z) \\
     map_{/\tau_n X}(\tau_n X, \tau_n Z) \arrow[u, "\simeq"] \arrow[r, "\tau_n(f)^*"] & map_{/\tau_n X}(\tau_n Y, \tau_n Z)\arrow[u, "\simeq"]
    \end{tikzcd}
   \end{center}
   the vertical maps are equivalences as $X \to \tau_nX$ and $Y \to \tau_n Y$ are $n$-connected
   and $\tau_n Z \to \tau_n X$ is $n$-truncated.
   Thus it suffices to prove the bottom map is an equivalence.
   However, that follows immediately from the assumption.
  \end{proof}
  
  \begin{remone}
   We have now proven following chain of results
   \begin{center}
    \begin{tikzcd}[row sep=0.5in, column sep=0.7in]
     f \ \text{ is } n-\text{connected} \arrow[r, Rightarrow, "\text{Lemma }\ref{Lemma F n Conn then Tau N F equiv}"] & 
     \tau_n(f)\text{ is an equivalence} \arrow[r, Rightarrow, "\text{Theorem }\ref{The:Tau n F Equiv then f N Min One Conn}"] & 
     f \ \text{ is } (n-1)-\text{connected} 
    \end{tikzcd}
   \end{center}
   Thus, we can think of $\tau_n(f)$ being an equivalence as a condition between being $n$-connected and $(n-1)$-connected. 
   This has, for example, let Schlank and Yanovski to call such a morphism $(n - \frac{1}{2})$-connected 
   \cite[Definition 4.3.1]{SY19}.
  \end{remone}

  We can use this Lemma to give a partial inverse to Proposition \ref{Prop Conn fg NNO}.
  
  \begin{corone}
   Let $f: X \to Y$, $g: Y \to Z$ be two maps, such that $g$ and $gf$ are $n$-connected.
   Then $f$ is $(n-1)$-connected.
  \end{corone}
 
  \begin{proof}
   We can think of $g$ and $gf$ as objects in $\E_{/Z}$. The fact that they are $n$-connected 
   implies that $\tau_n^Z(g) \simeq \tau_n^Z(f) \simeq id_Z$, the final object in $\E_{/Z}$ (Proposition \ref{Prop Conn iff Tau Final Object}).
   Thus, the map $\tau_n^Z(f): \tau_n^Z(g) \to \tau_n^Z(f)$ is an equivalence.
   The result now follows immediately from the previous Lemma. 
  \end{proof}

\section{Constructing Truncations} \label{Sec:Constructing Truncations}
 In this section we present various ways of constructing the $n$-truncation functor.
 \par 
 First, in Subsection \ref{Subsec:Constructing Localizations via Universes}
 we take a much more abstract approach: 
 We give several equivalent conditions that characterize a localization functor (Theorem \ref{The:Localizations}).
 The main difficulty of the proof is to construct a localization functor out of a sub-universe as an ``internal right Kan extension".
 This is proven separately in Proposition \ref{Prop:Sub Univ gives Trunc}.
 Having done so, we immediately get a truncation functor (Corollary \ref{Cor:Truncation Functor Construction}).
 \par 
 In Subsection \ref{Subsec:Inductive Construction of Truncations} we want to give an alternative characterization of the $n$-truncation functor that 
 inductively defines the $(n+1)$-trunction as the $n$-truncation of the loop objects. The definition is theoretically more convoluted, but
 is more intuitive than an abstract internal right Kan extension. This is the content of Theorem \ref{The:Constructing Trunctions Inductively}.
 
 \begin{remone} \label{Rem:Intro Condition for Localization Section}
   In this section $\E$ is an $(\infty,1)$-category that satisfies following conditions:
   \begin{enumerate}
    \item It has finite limits and colimits.
    \item It is locally Cartesian closed (Definition \ref{Def:LCCC}).
    \item It has sufficient universes $\U$ (Definition \ref{Def:Sufficient Universes})
    that are closed under finite limits and colimits and are locally Cartesian closed (Definition \ref{Def:Closed Universes}).
   \end{enumerate}

  \end{remone}
  
\subsection{Constructing Localizations via Universes} \label{Subsec:Constructing Localizations via Universes}
 In this subsection we give several equivalent ways to construct localization functors (Theorem \ref{The:Localizations}) 
 and then use it to construct $n$-truncation functors (Corollary \ref{Cor:Truncation Functor Construction}).
 
 First we give several definitions and then we prove that the various definitions are equivalent.
 
 \begin{defone}
  A {\it reflective subcategory of arrows} is a full subcategory $L\Arr(\E) \hookrightarrow \Arr(\E)$ that has a left adjoint 
  \begin{center}
   \begin{tikzcd}[row sep=0.5in, column sep=0.5in]
    \Arr(\E) \arrow[dr, "t"'] \arrow[rr, "L", shift left=1.5] \arrow[rr, hookleftarrow, shift right=1.5]& & L\Arr(\E) \arrow[dl, "t"]  \\
    & \E & 
   \end{tikzcd}
  \end{center}
 \end{defone}

 \begin{defone} \label{Def:Family of Refl Subcat}
  A {\it family of reflective subcategories} is a choice of reflective subcategories $L(\E_{/X}) \hookrightarrow \E_{/X}$ that 
  is pullback stable, meaning that for every map $f:Y \to X$ we have following diagram of adjunctions
  \begin{center}
   \begin{tikzcd}[row sep=0.5in, column sep=0.5in]
    \E_{/X} \arrow[r, "L", "\perp"', shift left=1.8] \arrow[r, hookleftarrow, shift right=1.8] 
    \arrow[d, "f^*", "\vdash"', shift left=1.8] \arrow[d, "f_*"', shift right=1.8, leftarrow]
    & L\E_{/X} \arrow[d, "f^*", "\vdash"', shift left=1.8] \arrow[d, "f_*"', shift right=1.8, leftarrow] \\
    \E_{/Y} \arrow[r, "L", "\perp"', shift left=1.8] \arrow[r, hookleftarrow, shift right=1.8] & L\E_{/Y}
   \end{tikzcd}
  \end{center}
 \end{defone}

 \begin{defone}
  Let $\U_\bullet$ be a complete Segal universe. A {\it reflective sub-complete Segal universe} $\U^{loc}_\bullet \hookrightarrow \U_\bullet$ 
  is a sub-complete Segal universe such that there is an adjoint of complete Segal objects
  \begin{center}
   \adjun{\U_\bullet}{\U^{loc}_\bullet}{}{}
  \end{center}
 \end{defone}
 
 \begin{defone}
  Let $\U_\bullet$ be a complete Segal universe. An {\it ideal sub-complete Segal universe} $\U^{loc}_\bullet$ is a sub-complete Segal universe 
  such that for all $f$ classified by $\U^{loc}_\bullet$ and $g$ classified by $\U_\bullet$, $g_*f$ is also classified by $\U^{loc}_\bullet$.
 \end{defone}
 
 \begin{defone}
  Let $\U$ be a universe. An {\it ideal sub-universe} $\U^{loc} \hookrightarrow \U$ is a subuniverse such that for all  
  $f$ classified by $\U^{loc}_\bullet$ and $g$ classified by $\U_\bullet$, $g_*f$ is also classified by $\U^{loc}_\bullet$.
 \end{defone}

 \begin{defone} \label{Def:Factorization system}
  A {\it factorization system} is a choice of two functors $\mathcal{L}, \mathcal{R} : \Arr(\E) \to \Arr(\E)$ such that 
  \begin{enumerate}
   \item For every $f$ we have $f \simeq \mathcal{R}(f) \circ \mathcal{L}(f)$.
   \item If $f: A \to B$ is in the image of $\mathcal{L}$ and $g: Y \to X$ is in the image of $\mathcal{R}$ the commutative square 
   \begin{center}
    \pbsq{map(B,Y)}{map(A,Y)}{map(B,X)}{map(A,X)}{f^*}{g_*}{g_*}{f^*}
   \end{center}
   is a homotopy pullback square of spaces.
  \end{enumerate}
 \end{defone}

 We want to prove following theorem about these definitions.
 
 \begin{theone} \label{The:Localizations}
  Let $\E$, $\U_\bullet$ and $S$ as in Remark \ref{Rem:Intro Condition for Localization Section}. The following data are equivalent: 
  \begin{enumerate}
   \item A reflective subcategory of arrows $L\Arr(\E)$.
   \item A family of reflective subcategories $L\E_{/X}$.
   \item A reflective sub-complete Segal Universe $\U^{loc}_\bullet$.
   \item An ideal sub-complete Segal universe $\U^{loc}_\bullet$.
   \item An ideal sub-universe $\U^{loc}$
   \item A factorization system $\mathcal{L},\mathcal{R}: \Arr(\E) \to \Arr(\E)$.
  \end{enumerate}
 \end{theone}

 \begin{proof}
  { \it (1) $\Leftrightarrow$ (2) }
  The target fibration $\Arr(\E) \to \E$ classifies the over-categories functor that takes an object $X$ to  $\E_{/X}$.
  Thus, a map of target fibrations corresponds to a natural transformation from $\E_{/X}$ to itself. This immediately implies that 
  the desired equivalence
  
  {\it (2) $\Leftrightarrow$ (3)}
  This follows immediately from the Yoneda Lemma for representable Cartesian fibrations (Corollary \ref{Cor:Yoneda Lemma for CSU}).
  
  {\it (4) $\Leftrightarrow$ (5)}
  A full sub-complete Segal universe of $\U_\bullet$ is determined by a sub-universe of $\U_0$.
  
  {\it (2) $\Rightarrow$ (4)}
  By definition of a family of reflective subcategories for any morphism $f: Y \to X$, we get a functor 
  $$f_*: L(\E_{/Y})^S \to L(\E_{/X})^S$$
  which gives us the desired result.
  
  {\it (5) $\Rightarrow$ (2)}
  We want to prove that there are localization functors $L: \E_{/X} \to L(\E_{/X})$ such that the diagram in Definition 
  \ref{Def:Family of Refl Subcat} commutes. To prove that such a functor exists it suffices to prove it for the case $L: \E \to \E$.
  This argument is quite involved and thus is given separately in Proposition \ref{Prop:Sub Univ gives Trunc}. 
  \par 
  It remains to show that the diagram in Definition 
  \ref{Def:Family of Refl Subcat} commutes. By uniqueness of adjoints it suffices to check the right adjoints commute. 
  However, the commutativity of the right adjoint follows immediately from being an ideal sub-universe.
  
  {\it (2) $\Leftrightarrow$ (6)}
  First let $\mathcal{R}: \Arr(\E) \to \Arr(\E)$ be defined as $iL$, where $i$ is the right adjoint to our localization $L$. 
  Note $L$ is an adjunction and thus has a unit map 
  $$u: \Arr(\E) \to \Sq(\E)$$
  where $\Sq(\E)$ is the $(\infty,1)$-category of commutative squares in $\E$. 
  Because $L$ is a functor over $\E$. The image of a map $f: X \to Y$ takes the form.
  \begin{center}
   \comsq{X}{L_YX}{Y}{Y}{}{}{}{id_Y}
  \end{center}
   We define $\mathcal{L}: \Arr(\E) \to \Arr(\E)$ as the $u$ precomposed with restricting to the top arrow $X \to L_YX$.
   \par 
   Notice we immediately have the desired factorization $f \simeq \mathcal{R}(f) \circ \mathcal{L}(f)$.
   Thus, it remains to prove the second condition of a factorization system
   \par
   We can characterize maps in the image of $\mathcal{L}$ as all maps $f: Y \to X$ such that for all maps $g: Z \to X$ 
   in the image of $\mathcal{R}$, the induced map 
   $$ map_{/X}(X,Z) \to map_{/X}(Y,Z)$$
   is an equivalence. We will prove this separately in Lemma \ref{Lemma:Image of L have Lf triv}.
   \par 
   Let $f: A \to B$ be in the image of $\mathcal{L}$ and $g: Y \to X$ in the image of $\mathcal{R}$.
   Then we have following diagram
   \begin{center}
    \begin{tikzcd}[row sep=0.5in, column sep=0.5in]
     map_{/X}(B,Y) \arrow[r] \arrow[d, "(f^*)_j"] \arrow[dr, phantom, "\ulcorner", very near start] & map(B,Y) \arrow[d, "h"] \\
     map_{/X}(A,Y) \arrow[r] \arrow[d] \arrow[dr, phantom, "\ulcorner", very near start] & map(B,X) \times_{map(A,X)} map(A,Y) \arrow[d, "\pi_1"] \\
     * \arrow[r, " \{ j \} "] & map(B,X)  
    \end{tikzcd}
   \end{center}
   According to the previous paragraph the map $(f^*)_j$ is an equivalence, which implies that the $h$ is an equivalence 
   giving us the desired result.
   \par 
   On the other hand, let $(\mathcal{L},\mathcal{R})$ be a factorization system. We want to prove that $\mathcal{R} : \Arr(\E) \to \Arr(\E)$
   is the desired adjunction. First notice that $\mathcal{R}$ is a map over $\E$ as it has to preserve the target of each morphism.
   It remains to show that it is a left adjoint. Thus, we need to prove that for a given map $f: Y \to X$ and for a map $g: Y \to X$ in 
   the image of $\mathcal{R}$ we get an equivalence 
   $$map(\mathcal{R}(f),g) \simeq map(f,g)$$
   However, this follows immediately from the fact that the map $f \to L_YX$ is in $\mathcal{L}$ which we will prove in 
   Lemma \ref{Lemma:Image of L have Lf triv}.
 \end{proof}

 The remainder of this section is focused on finishing the proof of Theorem \ref{The:Localizations}.
 This requires us to introduce some notations and definitions.
 
 \begin{notone}
  Recall that $\E$ is locally Cartesian closed. We denote the internal mapping functor by 
  $$\E_X(-,-) : (\E_{/X})^{op} \times \E_{/X} \to \E_{/X}$$
  in particular if $X$ is the final object we denote it by 
  $$\E(-,-): \E^{op} \times \E \to \E$$
 \end{notone}
 
 \begin{remone}
  Note that if $\U^{loc}$ is an ideal sub-universe then for any two maps $f: Y \to X$ and $g: Z \to X$, if $g$ is classified by $\U^{loc}$, then 
 $\E_{/X}(f,g) \to X$ is also classified by $\U^{loc}$. In particular, the subcategory classified by $\U^{loc}$ is locally Cartesian closed.
 \end{remone}

 \begin{defone} \label{Def Loc into Univ}
  Let $p: E \to B$ be any morphism in $\E$. Define $\T_p: \E \to \E$ as 
  $$\T_p: \E \xrightarrow{ - \times B } \E_{/B} \xrightarrow{ \E_{/B}(-,p) } (\E_{/B})^{op} \xrightarrow{ \E_{/B}(-,p) } \E_{/B} 
  \xrightarrow{ \Pi_B } \E$$
 \end{defone}

 \begin{notone}
  If the map is $p: E \to B$ is determined by the codomain $B$ then we denote the functor by $\T_B$.
 \end{notone}
  
 \begin{intone} \label{Int:Internal Right Kan Extension}
  In order to better understand our internal characterization of localizations it is helpful to take another look at localization 
  from the perspective of right Kan extensions. 
  Let $\E$ be a category and $L\E$ be a full subcategory. We want to define a left adjoint to the inclusion map 
  $$i : L\E \hookrightarrow \E$$
  Thus, for a given object $C$ we want an object $LC$ such that for any object $D$ in $L\E$ there is an equivalence 
  $$Map_\E(LC,D) \to Map_\E(C,D)$$
  One way would be to simply define $LC$ as the universal object with maps $LC \to D$ for any map $C \to D$. 
  This can be done explicitly via following limit diagram:
  $$LC = \underset{\underset{D \in L\E}{C \to D}}{\lim} D= \lim(L\E_{C/} \to L\E)$$
  We can think of the limit diagram as the right Kan extension:
  \begin{center}
   \begin{tikzcd}[row sep=0.5in, column sep=0.5in]
    L\E \arrow[d, "i"] \arrow[r, "id"] & L\E \\
    \E \arrow[ur, dashed]
   \end{tikzcd}
  \end{center}
  A point in such a limit object exactly has all the desired data of maps $C \to D$ for all objects $D$ in $L\E$.
  If such a limit existed, it would exactly give us the desired localization functor.
  \par 
  However, the problem is that we are taking limit over the diagram $L\E$, which is usually not a small diagram
  (for example if we want to define $n$-truncations then the diagram is certainly not small if the category is not small, for $n \geq 0$).
  Thus, the limit cannot be defined directly. 
  \par 
  Here we would usually employ the language of presentability. If we assume that our category $\E$ has a {\it set} of generators, 
  then we can restrict our original diagram to a small sub diagram and take the limit on that sub diagram to get the desired limit. 
  \par 
  We want to take an alternative approach and internalize the whole construction. We will internally form the diagram above and then 
  use the internal limit to construct the right Kan extension. The benefit of the internal construction is that we don't have to worry about 
  any size issues and everything can be formally defined and proven.
  \par 
  Concretely, we take following steps:
  \begin{enumerate}
   \item Start with an object $A$.
   \item Parametrize it over the local objects by forming the product $\pi_2: A \times \U^{loc} \to \U^{loc}$.
   \item Form the desired indexing category for our limit, by taking the collection of maps $A \to D$, where $D$ is a 
   local object. We achieve this by forming the internal exponential 
   $$\E_{/\U^{loc}}(A \times \U^{loc},\U^{loc}_*).$$
   \item Choosing the target of our diagram, which takes every diagram to its target. We achieve this by taking 
   another internal exponentiation
   $$\E_{\U^{loc}}(\E_{/\U^{loc}}(A \times \U^{loc},\U^{loc}_*), \U^{loc}_*).$$
   \item Constructing the internal limit of this diagram with the given values by taking sections of the diagram
   $$\Pi_{\U^{loc}}(\E_{\U^{loc}}(\E_{/\U^{loc}}(A \times \U^{loc},\U^{loc}_*), \U^{loc}_*)).$$
  \end{enumerate}
  The goal of the proof is to use formal methods to show that this actually constructs the desired localization functor.
 \end{intone}

 \begin{remone}
  Special thanks to Asaf Horev for giving me this intuition about the result.
 \end{remone}

 \begin{propone} \label{Prop:Sub Univ gives Trunc}
  Let $\U^{loc}$ be a ideal sub-universe of $\U$ with universal fibration $p_{\U^{loc}}: \U^{loc}_* \to \U^{loc}$. 
  Moreover, let us denote the class of morphisms classified by $\U$ as $S$ and $\U^{loc}$ by $S^{loc}$. 
  Then 
  \begin{enumerate}
   \item $\T_{\U^{loc}}: \E^S \to \E^S$ takes value in the full subcategory $\E^{S^{loc}}$.
   \item The restricted functor 
   $$\T_{\U^{loc}}: \E^S \to \E^{S^{loc}}$$ 
   is the left adjoint to the inclusion functor $\E^{S^{loc}} \to \E^S$.
  \end{enumerate}
 \end{propone}
 
 \begin{proof}
  Throughout this proof we fix an object $A$ in $\E^S$.
  
  {\it (1):} 
  We need to show that $\T_{\U^{loc}}(A)$ is classified by $\U^{loc}$. It suffices to show that 
  $$\E(\E(\pi_2: A \times \U^{loc} \to \U^{loc}, p_{\U^{loc}} : \U^{loc}_* \to \U^{loc}), p_{\U^{loc}}: \U^{loc}_* \to \U^{loc}) \to \U^{loc}$$
  is classified by $\U^{loc}$. However, this follows immediately from the fact that $\U^{loc}$ is ideal and $p_{\U^{loc}}$ is 
  classified by $\U^{loc}$.
  
  {\it (2):} 
  Before we begin the proof we need some notations and constructions.
  First, notice, we have an adjunction 
   \begin{center}
    \adjun{\E}{\E_{/ \U^{loc}}}{- \times \U^{loc}}{\prod_{\U^{loc}}}
   \end{center}
  which is the last step of Definition \ref{Def Loc into Univ}. This gives us an equivalence 
  $$ map(A, \T_{\U^{loc}}B) \simeq map_{/\U^{loc}}(\U \times A, \E_{/\U^{loc}}(\E_{/\U^{loc}}(B \times \U^{loc}, \U^{loc}_*), \U^{loc}_*))$$
  Concretely this equivalence comes from unit and counit maps, which we will denote by 
  $$\eta_A: A \to A^{\U^{loc}}  $$
  $$\epsilon_B: \U^{loc} \times \Pi_{\U^{loc}} B \to B$$
  Moreover, the unit and counit maps give us equivalences 
  \begin{center}
   \begin{tikzcd}[row sep=0.5in, column sep=0.5in]
    map_\E(A,\Pi_{\U^{loc}}B) \arrow[r, "- \times \U^{loc}"] \arrow[rr, bend right = 10, "\simeq"'] & 
    map_{/\U^{loc}}(\U^{loc} \times A, \U^{loc} \times \Pi_{\U^{loc}}B) \arrow[r, "(\epsilon_B)_*"] & 
    map_{/\U^{loc}}(\U^{loc} \times A, B) \\
    map_{\U^{loc}}(A \times \U, B) \arrow[r, "\Pi_{\U^{loc}}"] \arrow[rr, bend right = 10, "\simeq"'] & 
    map_\E(A^{\U^{loc}}, \Pi_{\U^{loc}}B) \arrow[r, "(\eta_A)^*"] & 
    map_\E(A, \Pi_{\U^{loc}}B) 
   \end{tikzcd}
  \end{center}
  We will use these adjunctions and their unit and counit throughout the proof.
  \par
  Next, we will construct two maps.
  Notice that we have equivalences
  $$map(A,\T_{\U^{loc}}(A)) \simeq map_{/\U^{loc}}(A \times \U^{loc}, \E_{/\U^{loc}}(\E_{/\U^{loc}}(A \times \U^{loc}, \U^{loc}_*), \U^{loc}_*)) \simeq $$
  $$map_{/\U^{loc}}((A \times \U^{loc}) \times_{\U^{loc}} \E_{/\U^{loc}}(A \times \U^{loc}, \U^{loc}_*), \U^{loc}_*) \simeq $$
  $$map_{/\U^{loc}}(\E_{/\U^{loc}}(A \times \U^{loc}, \U^{loc}_*),\E_{/\U^{loc}}(A \times \U^{loc}, \U^{loc}_*))$$
  Given that we have the identity map on the right hand side, we have map 
  $$u: A \to \T_{\U^{loc}}(A)$$
  this induces a map 
  $$(u \times id)^*: \E_{/\U^{loc}}(\T_{\U^{loc}}(A) \times \U^{loc} , \U^{loc}_*) \to \E_{/\U^{loc}}(A \times \U^{loc}, \U^{loc}_*)$$
  On the other hand, we have equivalences
  $$map(\T_{\U^{loc}}(A), \T_{\U^{loc}}(A)) \simeq 
  map_{/\U^{loc}}(\T_{\U^{loc}}(A) \times \U^{loc}, \E_{/\U^{loc}}(\E_{/\U^{loc}}(A \times \U^{loc}, \U^{loc}_*), \U^{loc}_*)) \simeq $$
  $$ map_{/\U^{loc}}((\T_{\U^{loc}}(A) \times \U^{loc}) \times_{\U^{loc}} \E_{/\U^{loc}}(A \times \U^{loc}, \U^{loc}_*), \U^{loc}_*) \simeq $$
  $$ map_{/\U^{loc}}(\E_{/\U^{loc}}(A \times \U^{loc}, \U^{loc}_*),\E_{/\U^{loc}}(\T_{\U^{loc}}(A) \times \U^{loc}, \U^{loc}_*))$$
  The identity map on the left hand side gives us a map 
  $$\Lprime: \E_{/\U^{loc}}(A \times \U^{loc}, \U^{loc}_*) \to \E_{/\U^{loc}}(\T_{\U^{loc}}(A) \times \U^{loc} , \U^{loc}_*)$$
  We will show that these two maps are inverses of each other.
  \par 
  First, we show that 
  $$ (u \times id)^* \circ \Lprime \simeq id$$
  Notice we have a diagram 
  \begin{center}
   \begin{tikzcd}[row sep=0.5in, column sep=1.5in]
    \U^{loc} \times \T_{\U^{loc}}A \arrow[r, "\epsilon_{\E_{\U^{loc}}(\E_{\U^{loc}}(\U^{loc} \times A, \U^{loc}_*),\U^{loc}_*)) }"] & 
    \E_{\U^{loc}}(\E_{\U^{loc}}(\U^{loc} \times A, \U^{loc}_*), \U^{loc}_*) \\
    \U^{loc} \times A \arrow[ur, "ev"'] \arrow[u, "id \times u"]
   \end{tikzcd}
  \end{center}
  which is just the explicit construction of the adjoint, as explained in the first paragraph of this proof. 
  Adjoining the right hand $\E_{\U^{loc}}(\U^{loc} \times A, \U^{loc}_*)$ we get
  \begin{center}
   \begin{tikzcd}[row sep=0.5in, column sep=1in]
    (\U^{loc} \times \T_{\U^{loc}}A) \underset{\U^{loc}}{\times} \E_{\U^{loc}}(\U^{loc} \times A, \U^{loc}_*) \arrow[r] & 
    \U^{loc}_* \\
    (\U^{loc} \times A) \underset{\U^{loc}}{\times} \E_{\U^{loc}}(\U^{loc} \times A, \U^{loc}_*)  \arrow[ur, "ev"'] \arrow[u, "id \times u \times id"]
   \end{tikzcd}
  \end{center}
  Now, by keeping $\E_{\U^{loc}}(\U^{loc} \times A, \U^{loc}_*)$ on the left hand side and adjoining the rest to the other side, we get
  \begin{center}
   \begin{tikzcd}[row sep=0.5in, column sep=0.5in]
    \E_{\U^{loc}}(\U^{loc} \times A, \U^{loc}_*) \arrow[r, "\Lprime"] \arrow[dr, "id"'] &
    \E_{\U^{loc}}(\U^{loc} \times \T_{\U^{loc}} A, \U^{loc}_*) \arrow[d, "(u \times id)^*"] \\ 
    & \E_{\U^{loc}}(\U^{loc} \times A, \U^{loc}_*) 
   \end{tikzcd}
  \end{center}
  which proves one side of the argument.
  \par 
  Now we will prove that 
  $$\Lprime \circ (u \times id)^* \simeq id$$
  We have following diagram
  \begin{center}
   \begin{tikzcd}[row sep=0.5in, column sep=0.5in]
    \E_{\U^{loc}}(A \times \U^{loc}, \U^{loc}_*) \arrow[r, "\Lprime"] & \E_{\U^{loc}}(\T_{\U^{loc}}A \times \U^{loc}, \U^{loc}_*) \\
    \E_{\U^{loc}}(\T_{\U^{loc}}A \times \U^{loc}, \U^{loc}_*) \arrow[ur] \arrow[u, "(u \times id)^*"]
   \end{tikzcd}
  \end{center}
  We want to show the diagonal map is equivalent to the identity. We first adjoin the right hand $\T_{\U^{loc}}A \times \U^{loc}$
  to get the diagram 
  \begin{center}
   \begin{tikzcd}[row sep=0.5in, column sep=0.5in]
    \E_{\U^{loc}}(A \times \U^{loc}, \U^{loc}_*) \underset{\U^{loc}}{\times} (\T_{\U^{loc}}A \times \U^{loc})\arrow[r] &  \U^{loc}_* \\
    \E_{\U^{loc}}(\T_{\U^{loc}}A \times \U^{loc}, \U^{loc}_*) \underset{\U^{loc}}{\times} (\T_{\U^{loc}}A \times \U^{loc}) \arrow[ur] 
    \arrow[u, "(u \times id)^* \times id"]
   \end{tikzcd}
  \end{center}
  Now we keep the $(\T_{\U^{loc}}A \times \U^{loc})$ on the left hand side and adjoin the rest over to the diagram
  \begin{center}
   \begin{tikzcd}[row sep=0.5in, column sep=1.5in]
    \T_{\U^{loc}}A \times \U^{loc} \arrow[r, "\epsilon_{\E_{\U^{loc}}(\E_{\U^{loc}}(\U^{loc} \times A, \U^{loc}_*),\U^{loc}_*)) }"] \arrow[dr, "m"'] & 
    \E_{\U^{loc}}(\E_{\U^{loc}}(A \times \U^{loc}, \U^{loc}_*), \U^{loc}_*) 
    \arrow[d, "\E_{\U^{loc}}(\E_{\U^{loc}}(u \times \U^{loc} \comma \U^{loc}_*) \comma \U^{loc}_*)"] \\
    & \E_{\U^{loc}}(\E_{\U^{loc}}(\T_{\U^{loc}}A \times \U^{loc}, \U^{loc}_*), \U^{loc}_*)
   \end{tikzcd}
  \end{center}
  By definition of $\Lprime$ the horizontal map is the counit map, as defined in the beginning of this proof. 
  In order to simplify notation we will denote this counit map simply by $\epsilon$ as the subscript will not change.
  Moreover, we will denote the diagonal map by $m$. Our goal is to prove that $m$ is the evaluation map. 
  We will prove it by finding the adjoint using the counit map.
  \par 
  Applying $\Pi_{\U^{loc}}$ to this diagram 
  \begin{center}
   \begin{tikzcd}[row sep=0.5in, column sep=0.5in]
    \Pi_{\U^{loc}}\T_{\U^{loc}}A \times \U^{loc} \arrow[r, "\Pi_{\U^{loc}}\epsilon"] \arrow[dr, "\Pi_{\U^{loc}} m"'] & 
    \Pi_{\U^{loc}}\E_{\U^{loc}}(\E_{\U^{loc}}(A \times \U^{loc}, \U^{loc}_*), \U^{loc}_*) 
    \arrow[d, "\Pi_{\U^{loc}}\E_{\U^{loc}}(\E_{\U^{loc}}(u \times \U^{loc} \comma \U^{loc}_*) \comma \U^{loc}_*)"] \\
    & \Pi_{\U^{loc}}\E_{\U^{loc}}(\E_{\U^{loc}}(\T_{\U^{loc}}A \times \U^{loc}, \U^{loc}_*), \U^{loc}_*)
   \end{tikzcd}
  \end{center}
  Notice, by definition, we have 
  $$\Pi_{\U^{loc}}\E_{\U^{loc}}(\E_{\U^{loc}}(A \times \U^{loc}, \U^{loc}_*), \U^{loc}_*) = \T_{\U^{loc}}(A)$$ 
  and, similarly, 
  $$\Pi_{\U^{loc}}\E_{\U^{loc}}(\E_{\U^{loc}}(\T_{\U^{loc}}A \times \U^{loc}, \U^{loc}_*), \U^{loc}_*) = \T_{\U^{loc}}(\T_{\U^{loc}}A)$$
  and more generally 
  $$\Pi_{\U^{loc}}\E_{\U^{loc}}(\E_{\U^{loc}}(u \times \U^{loc} \comma \U^{loc}_*) \comma \U^{loc}_*) = \T_{\U^{loc}}(u).$$
  We also have
  $$\Pi_{\U^{loc}}\T_{\U^{loc}}A \times \U^{loc} = \E(\U^{loc},\T_{\U^{loc}}(A)).$$
  Moreover, we can precompose the element in the top left corner with the unit map 
  $$\eta_{[\T_{\U^{loc}}A]}: \T_{\U^{loc}}A \to \Pi_{\U^{loc}} \T_{\U^{loc}}A \times \U^{loc}$$
  to get a new diagram 
  \begin{center}
   \begin{tikzcd}[row sep=0.5in, column sep=0.5in]
    \T_{\U^{loc}}A \arrow[r, "\eta_{\leb \T_{\U^{loc}}A \reb}"] \arrow[rr, "id", bend left=30, shift left =1] 
    \arrow[drr, "\T_{\U^{loc}}(u)"', bend right =10] & 
    \E(\U^{loc},\T_{\U^{loc}}(A)) \arrow[r, "\prod_{\U^{loc}}(\epsilon)"] \arrow[dr, "\prod_{\U^{loc}} (m)"'] & 
    \T_{\U^{loc}}(A) \arrow[d, "\T_{\U^{loc}}(u)"] \\
    & & \T_{\U^{loc}}(\T_{\U^{loc}}(A))
   \end{tikzcd}
  \end{center}
  By the triangle identity of adjunctions, the composition of the two horizontal maps is the identity $\T_{\U^{loc}}A \to \T_{\U^{loc}}A$.
  Thus, the image of the map $m:\T_{\U^{loc}}A \times \U^{loc} \to \E_{\U^{loc}}(\E_{\U^{loc}}(\T_{\U^{loc}}A \times \U^{loc}, \U^{loc}_*), \U^{loc}_*)$
  under the equivalence 
  \begin{center}
   \begin{tikzcd}[row sep=0.5in, column sep=0.5in]
    map_{/ \U^{loc}}(\T_{\U^{loc}}A \times \U^{loc}, \E_{\U^{loc}}(\E_{\U^{loc}}(\T_{\U^{loc}}A \times \U^{loc}, \U^{loc}_*), \U^{loc}_*)) 
    \arrow[d, "\ds\prod_{\U^{loc}}"] \arrow[dd, "\simeq", shift left = 55, bend left =55] \\
    map_{\E}(\Pi_{\U^{loc}}(\T_{\U^{loc}}A \times \U^{loc}), 
    \Pi_{\U^{loc}}\E_{\U^{loc}}(\E_{\U^{loc}}(\T_{\U^{loc}}A \times \U^{loc}, \U^{loc}_*), \U^{loc}_*)) \arrow[d, "(\eta_{\leb \T_{\U^{loc}}A \reb})^*"] \\
    map_{\E}(\T_{\U^{loc}}(A),\Pi_{\U^{loc}}\E_{\U^{loc}}(\E_{\U^{loc}}(\T_{\U^{loc}}A \times \U^{loc}, \U^{loc}_*), \U^{loc}_*))
   \end{tikzcd}
  \end{center}
  is the map $\T_{\U^{loc}}(u)$. However, $u$ was chosen by definition as the map that correspond to the evaluation map. 
  Thus, $m$ is equivalent to the evaluation map. This finishes the proof.
 \end{proof}
 
 \begin{lemone} \label{Lemma:Localized equivalences}
  Let $L: \Arr(\E) \to L\Arr(\E)$ be a localization functor over $\E$ and let $\alpha: f \to g$ be a morphism in $\Arr(\E)$. 
  Then $L\alpha$ is an equivalence if and only if for every object $h$ in $L\Arr(\E)$ the induced map 
  $$\alpha^*: map_{\Arr(\E)}(f,h) \to map_{\Arr(\E)}(g,h)$$
  is an equivalence of spaces.
 \end{lemone}
 
 \begin{proof}
  By the Yoneda lemma $L\alpha$ is an equivalence in $L\Arr(\E)$ if and only if $(L\alpha)^*$ is an equivalence for every object $h$.
  But $L\Arr(\E)$ is a full subcategory of $\Arr(\E)$ and so the result follows.
 \end{proof}

 \begin{lemone} \label{Lemma:Image of L have Lf triv}
  Let $L: \Arr(\E) \to L\Arr(\E)$ be a localization functor over $\E$. Then a map $u$ is the unit map of the adjunction 
  if and only if $L(u)$ is equivalent to the identity.
 \end{lemone}

 \begin{proof}
  If $L(u)$ is equivalent to the identity then $u$ is the unit map of $u$ itself. On other hand let $f: Y \to X$ be an object 
  in $\Arr(\E)$ with value $L(f) = ||f|| \to X$ and $u: Y \to ||f||$ be the unit of the adjunction. By definition 
  of the adjunction we have equivalences 
  $$map_\E(L(f),g) \to map(f,g)$$
  for any object $g:Z \to W$ in $L\Arr(\E)$, which by the previous lemma implies that $L(u)$ is equivalent to the identity.
 \end{proof}

 \begin{remone}
  Notice that Lemma \ref{Lemma:Localized equivalences} and Lemma \ref{Lemma:Image of L have Lf triv} were used in the proof of 
  Theorem \ref{The:Localizations}, but are also of independent interest and were in particular proven for the specific case of 
  truncations (Lemma \ref{Lemma:Tau Equiv if mapped in Trunc} and Lemma \ref{Lemma:Unit Connected}).
 \end{remone}

  The result immediately gives us an alternative way to construct $(-1)$-truncations, assuming we have subobject classifiers. 

 \begin{corone} \label{Cor:Neg One Trunc via SOC}
  Let $\E$ be an $(\infty,1)$-category satisfying the condition of Remark \ref{Rem:Intro Condition for Localization Section} 
  which has a subobject classifier $\Omega$. 
  Then the functor 
  $$\T_\Omega: \E \to \tau_{-1}\E$$ 
  is the $(-1)$-truncation functor.
 \end{corone}
 
  \begin{remone} \label{Rem:Neg One Trunc via SOC or Join}
  In Section \ref{Sec:The Join Construction} we showed how to construct the $(-1)$-truncation 
  if it is locally Cartesian closed with a universe and finite colimits and natural number object (Theorem \ref{The Neg One Adj}). 
  In this section we showed that we can also construct the $(-1)$-truncation if we instead assume it is 
  locally Cartesian closed with a universe and subobject classifier.
 \end{remone}

 We want to generalize this result to higher truncations. However for that we need to construct the 
 sub-universe of $n$-truncated objects.
 We will actually give a method for constructing a sub-universe of $A$-local objects and 
 $n$-truncated objects will then just be a special case.
 
 \begin{defone} \label{Def:A local obj}
  Let $A$ be an object in $\E$. We say an object $X$ is {\it A-local} if the map 
  $$X \to X^A$$
  is an equivalence.
 \end{defone}

 \begin{notone}
  We denote the full subcategory of $A$-local objects as $\E_{loc_A}$.
 \end{notone}

 \begin{exone} \label{Ex:Trunc Sn local}
  Let $n$ be a natural number. Then the $S^{n+1}$-local objects are exactly the $n$-truncated objects.
 \end{exone}

 We want to construct the subuniverse of $A$-local objects.
 
 \begin{propone} \label{Prop:Sub Universe of N Trunc Obj}
  Let $\E$ be locally Cartesian closed category with locally Cartesian closed universe $\U_0$. 
  Let $A$ be an object in $\E$
  Then we can define the sub-universe of $A$-local objects $\U^{loc_A}$ as the image
    $$\U^{loc_A}_0 = \tau_{-1}^{\U_0}( P)$$
  where $P$ is defined as the pullback
  \begin{center}
   \pbsq{P}{\U_0}{\U_0}{\U_1}{}{}{\Delta_n}{eq}
  \end{center}
 \end{propone}

 \begin{proof}
  We first give a better description of the pullback diagram. Let 
  $$\Delta_n: \U_0 \to \U_1$$
  be the map that takes a morphism $Y \to X$ to the $Y^{X \times A} \to Y$ over $X$.
  Moreover, let $eq: \U_0 \to \U_1$ be the degenerate inclusion map and notice that, by completeness, a map is an equivalence 
  if it is in the image of $eq$. We can now define $P$ as the pullback of these two maps:
  \begin{center}
   \pbsq{P}{\U_0}{\U_0}{\U_1}{}{}{\Delta_n}{eq}
  \end{center}
  Then $\tau_{-1}^{\U_0}( P) \hookrightarrow \U_0$ consists of all morphisms $Y \to X$ such that $Y^{X \times A} \xrightarrow{ \ \simeq \ } Y$, 
  which are exactly the $A$-local ones. 
 \end{proof}
 
 We can now combine the results to immediately get following corollary.
 
  \begin{corone} \label{Cor:A Localization Construction}
  Let $\E$ be an $(\infty,1)$-category satisfying the condition of Remark \ref{Rem:Intro Condition for Localization Section}
  with locally Cartesian closed universe $\U$ that classifies $S$.
  Let $A$ be an object and $\U^{loc_A}$ be the sub-universe of $A$-local objects.
  Then $\T_{\U^{loc_A}}: \E^S \to \E^S$ induces a localization functor 
  $$\T_{\U^{loc_A}}: \E^S \to (\E^{loc_A})^S$$
 \end{corone}
 
  \begin{proof}
   We have to prove that $\U^{loc_A}$ is an ideal sub-universe. Let $f: Z \to Y, g:Y \to X$ be in $S$ and $f$ be $A$-local. 
   We have to prove that $g_*f$ is also $A$-local. For any $h:W \to X$ we have an equivalence 
   $$\E_{/X}(h, g_*f) \simeq \E_{/Y}(g^*h,f)$$
   The right hand side $A$-local as $f$ is $A$-local and so the left hand side is $A$-local as well.
 \end{proof}

 Using this for the case $A=S^{n+1}$ and remembering Example \ref{Ex:Trunc Sn local} we get following corollary.
 
 \begin{corone} \label{Cor:Truncation Functor Construction}
  Let $\E$ be an $(\infty,1)$-category satisfying the condition of Remark \ref{Rem:Intro Condition for Localization Section}
  with locally Cartesian closed universe $\U$ that classifies $S$.
  Then $\T_{\U^{\leq n}}: \E^S \to \E^S$ induces a localization functor 
  $$\T_{\U^{\leq n }}: \E^S \to \tau_n\E^S$$
 \end{corone}
  
\subsection{Inductive Construction of Truncations} \label{Subsec:Inductive Construction of Truncations}
  In this subsection we want to give an alternative way to construct the truncation functor 
  using an inductive argument. That is the content of Theorem \ref{The:Constructing Trunctions Inductively}. 
  
  \begin{remone}
   We need two further assumptions
   in addition to the conditions on $\E$ given at the beginning of the section (Remark \ref{Rem:Intro Condition for Localization Section}).
   Concretely: 
   \begin{enumerate}
    \item[(4)] It has a subobject classifier (Definition \ref{Def:SOC}).
   \end{enumerate}
     those conditions also imply
    \begin{enumerate}
     \item[(5)] It has a natural number object (Theorem \ref{The:EHT has NNO}).    
    \end{enumerate}
  \end{remone}

  Before we move on to the main theorem and its proof we want to give an intuition for the construction. 
  
  \begin{remone}
   The intuition outlined here has already been developed in homotopy type theory (\cite{Ri17}, \cite{DORS18}) and serves as motivation 
   for the steps.
  \end{remone}
 
  If $X$ is a space then we can construct $(n+1)$-truncation of $X$ by $n$-truncating the loop spaces $\Omega_x X$.
  This gives us a way to define the $(n+1)$-truncation inductively. However, we cannot use such a point-set definition 
  in an arbitrary $(\infty,1)$-category. We thus need to refine our approach. 
  \par 
  The key observation is the ``undirected Yoneda Lemma for spaces", which states that for any $\kappa$-small space there is an embedding 
  $$K \to (\s^\kappa)^K$$
  which takes a point $k$ to the map of spaces that takes $l$ to the path space $\Path_K(k,l)$.
  \par 
  We can then post-compose with the $n$-truncation functor 
  $$K \to (\s^\kappa)^K \xrightarrow{ \ (\tau_n)^K \ } (\s^\kappa)^K$$
  The $(n+1)$-truncation is then just the image of $K$ in $(\s^\kappa)^K$. 
  This has categorified our initial construction that $n$-truncated the loop space.
  \par 
  We want to generalize this argument to an arbitrary elementary $(\infty,1)$-topos. 
  The key is that we still have an ``undirected Yoneda Lemma for elementary $(\infty,1)$-toposes":
  
  \begin{theone} \label{The:Yoneda Lemma}
  \cite[Theorem 3.3]{Ra18d}
  Let $A$ be a fixed object and $\U$ be a universe that classifies the diagonal map $\Delta:A \to A \times A$.
  Then the induced map 
  $\Y_A : A \to \U^A$
  is $(-1)$-truncated.
 \end{theone}
 
 Thus, the goal is to show that the same argument we outlined for spaces allows us to construct truncations 
 in an arbitrary elementary $(\infty,1)$-topos.
  
  \begin{theone} \label{The:Constructing Trunctions Inductively}
   Let $n$ be a natural number. We can define the $(n+1)$-truncation functor 
   \begin{center}
    \adjun{\E}{\tau_{n+1}\E}{\tau_{n+1}}{i}
   \end{center}
   inductively by $n$-truncating the loop objects.
  \end{theone}
  
  \begin{proof}
   The proof breaks down into several steps. Fix a locally Cartesian closed universe $\U$.
   \begin{enumerate}
    \item First, we show there is a map 
    $$ \Sep: \U^\U \to \U^\U$$
    This is the content of Proposition \ref{Prop:Sep Construction}.
    \item Using the definition of a natural number object (Definition \ref{Def:NNO}) we thus get a map $(\tau_n)^{core}$ that 
    makes following diagram commute
    \begin{center}
     \begin{tikzcd}[row sep=0.3in, column sep=0.6in]
      & \mathbb{N} \arrow[r, "s"] \arrow[dd, "(\tau_n)^{core}"] & \mathbb{N} \arrow[dd, "(\tau_n)^{core}"] \\
      1 \arrow[ur, "o"] \arrow[dr, "\{ 1 \}"'] & & \\
      & \U^\U \arrow[r, "\Sep"] & \U^\U
     \end{tikzcd}
    \end{center}
    Here $\{1\}: \U \to \U$ is the map that takes everything to the final object in $\U$.
     \item We prove that $(\tau_n)^{core}(X)$ is the $n$-truncation of $X$. This is shown in Lemma \ref{Lemma:TauN Core Truncation}. 
     \item This implies that the image of $Im((\tau_n)^{core}(\U)) \simeq \U^{\leq n}$, the sub-universe of $n$-truncated objects.
     \item Using Theorem \ref{The:Localizations} with the ideal sub-universe $\U^{\leq n} \hookrightarrow \U$ we thus get 
     the desired truncation functor 
     $$ \tau_n : \E \to \tau_n \E$$
     giving us the desired result.
   \end{enumerate}
  \end{proof}

   Notice that this approach not only proves the existence of such a truncation, but also gives a very explicit construction.
   In particular, we have following corollary.
  
  \begin{corone}
   Let $\U^{\leq n} \hookrightarrow \U$ be the sub-universe of $n$-truncated objects. Then by universality of localizations we 
   have an equivalence 
   $$(\T_{\U^{\leq n}})^{core} \simeq (\tau_n)^{core}$$
   where $(\tau_n)^{core}$ is the functor constructed in Theorem \ref{The:Constructing Trunctions Inductively}.
  \end{corone}
 
  \begin{remone}
  We have defined the $n$-truncation map $\tau_n: \U \to \U$ for natural number in $\E$. However, a natural number object 
  can have non-standard natural numbers, thus in a general elementary $(\infty,1)$-topos, we have natural number objects that 
  don't exist in a Grothendieck $(\infty,1)$-topos. For an example of such non-standard truncations see 
  Section \ref{Sec:Filter Quotients and Truncations}.
 \end{remone}
 
  \begin{propone} \label{Prop:Sep Construction}
    There is a functor 
   $$\Sep : \U^\U \to \U^\U$$ 
   such that for any map $F: \U \to \U$ and $f: Y \to X$ we have 
   $$\Sep(F)(f) \simeq \tau_{-1}^X(F^X \circ \widehat{\ulcorner \Delta_X \urcorner})$$
  \end{propone}
   
  \begin{notone}
   In order to make the main diagram of the proof more readable we will use the notation $(-)^\simeq$ instead of $(-)^{core}$
   to denote the underlying $(\infty,1)$-groupoid of an $(\infty,1)$-category.
  \end{notone}

  \begin{proof}
   Fix a map $F: \U \to \U$. We want to define $\Sep(F): \U \to \U$. By the Yoneda lemma we can define maps 
   $$Map(X,\Sep(F)): Map(X,\U) \to Map(X,\U)$$
   By the definition of the universe $Map(X,\U) \simeq (\E_{/X})^S$. Thus, it suffices to construct a map 
   $$\Sep_X(F): (\E_{/X})^S \to (\E_{/X})^S$$
   To simplify notation, we will not use the notation $S$ to denote the full subcategory classified by $\U$ and we 
   will work in the case $X =1$. The case for a general $X$ follows similarly.
   \par 
   We have following diagram of spaces.
   
   \hspace{-0.8in}
    \begin{tikzcd}[row sep=0.5in, column sep=0.3in]
     & \Arr(\E)^{\simeq} & & 
     (\Arr(\E) \underset{\E}{\times} \E^{op})^{\simeq} \arrow[r, "\simeq", "\Tw"'] &  
     (\Arr(\E) \underset{\E}{\times} \E)^{\simeq} &  \Arr(\E)^\simeq \arrow[r, "\tau_{-1}"] & \Arr(\E)^\simeq \\
     \E^{\simeq} \arrow[r, twoheadrightarrow, "\simeq"] \arrow[ur, hookrightarrow, "\Diag"] & 
     \Arr^{Diag}(\E)^\simeq \arrow[r, twoheadrightarrow, "\simeq"] \arrow[dr, hookrightarrow, "\Uni"'] \arrow[u, hookrightarrow] & 
     (\E_{/\U})^{\times \comma \simeq} \arrow[r, twoheadrightarrow, "\simeq"] \arrow[d, hookrightarrow] \arrow[ur, "\Adj", hookrightarrow] &
     (\E_{/\U^{(-)^{op}}})^{\simeq} \arrow[r, "\simeq", "\Tw"'] \arrow[u, hookrightarrow] & 
     (\E_{/\U^{(-)}})^{\simeq} \arrow[r, "F^{(-)}"]  \arrow[u, hookrightarrow] & 
     (\E_{/\U^{(-)}})^{\simeq} \arrow[r, "\tau_{-1}"] \arrow[u, hookrightarrow] & 
     \Arr(\E)^\simeq \arrow[r, "s"] \arrow[u, hookrightarrow] & 
     \E^{\simeq}\\
     & & (\E_{/\U})^{\simeq}
    \end{tikzcd}
   
   The arrows in the diagram above can be explained as follows:
   \begin{itemize}
    \item The map  
    $$\Diag: \E^\simeq \to \Arr(\E)^\simeq$$
    is the map that takes an object $X$ to the diagonal map $\Delta_X: X \to X \times X$ and a morphism $f: X \to Y$ 
    to the commutative square 
    \begin{center}
     \comsq{X}{Y}{X \times X}{Y \times Y}{f}{\Delta_X}{\Delta_Y}{f \times f}
    \end{center}
    Let $\Arr^{Diag}(\E)^\simeq$, be the space with $0$-vertices maps of the form $\Delta_X$ and $1$-cells commutative 
    squares given above. Then we get an equivalence of spaces $\E^\simeq \to \Arr^{Diag}(\E)^\simeq$
    \item Let $\Uni$ be the map that takes a morphism $X \to X \times X$ to its corresponding classifying map $X \times X \to \U$. 
    Then the image of $\Uni$ consists of maps $X \times X \to \U$ that classify $\Delta$ and morphisms commutative squares 
    \begin{center}
     \begin{tikzcd}[row sep=0.5in, column sep=0.5in]
      X \times X \arrow[dr, "\ulcorner \Delta_X \urcorner"] \arrow[rr, "f \times f"] & & Y \times Y  \arrow[dl, "\ulcorner \Delta_Y \urcorner"]\\
      & \U &
     \end{tikzcd}
    \end{center}
    which we name $(\E_{/\U})^{\times \comma \simeq}$.
    \item The category $(\Arr(\E) \underset{\E}{\times} \E^{op})$ is defined by the pullback 
    \begin{center}
     \pbsq{\Arr(\E) \underset{\E}{\times} \E^{op}}{\Arr(\E)}{\E}{\E^{op}}{}{}{t}{\U^{(-)}}
    \end{center}
    where $\U^{(-)}$ is the exponent map.
    \item The map $\Adj: (\E_{/\U})^{\times, \simeq} \to \Arr(\E) \underset{\E}{\times} \E^{op}$ is a functorial adjunction map, 
    that takes a point $X \times X \to \U$ to the adjoint $X \to \U^X$. We denote the image of the map of spaces by 
    $(\E_{/\U^{(-)^{op}}})^\simeq$.
    \item Let $(\E^{op})^\simeq \to \E^\simeq$ be a choice of equivalence and denote the induced equivalence 
    $$\Tw: (Arr(\E) \underset{\E}{\times} \E^{op})^\simeq \to (Arr(\E) \underset{\E}{\times} \E^{op})^\simeq$$
    We denote the image of $(\E_{/\U^{(-)^{op}}})^\simeq$ under the map $\Tw$ by $(\E_{/\U^{(-)}})^\simeq$
    \item The map $F^{(-)}$ post-composes a point $X \to \U^X$ with $F^X: \U^X \to \U^X$.
    \item Let $\tau_{-1}: \Arr(\E) \to \Arr(\E)$ be the parameterized $(-1)$-truncation map that takes a morphism 
    $Y \to X$ to the $(-1)$-truncation $\tau_{-1}^X(Y) \to X$. We denote the restriction map to $(\E_{/\U^{(-)}})^\simeq$
    as $\tau_{-1}$ as well.
   \end{itemize}
  Thus, we get a map $\E^\simeq \to \E^\simeq$. The same argument can be applied to get maps $(\E_{/X})^\simeq \to (\E_{/X})^\simeq$
  and thus we get a map 
  $$\Sep(F): \U \to \U$$
  and by its construction it satisfies the desired equivalence given in the statement of the proposition.
  \end{proof}

  \begin{remone}
   It is instructive to see what the map does at the level of a single object.
   Starting with the object $X$ we form following diagram
   \begin{center}
    \begin{tikzcd}[row sep=0.5in, column sep=0.5in]
     X \arrow[d, twoheadrightarrow] \arrow[r, "\ulcorner \Delta_X \urcorner"] & \U^X \arrow[r, "F^X"] & (\U)^X \\
     \Sep(F)(X) \arrow[urr, hookrightarrow]
    \end{tikzcd}
   \end{center}
   So, in particular $\Sep(F)(X)$ is a subobject of $F(\U)^X$.
  \end{remone}

  \begin{lemone} \label{Lemma:TauN Core Truncation}
   The map $(\tau_n)^{core}: \U \to \U$ is the $n$-truncation map.
  \end{lemone}

  \begin{proof}
   We use the fact that a natural number object can also be characterized via an internal induction argument 
   (Theorem \ref{The:Peano NNO}).   
   For $n=-2$, $(\tau_{-2})^{core}$ is the $(-2)$-truncation by definition. Let us assume $(\tau_n)^{core}$ be 
   the $n$-truncation. We want to show that $(\tau_{n+1})^{core}$ is the $(n+1)$-truncation. 
   Fix an object $X$. 
   Using Corollary \ref{Cor Conn is Trunc} it suffices to prove that
   $(\tau_{n+1})^{core}(X)$ is $(n+1)$-truncated and $X \to (\tau_{n+1})^{core}(X)$ is $(n+1)$-connected.
   \par 
   First, notice 
   $$\tau_{n+1}^{core}(X) \hookrightarrow (\U^{\leq n})^X$$
   and $\U^{\leq n}$ is $(n+1)$-truncated (Theorem \ref{The:U Leq n is n plus trunc}) and so $\tau_{n+1}^{core}(X)$ is $(n+1)$-truncated.
   \par 
   We now want to prove that $X \to (\tau_{n+1})^{core}(X)$ is $(n+1)$-connected. 
   To simplify notation we denote $\I = (\tau_{n+1})^{core}(X)$.
   \par
   According to Proposition \ref{Prop Loop n Conn N One Conn}, it suffices to prove that $\Delta_{X \to \I}: X \to X \times_\I X$ is $n$-connected,
   as we already know that the map is $(-1)$-connected. By Proposition \ref{Prop Conn iff Tau Final Object}, this is equivalent to 
   $\tau^{X \times_\I X}_n(X) \to X \times_\I X$ being an equivalence.
   \par 
   Assuming that $\Delta_{X \to \I}$ is classified by $\ulcorner \Delta_{X \to \I} \urcorner : X \times_\I X \to \U$, 
   $\tau_n^{X \times_\I X}(X)$ is classified by $\tau_n \circ \ulcorner \Delta_{X \to \I} \urcorner \to \U^{\leq n}$.
   Similarly $\tau_n^{X \times X}(\Delta_X)$ is classified by $\tau_n \circ \ulcorner \Delta_X \urcorner \to \U^{\leq n}$.
   \par 
   The map $\I \to 1$ induces a map $X \times_\I X \to X \times X$, which gives a commutative diagram 
   \begin{center}
    \begin{tikzcd}[row sep=0.25in, column sep=0.5in]
     X \times_\I X \arrow[dd] \arrow[dr] & \\
     & \U \arrow[r, "\tau_n"] & \U \\
     X \times X \arrow[ur] & & 
    \end{tikzcd}
   \end{center}
   Adjoining the maps we get the diagram 
   \begin{center}
    \begin{tikzcd}[row sep=0.125in, column sep=0.25in]
       & & \leb \U \times \I \to \I\reb^{\leb X \to \I\reb} \arrow[rr, "\tau_n^{X}"] \arrow[dddd] & &
      \leb\U^{\leq n} \times \I \to \I\reb^{\leb X \to \I \reb} \arrow[dddd] \\
      & & & \Im_{X \times_\I X}(X) \arrow[dd, "\simeq"] \arrow[ur, hookrightarrow] \\
     X \arrow[uurr, hookrightarrow] \arrow[ddrr, hookrightarrow] \arrow[urrr, twoheadrightarrow, bend right=8] 
     \arrow[drrr, twoheadrightarrow, bend left=8] &  \\
     & & & \Im(X) \arrow[dr, hookrightarrow] \\
      & & \U^X \arrow[rr, "\tau_n^X"'] & & (\U^{\leq n})^X 
    \end{tikzcd}
   \end{center}
   By the ``undirected Yoneda Lemma for elementary $(\infty,1)$-toposes" (Theorem \ref{The:Yoneda Lemma}) the two maps
   $$X \hookrightarrow \leb \U \times \I \to \I\reb^{\leb X \to \I\reb}$$
   $$X \hookrightarrow \U^X$$
   are inclusions. Thus, the image of $X$ in $[\U^{\leq n} \times \I \to \I]^{[X \to \I]}$ and $(\U^{\leq n})^X$ are equivalent.
   \par 
   This means we get a map 
   $$X  \to \I \to \leb\U^{\leq n} \times \I \to \I\reb^{\leb X \to \I \reb} $$
   where the image of $\I$ is trivial. 
   Adjoining the map we get 
   $$X \times_\I X \xrightarrow{ \ \pi_1 \ } X \to \U^{\leq n}.$$
   where the map $X \to \U^{\leq n}$ is the trivial map.
   \par 
   Pulling back along the universal fibration we thus get a pullback square
   \begin{center}
    \pbsq{\tau_n^{X \times_\I X}(X)}{X}{X \times_\I X}{X}{}{}{}{\pi_1}
   \end{center}
   which implies that $\tau_n^{X \times_\I X}(X) \to X \times_\I X$ is an equivalence, 
   $X \to X \times_\I X$ is $n$-connected and so $X \to \I$ is $(n+1)$-connected.
   Hence $\I$ is the $(n+1)$-truncation of $X$.
  \end{proof}

\section{Blakers Massey Theorem for Modalities} \label{Sec Blakers Massey Theorem for Modalities}
 In this section we prove that every {\it modality} in an elementary $(\infty,1)$-topos satisfies the {\it Blakers-Massey Theorem}
 (Theorem \ref{The Blakers Massey}).
 Then we use that for the modality of $n$-truncated and $n$-connected maps (Corollary \ref{Cor:BMT Truncations}).
 There are various proofs of this statement for Grothendieck $(\infty,1)$-toposes, such as in \cite{Re05}.
 However, the proof in \cite{ABFJ17} stands out in the sense that it makes only minimal use of infinite colimits 
 and presentability.
 Thus, instead of giving a completely new proof we just demonstrate how their proof still holds without 
 assuming that the category is presentable or has any infinite colimits.
 \par 
 In order to make the translation as smooth as possible, we will preserve some of their notation and language.
 
 \begin{remone} \label{Rem:E for Section on BMT}
   In this section $\E$ is an $(\infty,1)$-category that satisfies following conditions:
   \begin{enumerate}
    \item It has finite limits and colimits.
    \item It is locally Cartesian closed (Definition \ref{Def:LCCC}), which implies that 
    colimits are universal (Definition \ref{Def:Colimits Universal}, Lemma \ref{Lemma:LCCC Colimits Universal})
    \item  It satisfies descent (Definition \ref{Def:Descent}).
    \item It has a $(-1)$-truncation functor (Definition \ref{Def:Neg One Truncation}).
   \end{enumerate}
     In Subsection \ref{Subsec:BMT for Truncations} we also need the following condition
    \begin{enumerate}
     \item[(5)] One of these conditions has to hold:
     \begin{enumerate}
      \item It has a natural number object (Definition \ref{Def:NNO}) and $n$-truncation functors.
      \item It has sufficient universes $\U$ (Definition \ref{Def:Sufficient Universes})
      that are closed under finite limits, colimits and locally Cartesian closed (Definition \ref{Def:Closed Universes}).
     \end{enumerate}
    \end{enumerate}
  \end{remone}
 
 \subsection{Modalities} \label{Subsec:Modalities}
 In this subsection we review the important definitions and establish some basic propositions regarding modalities that we need 
 in the next subsection. The notation here will closely follow \cite{ABFJ17}.
  
  In Section \ref{Sec:Constructing Truncations} we defined a factorization system (Definition \ref{Def:Factorization system}).
  Recall, it is a choice of two functors $\mathcal{L}, \mathcal{R} : \Arr(\E) \to \Arr(\E)$ that satisfy two conditions: 
  \begin{enumerate}
   \item For every $f$ we have $f \simeq \mathcal{R}(f) \circ \mathcal{L}(f)$.
   \item If $f: A \to B$ is in the image of $\mathcal{L}$ and $g: Y \to X$ is in the image of $\mathcal{R}$ the commutative square 
   \begin{center}
    \pbsq{map(B,Y)}{map(A,Y)}{map(B,X)}{map(A,X)}{f^*}{g_*}{g_*}{f^*}
   \end{center}
   is a homotopy pullback square of spaces.
  \end{enumerate}
  
  \begin{notone}
   We will denote the factorization of $f:X \to Y$ as $X \to ||f|| \to Y$.
  \end{notone}

 \begin{notone}
  We denote the subcategory of the arrow category $Arr(\E)$ consisting of maps in the image of $\R$ by 
  $Arr(\R)$. Similarly, $Arr(\Lprime)$ is the subcategory generated by the image of $\Lprime$.
 \end{notone}

 We have following facts about factorization systems:
 
 \begin{lemone} \label{Lemma LR adjunctions}
  \cite[Lemma 3.1.5]{ABFJ17} 
  The map $\Lprime: Arr(\E) \to Arr(\Lprime)$ ($\R: Arr(\E) \to Arr(\R)$) is the right (left) adjoint to the inclusion. 
 \end{lemone}
 
 \begin{lemone}
  \cite[Lemma 3.1.6]{ABFJ17}
  \begin{enumerate}
   \item If $fg \in \Lprime$ and $f \in \Lprime$ then $g \in \Lprime$.
   Also, if $fg \in \R$ and $g \in \R$ then $f \in \R$.
   \item $\Lprime(f)$ is an equivalence if and only if $f \in \R$.
   Similarly, $\R(f)$ is an equivalence if and only if $f \in \Lprime$.
  \end{enumerate}
 \end{lemone}
 
 We cannot prove the Blakers-Massey theorem for every factorization system. We need one further condition:
 
 \begin{defone}
  A {\it modality} is a factorization system such that $\Lprime$ is closed under pullback.
 \end{defone}
 
 \begin{defone}
  Let us assume we have following commutative square.
  \begin{center}
   \comsq{Z}{Y}{X}{W}{g}{f}{k}{h}
  \end{center}
  Then we define the {\it cogap map} as the map $\lfloor h,k \rfloor : X \coprod_Z Y \to W$
  and the {\it gap map} $(f,g):  Z \to X \times_W Y$.
 \end{defone}

 \begin{defone}
  Let $f: A \to B$ and $f': A' \to B'$ be two maps in $\E$. We define $f \square f'$, called the {\it pushout product}, in the following diagram
  \begin{center}
   \begin{tikzcd}[row sep = 0.5in, column sep=0.5in]
    A \times A' \arrow[d, "f \times id_{A'}"] \arrow[r, "id_{A} \times f'"] & B \times A' \arrow[d] \arrow[rdd, bend left = 20, "id_B \times f'"] & \\
    A \times B' \arrow[r] \arrow[rrd, bend right = 20, "f \times id_{B'}"] & \ds B \times A' \coprod_{A \times A'} A \times B' \arrow[dr, "f \square f'"] & \\
    & & B \times B'
   \end{tikzcd}
  \end{center}
 \end{defone}
 
 \begin{notone}
  We will denote the relative pushout product construction in $\E_{/Z}$ as $\square_Z$.
 \end{notone}
 
 \begin{defone} \label{Def Local Maps}
  A class of maps $\M$ is {\it local} if in the pullback square where $g$ is $(-1)$-connected
  \begin{center}
   \begin{tikzcd}[row sep=0.5in, column sep=0.5in]
    X' \arrow[r, "g'"] \arrow[d, "f'"] \arrow[dr, phantom, "\ulcorner", very near start] & X \arrow[d, "f"] \\
    Y' \arrow[r, "g", twoheadrightarrow] & Y
   \end{tikzcd}
  \end{center}
  
  we have $f' \in \M$ if and only if $f \in \M$.
 \end{defone}
 
 \begin{remone}
  The definition above differs from the one given \cite[Definition 3.7.1]{ABFJ17}.
  However, the two definitions are equivalent
  in a Grothendieck $(\infty,1)$-topos \cite[Remark 3.7.2]{ABFJ17} \cite[Proposition 6.2.3.14]{Lu09}. 
  More importantly, we will only require the condition stated in Definition \ref{Def Local Maps}
  for the later proofs.
 \end{remone}
 
 \begin{exone}
  Notice Proposition \ref{Prop Equiv Local} implies that equivalences, truncated maps and connected maps are local.
 \end{exone}

 This example is not a coincidence. In fact we want to prove that in every modality the left and right classes 
 of maps are local.
 
 \begin{defone}
  Let $\M$ be a local class of maps and assume we have following commutative square.
  \begin{center}
   \comsq{Z}{Y}{X}{W}{g}{f}{k}{h}
  \end{center}
  Then we say the square is $\M$-Cartesian if the gap map $(f,g)$ is in $\M$.
 \end{defone}

 \begin{propone} \label{Prop LR local}
  $\Lprime$ and $\R$ are local.
 \end{propone}

 \begin{proof}
  It suffices to prove that $\Lprime$ is local. In the pullback square
   \begin{center}
   \begin{tikzcd}[row sep=0.5in, column sep=0.5in]
    X' \arrow[d, "f'"] \arrow[r] \arrow[dr, phantom, "\ulcorner", very near start] & X \arrow[d, "f"] \\
    Y' \arrow[r, twoheadrightarrow, "g"] & Y 
   \end{tikzcd}
  \end{center}
   let $g$ be $(-1)$-connected and $f' \in \Lprime$. We have to show that $f \in \Lprime$.
   Using the factorization of $f$ and $f'$ we get the diagram
   \begin{center}
   \begin{tikzcd}[row sep=0.5in, column sep=0.5in]
    X' \arrow[d, "\Lprime(f')"] \arrow[r] \arrow[dr, phantom, "\ulcorner", very near start] & X \arrow[d, "\Lprime(f)"] \\
    \strut ||f'|| \arrow[d, "\R(f')", "\simeq"'] \arrow[r] \arrow[dr, phantom, "\ulcorner", very near start] & \strut ||f|| \arrow[d, "\R(f)", "\simeq"'] \\
    Y' \arrow[r, twoheadrightarrow, "g"] & Y 
   \end{tikzcd}
  \end{center}
  As $f' \in \Lprime$ we know that $\R(f')$ is an equivalence. As equivalences are local (Proposition \ref{Prop Equiv Local}) this means that 
  $\R(f)$ is also an equivalence. However this means that $f \in \Lprime$ which finishes the proof.
 \end{proof}
 
  \begin{remone}
  This proposition is proven in \cite[Proposition 3.7.5]{ABFJ17}, but the proof given there does not hold in our setting
  and needed to be adjusted.
 \end{remone}
 
 We have now reviewed all the necessary background material and can move on to the statement and the proof of the Blakers-Massey Theorem.
 
 \subsection{Blakers-Massey Theorem for Modalities} \label{Subsec:BMT for Modalities}
 In this subsection we prove the Blakers-Massey theorem for every modality. Let us first state the relevant theorems.
 
 \begin{theone} \label{The Blakers Massey}
 \cite[Theorem 4.1.1]{ABFJ17}
 (Blakers-Massey Theorem) 
 Let $(\Lprime,\R)$ be a modality in an $(\infty,1)$-category $\E$ that satisfies the conditions 
 of Remark \ref{Rem:E for Section on BMT}. Consider the pushout square:
 \begin{center}
  \begin{tikzcd}[row sep=0.5in, column sep=0.5in]
    Z \arrow[d, "f"] \arrow[r, "g"] & Y \arrow[d,"k"] \\
    X \arrow[r, "h"] & W \arrow[ul, phantom, "\ulcorner", very near start] 
  \end{tikzcd}
 \end{center}
 Suppose that $\Delta f \square_Z \Delta g \in \Lprime$. Then the square is $\Lprime$-Cartesian.
\end{theone}

\begin{theone} \label{The Dual Blakers Massey}
 \cite[Theorem 3.6.1]{ABFJ17}
 (Dual Blakers-Massey Theorem) Let $(\Lprime,\R)$ be a modality in an $(\infty,1)$-category $\E$ that satisfies the conditions 
 of Remark \ref{Rem:E for Section on BMT}. 
 Suppose we are given a pullback square
 \begin{center}
  \pbsq{X}{Z}{Y}{W}{h}{f}{g}{k}
 \end{center}
 and suppose that the map $k \square g \in \Lprime$. Then the cogap map $\lfloor k,g \rfloor : Y \coprod_X Z \to W$
 is in $\Lprime$.
\end{theone}

 As there is already a detailed proof in \cite{ABFJ17}. We will simply show that $\E$
 satisfies all the conditions mentioned in the steps of that proof.
 \par 
 Before we do so let us review the structure of the paper \cite{ABFJ17}.
 The first two sections focus on reviewing concepts about higher topos theory. 
 Most of the third section focuses on certain conditions under which modalities exists,
 which are not relevant to our discussion.
 The only results we need from this section are in 
 \cite[Subsection 3.7, Subsection 3.8]{ABFJ17}, where the authors prove a descent condition for modalities. 
 The fourth section then gives us the proofs of the Blakers-Massey theorem. 
 The proofs in \cite[Section 4]{ABFJ17} generalize in a straightforward way, 
 but some proofs in \cite[Subsection 3.7, Subsection 3.8]{ABFJ17}
 need adjustments.
 \par 
 Thus, we will not discuss the proofs in  \cite[Section 4]{ABFJ17} any further. However, we will analyze the steps of the proofs in 
 \cite[Subsection 3.7, Subsection 3.8]{ABFJ17} and, when need be, explain how they can be 
 adjusted to hold for $\E$.

 \begin{enumerate}
  \item \cite[Definition 3.7.1]{ABFJ17}: This definition of a local class of maps does not hold in for $\E$ as we do not have 
  arbitrary coproducts. We instead use \cite[Remark 3.7.2]{ABFJ17} as the definition of local maps (Definition \ref{Def Local Maps}).
  \item \cite[Proposition 3.7.5]{ABFJ17}: This proof does not hold in for $\E$ as we do not have arbitrary coproducts.
  We presented an alternative proof in Proposition \ref{Prop LR local}.
  \item \cite[Lemma 3.8.5]{ABFJ17}: In this lemma we use the fact that $\Lprime$ is local. Notice we only need the locality condition in the sense
  of Definition \ref{Def Local Maps} and not \cite[Definition 3.7.1]{ABFJ17}.
  \item \cite[Lemma 3.8.6]{ABFJ17} This proof depends on \cite[Remark 3.1.4 (5)]{ABFJ17} which states under which conditions an infinite colimit diagram 
  gives us a $(-1)$-connected map. It does not hold for $\E$ in the form stated as we do not have infinite colimits. 
  However, the remark is only applied to the specific pushout $G \coprod_E F$ for which the condition does hold for $\E$
  (as proven in Lemma \ref{Lemma Pushout Cover}).
  \item \cite[Lemma 3.8.7]{ABFJ17} The proof follows by the same argument. Notice that although the proof states that $\Lprime$ has all colimits
  it suffices to have all finite colimits as we are only looking at pushout squares.
 \end{enumerate}
 
 All the remaining results that we need for our proof from \cite[Section 4]{ABFJ17} hold and do not need any adjustments.
 This finishes the proof for any $\E$ that satisfies the conditions of Remark \ref{Rem:E for Section on BMT}.

 \subsection{Blakers-Massey Theorem for Truncations} \label{Subsec:BMT for Truncations}
 Having proven Blakers-Massey Theorem for any modality we can finally show the case for truncations.

 \begin{propone} \label{Prop Conn Trunc Modality}
  Let $n$ be a natural number. Then the class of $n$-truncated maps and $n$-connected maps form a modality.
 \end{propone}
 
 \begin{proof}
  By Corollary \ref{Cor:Truncation Functor Construction} and Proposition \ref{Prop Trunc Conn Lift} it is a factorization system.
  Finally, by Corollary \ref{Cor N Conn Base Change} it is a modality.
 \end{proof}

 Using the Blakers-Massey Theorem in a topos we can recover the classical results. 
 
 \begin{corone} \label{Cor:BMT Truncations}
  (Classical Blakers-Massey Theorem)
  Let us assume we have a pushout square such that $f$ is $m$-connected and $g$ 
  is $n$-connected.
 \begin{center}
  \begin{tikzcd}[row sep=0.5in, column sep=0.5in]
    Z \arrow[d, "f"] \arrow[r, "g"] & Y \arrow[d,"k"] \\
    X \arrow[r, "h"] & W \arrow[ul, phantom, "\ulcorner", very near start] 
  \end{tikzcd}
 \end{center}
 Then, the gap map $(f,g): Z \to X \times_W Y$ is $(m+n)$-connected.
 \end{corone}
 
  \begin{corone}
   {\it (Join Theorem)} We have following diagram
   \begin{center}
    \begin{tikzcd}[row sep=0.5in, column sep=0.5in]
      X \underset{B}{\times} Y \arrow[d] \arrow[r] & Y \arrow[d] \arrow[rdd, bend left = 20, "g"] & \\
      X \arrow[r] \arrow[rrd, bend right = 20, "f"] &  \ds X \coprod_{X \underset{B}{\times} Y} Y \arrow[dr, "h"] & \\
      & & B
    \end{tikzcd}
   \end{center}
   If $f$ is $m$-connected and $g$ is $n$-connected then $h$ is $(m+n+2)$-connected.
  \end{corone}

  \begin{remone}
   This is a direct implication of the Dual Blakers-Massey Theorem, but is called the {\it join theorem}
   in \cite[Proposition 8.15]{Re05}, because when $B$ is the final object then 
   $X \coprod_{X \underset{B}{\times} Y} Y$ is simply the join $X \ast Y$ (Definition \ref{Def Join}). The join theorem then proves that taking successive 
  joins of one object raises the connectivity. This is the intuition we used in Theorem \ref{The Neg Trunc from Join} to construct truncation functors.
  \end{remone}
 
 This also gives us the Freudenthal suspension theorem.
 
 \begin{corone}
  (Freudenthal Suspension Theorem) Let $X$ be an $n$-connected object. Then the map $X \to \Omega \Sigma X$ is 
  $2n$-connected.
 \end{corone}
 
\section{Filter Products and Truncations}\label{Sec:Filter Quotients and Truncations}
 Up to this point we have extensively developed the theory of truncations for elementary $(\infty,1)$-toposes.
 In this section we want study truncations for a specific class of elementary $(\infty,1)$-toposes, 
 namely {\it filter quotients}.
 \par 
 In Subsection \ref{Subsec:Review Filter Hypercomp} we review the concept of filter quotient as introduced in \cite{Ra20}.
 We also discuss hypercomplete toposes and the construction of {\it hypercompletions}, as originally studied in \cite{Lu09}. 
 In Subsection \ref{Subsec:Los Truncation}, we show the relation between truncations of filter products by proving 
 $\Los$'s theorem for truncations (Theorem \ref{The:Trunc Obj filter quotient}, Corollary \ref{Cor:Trunc Obj filter product})
 as well as a theorem about hypercomplete filter products (Theorem \ref{The:Filter Product Hypercomplete}).
 \par 
 Finally, in Subsection \ref{Subsec:Non Standard Truncations} we focus on examples of elementary $(\infty,1)$-toposes 
 that are not presentable, do not have infinite colimits and in particular have non-standard natural number objects. 
 In particular, we will make a case why it is necessary to care about non-standard natural numbers by illustrating
 how ignoring non-standard natural numbers prevents us from having a well-defined hypercompletion.
 
  \begin{remone} \label{Rem:E Conditions Section on Los Theorem}
   In this section $\E$ is an $(\infty,1)$-category that satisfies following conditions:
   \begin{enumerate}
    \item It has finite limits and colimits.
   \end{enumerate}
    for Subsection \ref{Subsec:Los Truncation} we also need the conditions:
   \begin{enumerate}
    \item[(2)] It is locally Cartesian closed (Definition \ref{Def:LCCC}).
    \item[(3)] It has a natural number object (Theorem \ref{The:EHT has NNO}).
    \item[(4)] It has a universe (Definition \ref{Def:Sufficient Universes}) closed under finite limits and colimits
    (Definition \ref{Def:Closed Universes}).
   \end{enumerate}
   which are exactly the conditions we need to define internal truncation levels (Definition \ref{Def:Truncation NNO})
  \end{remone}
  
\subsection{Review of Filter Construction and Hypercompletions} \label{Subsec:Review Filter Hypercomp}
 This is a very short review of the filter construction of $(\infty,1)$-categories and hypercompletion.
 In particular, we define filter quotients (Definition \ref{Def:Filter Quotients}) and prove it preserves the 
 structure of an elementary $(\infty,1)$-topos (Theorem \ref{The:Filter Quotient EHT}).
 For more details on filter quotients see \cite{Ra20}. 
 Finally, we also give an overview of hypercompletions as discussed in \cite{Lu09}.
 
 \begin{defone}
  Let $(P, \leq)$ be a poset. A {\it filter} $F$ is a subset of $P$ that satisfies following three conditions:
  \begin{enumerate}
   \item It is non-empty.
   \item It is down-ward directed, meaning that for any $x,y \in F$ there exists $w \in F$ such that $w \leq x$, $w \leq y$.
   \item It is upward closed, meaning that for any $x \in F$ and $x \leq y$, we have $y \in F$.
  \end{enumerate}
 \end{defone}
 
 We will give two constructions of the filter quotient, depending on which model of $(\infty,1)$-categories we use.
 
 \begin{theone}
  Let $\E$ be a finite complete Kan enriched category and $\Phi$ be a filter in the poset $\Sub(1)$.  
  Then there exists a Kan enriched category $\E_\Phi$ defined as follows: 
  \begin{enumerate}
   \item It has the same objects as $\E$.
   \item For two object $X,Y$ in $\E$ we define the morphisms Kan complexes as 
   $$ Map_{\E_{\Phi}}(X,Y)_n = \left[ \coprod_{U \in \Phi} Map_{\E_{\Phi}}(X \times U,Y \times U)_n \right] / \sim $$
   where the equivalence relation is defined as 
   $$ f \sim g \Leftrightarrow \exists W \in \Phi ( f \times id_W = g \times id_W)$$
  \end{enumerate}
  which comes with a functor $P_\Phi : \E \to \E_\Phi$.
 \end{theone}
 
 \begin{theone}
  Let $\E$ be a finitely complete quasi-category or complete Segal space
  and let $\Phi$ be a filter in $\Sub(1)$.  Then there exists a quasi-category or complete Segal space
  $\E_\Phi$ defined as follows:
  $$\E_\Phi = \underset{U \in \Phi}{\colim} \  \E_{/U}$$
  which comes with a functor $P_\Phi : \E \to \E_\Phi$
 \end{theone}

 \begin{remone}
  The two definitions agree with each other in the sense that if we take an Kan enriched category $\E$, 
  then there is an equivalence 
  $$(N_\Delta(\E))_\Phi \simeq N_\Delta(\E_\Phi)$$
 \end{remone}

 Given that the definitions agree, we can give following definition.
 
 \begin{defone}\label{Def:Filter Quotients}
  We call the $(\infty,1)$-category $E_\Phi$ the {\it filter quotient} 
  of $\E$ with respect to the filter $\Phi$.
 \end{defone}
 
 The key observation is that the filter construction preserves all desirable topos theoretic properties:
 
 \begin{theone} \label{The:Filter Quotient EHT}
  (\cite[Theorem 2.23]{Ra20} \cite[Theorem 2.27]{Ra20})
   Let $\E$ be an $(\infty,1)$-category with finite limits. Then the functor 
   $$P_\Phi: \E \to \E_\Phi$$
   creates 
   \begin{enumerate}
    \item finite limits and colimits
    \item subobject classifiers
    \item complete Segal universes
   \end{enumerate}
   In particular if $\E$ is an elementary $(\infty,1)$-topos (Definition \ref{Def:EHT}), then $\E_\Phi$ is one as well.
 \end{theone}

 \begin{remone}
  As we have shown above, the filter quotient of every elementary $(\infty,1)$-topos is an elementary $(\infty,1)$-topos.
  However, the filter quotient of a Grothendieck $(\infty,1)$-topos need not be necessarily Grothendieck
  again. Thus, we can use the filter quotient construction to construct elementary $(\infty,1)$-toposes 
  that are not Grothendieck $(\infty,1)$-toposes. 
 \end{remone}
 
 In particular we can construct toposes with non-standard natural number objects and thus non-standard 
 truncations, which is the goal of Subsection \ref{Subsec:Non Standard Truncations}. 
 Here we will used the fact that, as we just observed, $P_\Phi$ also preserves natural number objects.
 
 We will also need the notion of hypercompletions.
 
 \begin{defone} \label{Def:Hypercomplete}
  Let $\E$ be an elementary $(\infty,1)$-topos (or any other $(\infty,1)$-category with arbitrary $n$-truncation functors). 
  Then $\E$ is called {\it hypercomplete} if every $\infty$-connected map (Definition \ref{Def:Infinite Connected Map NNO}).
  is an equivalence.
 \end{defone}
 
 \begin{exone}
  The category of spaces is hypercomplete (this is the content of Whitehead's theorem).
 \end{exone}
 
 \begin{exone}
  There are examples of Grothendieck $(\infty,1)$-toposes that are not hypercomplete.
  For a detailed example see \cite[11.3]{Re05}.
 \end{exone}
 
 If a Grothendieck $(\infty,1)$-topos is not hypercomplete, then there is a canonical way to make it so, called the 
 {\it hypercompletion}.
 
 \begin{defone}
  Let $\G$ be a Grothendieck $(\infty,1)$-topos. An object is hypercomplete if it is local with respect to all $\infty$-connected maps.
 \end{defone}

 \begin{defone} \label{Def:Hypercompletion}
  Let $\G$ be a Grothendieck $(\infty,1)$-topos. We call the full subcategory of hypercomplete objects the {\it hypercompletion}
  and denote it by $\G^{\wedge}$.
 \end{defone}
 
 $\G^{\wedge}$ is the universal hypercomplete topos. 
 
 \begin{propone} \label{Prop:Hypercompletion Existence}
  \cite[Proposition 6.5.2.13]{Lu09}
  Let $\G$ be a Grothendieck $(\infty,1)$-topos and $\H$ be a hypercomplete Grothendieck $(\infty,1)$-topos. Then there is an equivalence 
  $$\text{Fun}_*(\H, \G^{\wedge}) \to \text{Fun}_*(\H, \G)$$
 \end{propone}

 Thus, in the theory Grothendieck $(\infty,1)$-toposes, there is a universal hypercomplete $(\infty,1)$-topos.
 
 \begin{remone} \label{Rem:Hypercompletion EHT}
  As of now, there is no analogous construction of hypercompletions of an arbitrary elementary $(\infty,1)$-topos.
  However, we do expect such a construction to exist. In particular, we expect the subcategory of hypercomplete objects 
  to have the structure of a topos again.
 \end{remone}

 We will use this observation in Subsection \ref{Subsec:Non Standard Truncations} to make the case that we need non-standard truncations.
 
\subsection{$\Los$'s Theorem for Truncations of Filter Quotients} \label{Subsec:Los Truncation}
 In this subsection we want to study truncated objects in filter quotients. 
 For that purpose we prove a version of $\Los$'s theorem for truncations. 
 Recall that $\E$ is a finitely complete and cocomplete $(\infty,1)$-category with natural number object and a universe closed 
 under finite limits and colimits (by Remark \ref{Rem:E Conditions Section on Los Theorem}). 
 \par 
 We want to characterize truncated objects in the filter quotient $\E_\Phi$.
 
 \begin{theone} \label{The:Trunc Obj filter quotient}
  Let $n: 1 \to \mathbb{N}$ be a natural number. Then an object $X$ in $\E_\Phi$ is $n$-truncated if and only if 
  there exists $U \in \Phi$ such that $X \times U$ is $n$-truncated in $\E$.
 \end{theone}

 \begin{proof}
  $X$ is $n$-truncated if and only if 
  $$X^{S^{n+1}} \to X$$
  is an equivalence in $\E_\Phi$. However, by \cite[Theorem 3.7]{Ra20} this is only true if there exists $U$ in $\Phi$ such that 
  $$U \times X^{S^{n+1}} \to U \times X$$
  is an equivalence in $\E$. 
  Moreover, notice we have an equivalence $(U \times X)^{S^{n+1}} \simeq U \times X^{S^{n+1}}$. 
  Thus, this is equivalent to $U \times X$ being $n$-truncated in $\E$.
 \end{proof}
 
 One particularly interesting instance of a filter quotient is the filter product. Using the 
 previous theorem we thus immediately get following corollary, 
 which we can think of as $\Los$'s theorem for truncations in a filter product.
 
 \begin{corone} \label{Cor:Trunc Obj filter product}
  Let $\E$ be an elementary $(\infty,1)$-topos such that $1$ has two subobjects.
  Moreover, let $I$ be a set and $\Phi$ a filter on $P(I)$ and  
   fix a natural number $(a_i)_{i \in I}$ in $(\mathbb{N})_{i \in I}$.
  Then an object $(X_i)_{i \in I}$ is $(n_i)_{i \in I}$-truncated if and only if
  $$\{ i: X_i \text{ is } n_i-\text{truncated} \} \in \Phi$$
 \end{corone}

 \begin{exone}
  We can take our elementary $(\infty,1)$-topos to be the $(\infty,1)$-category of spaces. Then the theorem above 
  gives us a useful criterion to determine truncated objects in the filter product $\prod_\Phi \s$.
 \end{exone}
 
 We want to use this observation to prove a hypercompleteness (Definition \ref{Def:Hypercomplete}) result about filter products.
 
 \begin{theone} \label{The:Filter Product Hypercomplete}
  Let $\E$ be a hypercomplete $(\infty,1)$-topos such that $1$ has two subobjects and $\Phi$ is a filter on $I$. 
  Then $\prod_\Phi \E$ is hypercomplete.
 \end{theone}
 
 \begin{proof}
  We want to prove that if $f$ is not an equivalence in $\prod_\Phi \E$, then there exists a truncation level $(n_i)_{i \in I}$ such that 
  $\tau_{{(n_i)}_{i \in I}}(f)$ is not an equivalence as well. By $\Los$'s theorem for equivalences if $f$ is not an equivalence then 
  $$\W =\{ i \in I : f_i \text{ is an equivalence} \} \not\in \Phi$$.
  Let $\{n_i\}_{i \in I}: \{1\}_{i \in I} \to \{\mathbb{N}\}_{i \in I}$ be the natural number defined as follows:
  $$
  \{n_i\}_{i \in I} =
  \begin{cases}
   -2 & \text{ if } i \in \W \\
   a_i & \text{ if } i \not\in \W
  \end{cases}
 $$
 where $a_i: 1 \to \mathbb{N}$ is a truncation level such that $\tau_{a_i}(f_i)$ is not an equivalence, which exists by the assumption 
 that $\E$ is hypercomplete.
 \par 
 We want to prove that $\tau_{n_i}(f_i)$ is not an equivalence in $\prod_\Phi \E$. For that it suffices to note that 
 $$\{ i \in I : f_i \text{ is an equivalence } \} = \{ i \in I : \tau_{n_i}(f_i) \text{ is an equivalence } \} \not\in \Phi$$
 This finishes the proof.
 \end{proof}

\subsection{A Case for Non-Standard Truncation Levels} \label{Subsec:Non Standard Truncations}
 In this subsection we want to focus on one specific filter product, namely the filter product $\prod_\Phi \Kan$ where 
 $\Phi$ is the {\it Fr{\'e}chet filter} on $\mathbb{N}=\{ 0 , 1 , 2 , ... \}$. We will give a very broad description of $\prod_\Phi \Kan$ and
 refer the reader to \cite[Subsection 3.2]{Ra20} for further details.
 \par 
 Recall that $\prod_\Phi \Kan$ is a Kan enriched category with objects $(X_n)_{n \in \mathbb{N}}$ sequences of Kan complexes 
 and morphisms 
 $$Map_{(\Kan^\mathbb{N})_\Phi}((X_n)_{n \in \mathbb{N}},(Y_n)_{n \in \mathbb{N}}) =
  \left[ \coprod_{m \in \mathbb{N}} Map_{\Kan^{\mathbb{N}_{\geq m}}}((X_n)_{n \in \mathbb{N}_{\geq m}}, (Y_n)_{n \in \mathbb{N}_{\geq m}}) 
  \right]/ \sim $$
  where for two maps $f,g$ we have 
  $$ f \sim g \Leftrightarrow \exists N \ (\forall n > N) \ (f_n = g_n)$$
 Here $\mathbb{N}_{\geq m} = \{ m , m+1 , m+2 , ... \}$. Using the results of \cite{Ra20} stated in the previous subsection, we get following results:
 \begin{enumerate}
  \item The natural number object is just the constant sequence of natural numbers $(\mathbb{N})_{n \in \mathbb{N}}$.
  \item A natural number (Definition \ref{Def:Natural Number})
 is then just a sequence class $[(a_n)_{n \in \mathbb{N}}]$, where two sequences are in the same class if and only if they are eventually equal.
 \item The standard natural numbers correspond to the class of eventually constant sequences.
 \item Each such class $[(a_n)_{n \in \mathbb{N}}]$ gives us a truncation level and truncation functor 
 $\tau_{[(a_n)_{n \in \mathbb{N}}]}$ (Theorem \ref{Cor:Truncation Functor Construction}).
 \item $\prod_\Phi \Kan$ is hypercomplete (Theorem \ref{The:Filter Product Hypercomplete}). 
 \end{enumerate}

 We want to demonstrate the necessity of non-standard natural numbers and their truncation levels. 
 For the purposes of this example we need to distinguish between two different notions of $\infty$-connected maps:
 \begin{enumerate}
  \item We say $f$ is {\it Grothendieck $\infty$-connected} if it is $n$-connected only for the standard natural numbers.
  \item We say $f$ is {\it elementary $\infty$-connected} if it is $n$-connected for all natural numbers. 
 \end{enumerate}
 Notice this terminology is \underline{non-standard} and we will {\bf only} use it for this example in this subsection. 
 \par 
 As $\prod_\Phi \Kan$ is hypercomplete, every elementary $\infty$-connected map is necessarily an equivalence.
 However, this does not hold for Grothendieck $\infty$-connected maps.
 First, for a standard natural number $k \in \{ -2, -1, ... \}$ an object $(Z_n)_{n \in \mathbb{N}}$ is $k$-truncated
 if there exists $N$ such that for all $n>N$ $Z_n$ is a $k$-truncated space (Corollary \ref{Cor:Trunc Obj filter product}). 
 \par 
 Thus, a map $f: (X_n)_{n \in \mathbb{N}} \to (Y_n)_{n \in \mathbb{N}}$ is Grothendieck $\infty$-connected if for all 
 $k \in \{-2,-1,0,... \}$ there exists a constant $N_k$ such that for all $n > N_k$
 $f$ is $k$-connected. In particular, $\prod_\Phi \Kan$ has Grothendieck $\infty$-connected maps that are not equivalences, 
 for example the sequence $(S^n)_{n \in \mathbb{N}}$.
 \par 
 As we explained in Subsection \ref{Subsec:Review Filter Hypercomp}, in the theory of Grothendieck $(\infty,1)$-toposes, whenever we have a 
 topos that is not hypercomplete, there is a universal 
 way to localize the Grothendieck $(\infty,1)$-topos to make it hypercomplete, called the {\it hypercompletion}
 \cite[Proposition 6.5.2.13]{Lu09}.
 Thus, if we only use the standard truncation levels for $\prod_\Phi \Kan$, we end up with a non-hypercomplete $(\infty,1)$-topos
 and we would expect a hypercompletion topos. We claim that such a hypercompletion cannot exist running contrary to our intuition 
 about $(\infty,1)$-topos theory. 
 \par 
 Concretely, the objects of the hypercompletion would have to all be sequences of spaces $(X_n)_{n \in \mathbb{N}}$ that are local 
 with respect to all Grothendieck $\infty$-connected maps. These exactly correspond to the sequences which are eventually stably 
 truncated, meaning there exists a standard natural number $k \in \{-2,-1,0,... \}$ and $N$ such that for all 
 $n >N$ $X_n$ is $k$-truncated. As we will demonstrate this collection of objects does not even have finite colimits.
 \par 
 Denote the full subcategory of these local objects by $(\prod_\Phi \Kan)^{tr}$. The object $(S^1)_{n \in \mathbb{N}}$ is clearly 
 in this subcategory as it is $1$-truncated at every level. Let us assume the suspension $\Sigma (S^1)_{n \in \mathbb{N}}$ exists, 
 which we denote by $(C_n)_{n \in \mathbb{N}}$.
 Fix another object $(Y_n)_{n \in \mathbb{N}}$ and notice by assumption there exists a $k \in \{-2,-1,0,... \}$
 such that $(Y_n)_{n \in \mathbb{N}}$ is eventually $k$-truncated. Notice the subcategory $\tau_k \Kan$ is finitely complete 
 and cocomplete 
 and so we get a finitely complete and cocomplete filter product $\prod_\Phi \tau_k\Kan$, 
 which is a full subcategory of $(\prod_\Phi \Kan)^{tr}$. As $\tau_kS^2$ is the suspension of $S^1$ in $\tau_k\Kan$, we have
 $\Sigma (S^1)_{n \in \mathbb{N}} \simeq (\tau_kS^2)_{n \in \mathbb{N}}$ in $\prod_\Phi \tau_k\Kan$.
 \par 
 Now we have following chain of equivalences
 $$ Map_{(\prod_\Phi \Kan)^{tr}}((C_n)_{n \in \mathbb{N}},(Y_n)_{n \in \mathbb{N}}) = $$
 $$Map_{\prod_\Phi \tau_k\Kan}((C_n)_{n \in \mathbb{N}},(Y_n)_{n \in \mathbb{N}}) \simeq 
 Map_{\prod_\Phi \tau_k\Kan}((\tau_kS^2)_{n \in \mathbb{N}},(Y_n)_{n \in \mathbb{N}}) $$
 This implies that if the colimit exists, then $(\tau_kC_n)_{n \in \mathbb{N}} \simeq (\tau_kS^2)_{n \in \mathbb{N}}$
 for all $k \in \mathbb{N}$. But $(C_n)_{n \in \mathbb{N}}$ is in $(\prod_\Phi \Kan)^{tr}$ and thus must be eventually 
 truncated, which would only be possible if $S^2$ is truncated, which is not true \cite{Gr69}.
 This implies that such a $(C_n)_{n \in \mathbb{N}}$ cannot exist.
 \par
 What we just showed is that if we only use the standard truncation levels, not only do we get a topos that is not hypercomplete 
 (which could be reasonable), but get a topos which cannot be hypercompleted, which runs contrary to our understanding 
 of $(\infty,1)$-topos theory. It is only with the non-standard natural numbers that we get a well-behaved $(\infty,1)$-topos.
 \par 
 Notice that in a Grothendieck $(\infty,1)$-topos (or even elementary 
 $(\infty,1)$-topos with countable colimits) the natural number object is necessarily standard and so there was never any need to distinguish 
 between standard and non-standard truncation levels.
 \par 
 Based on this example, we end this section with an observation which naturally leads to a question.
 If $\Phi$ is a non-principal ultrafilter on a set $I$, then the ultraproduct $\prod_\Phi \Kan$ has some semblance to the 
 $(\infty,1)$-category of spaces \cite[Section 4]{Ra20}, but does not necessarily have infinite colimits and thus can have 
 non-standard natural numbers, which leads to non-standard truncations. 
 This suggests following question:
 
 \begin{queone}
  Can we use ultraproducts to construct non-standard models of spaces which have non-standard truncation levels?  
 \end{queone}

\section{Free Universes} \label{Sec:Free Universes}
 In our definition of an elementary $(\infty,1)$-topos (Definition \ref{Def:EHT}), one key condition was the existence of a universe, 
 meaning that for every $f: A \to B$ there exists a universe $\U$ and pullback 

\begin{center}
 \pbsq{A}{\U_*}{B}{\U}{}{f}{}{}
\end{center}

However, the map $B \to \U$ that is classifying $f$ is generally not universal in any sense. 
\par 
In this section we want to prove that using $(-1)$-truncations we can associate a {\it free universe} to $f$. 
The main result of this section is Theorem \ref{The:Smallest Universe EHT}, which proves that under suitable conditions 
on $\E$, we can always make our choice of universe universal and functorial.

\begin{remone}
 This section only depends on Section \ref{Sec:The Join Construction} and concretely the construction of $(-1)$-truncations 
 (Definition \ref{Def:Neg One Truncation}, Theorem \ref{The Neg One Adj}).
 However, we make extensive use of notation introduced in Appendix \ref{Sec:HTT}.
\end{remone}

 \begin{remone} \label{Rem:E Conditions Free Universe Section}
   In this section $\E$ is an $(\infty,1)$-category that satisfies following conditions:
   \begin{enumerate}
    \item It has finite limits and colimits.
    \item It is locally Cartesian closed (Definition \ref{Def:LCCC}).
    \item One of the following conditions hold:
    \begin{enumerate}
     \item There exists a $(-1)$-truncation functor (Definition \ref{Def:Neg One Truncation}).
     \item It has a subobject classifier (Definition \ref{Def:SOC}), (Theorem \ref{Cor:Neg One Trunc via SOC})
     \item It has sufficient universes (Definition \ref{Def:Sufficient Universes} closed under 
     finite limits and colimits (Definition \ref{Def:Closed Universes}) (Theorem \ref{The Neg One Adj})
    \end{enumerate}
     which all imply that there is a $(-1)$-truncation functor.
   \end{enumerate}
 \end{remone}
  
 The key observation that makes the choice of universal universe possible is the following observation from \cite{Ra18a}:
 
 \begin{theone}[\cite{Ra18a}]
  Let $p_\U: \U_* \to \U$ be a univalent universe and $i: \U' \hookrightarrow \U$ a $(-1)$-truncated map.
  Then the pullback $i^*p_\U: i^*\U_* \to \U'$ is univalent as well. 
 \end{theone}

 Thus, for a given morphism $f: A \to B$, we can form following $(-1)$-factorization diagram:

\begin{center}
 \begin{tikzcd}[row sep=0.5in, column sep=0.5in]
  A \arrow[d] \arrow[r] \arrow[dr, phantom, "\ulcorner", very near start] & 
  \U_* \times_\U \tau_{-1}^\U(B) \arrow[r] \arrow[dr, phantom, "\ulcorner", very near start] \arrow[d] & \U_* \arrow[d] \\
  B \arrow[r, twoheadrightarrow] & \tau_{-1}^\U(B) \arrow[r, hookrightarrow] & \U
 \end{tikzcd}
\end{center}

 So, the map $B \to \tau_{-1}^\U(B)$ is the ``smallest" map to a universe that classifies $f: A \to B$, in the sense
 that it is the unique $(-1)$-connected map that classifies $f$. We can thus think of $\tau_{-1}^\U(B)$ as the 
 free universe generated by $f: A \to B$. 
 \par 
 The goal of this section is to show that we can make this construction functorial i.e. we will construct a functor 
 $$\Fr_\U: \O_\E^{(S)} \to \E_{/ \Sub(\U)}$$
 that assigns every morphism $A \to B$ to the map $B \to \tau_{-1}^\U(B)$.
  
 \begin{defone}
  Let $F$ be the functor defined as the composition
  $$F : \E_{/ \U} \xrightarrow{ \ u \ } \Arr(\E_{/ \U}) \xrightarrow{ \ \pi \ } \Arr( \E )$$
  where $u$ is the unit natural transformation that takes an object $X \to \U$ to the object $X \to \tau_{-1}^\U(X) \to \U$
  in $\Arr(\E_{/ \U})$ and $\pi$ is just the projection. 
 \end{defone}
 
 Thus, $F$ takes an object $X \to \U$ to $X \to \tau_{-1}^{\U}(X)$. Notice that by definition $\tau_{-1}^{\U}(X)$ is a 
 subobject of $\U$. We can use this information to refine our construction.
 \par 
 Let $\Sub(\U) = \tau_{-1}(\E_{/\U})$ be the collection of subobjects of $\U$. 
 Then the functor $F$ factors as follows
 
 \begin{center}
  \begin{tikzcd}[row sep=0.5in, column sep=0.5in]
   \E_{/ \U} \arrow[r, "F"] \arrow[dr, dashed, "F"'] & \Arr(\E) \\
    & \Arr(\E) \underset{\E}{\times} \Sub(\U) \arrow[u, hookrightarrow]
  \end{tikzcd}
 \end{center}

 \begin{lemone}
  The functor 
  $$F: \E_{/\U} \to \Arr(\E) \underset{\E}{\times} \Sub(\U)$$
  is an embedding.
 \end{lemone}
 
 \begin{proof}
  Let $X \to \U$, $Y \to \U$ be two objects. We need to prove the map 
  $$map_{/\U}(X, Y) \to map_{\Arr(\E) \underset{\E}{\times} \Sub(\U)}(X \to \tau_{-1}^\U(X), Y \to \tau_{-1}^\U(Y))$$
  is an equivalence of spaces. 
  However right hand side is equivalent to the space 
  $$map_{\Sub(\U)}(\tau_{-1}^\U(X),\tau_{-1}^\U(Y)) \underset{map_\E(X,\U)}{\times} map_\E(Y,\U)$$
  the space $map_{\Sub(\U)}(\tau_{-1}^\U(X),\tau_{-1}^\U(Y))$ is contractible as $\Sub(\U)$ is a $0$-category.
  In particular the projection map 
  $$map_{\Arr(\E) \underset{\E}{\times} \Sub(\U)}(X \to \tau_{-1}^\U(X), Y \to \tau_{-1}^\U(Y)) \xrightarrow{ \ \simeq \ }
  *  \underset{map_\E(X,\U)}{\times} map_\E(Y,\U) = map_{/\U}(X,Y)$$
  is an equivalence, 
  which is an inverse to the map above and proves it is an equivalence as well.
 \end{proof}
 
 \begin{defone}
  Let $\E_{/ \Sub(\U)}$ be the essential image of $F$ in $\Arr(\E) \times_\E \Sub(\U)$.
  By the lemma above, the functor $F: \E_{/ \U} \to \E_{/ \Sub(\U))}$ is an equivalence. 
 \end{defone}
 
 We can describe the essential image as the subcategory of $\Arr(\E) \times_\E \Sub(\U)$ with objects 
 $(-1)$-connected maps.
 Fix an equivalence $G: \O_\E^{(S)} \to \E_{/\U}$ (using the fact that $\U$ is a universe).
 
 \begin{defone}
  Let 
  $$\Fr_\U : \O_{\E}^{(S)} \to \E_{/ \Sub(\U)}$$
  be the composition $FG$. 
 \end{defone}
  
 \begin{remone}
  Let $f: B \to A$ be an object in $\O_{\E}^{(S)}$. Then $\Fr_\U(f): B \to \tau^\U ( \ulcorner  f \urcorner )$.
  We will often abuse notation and denote the target of this morphism by $\Fr_\U(f)$.
  So, in particular we get a map $B \to \Fr_\U(f)$.
 \end{remone}
 
 Up to here we worked with a fixed universe. We want to generalize $\Fr_\U$ by taking colimits.

 \begin{remone}
  Recall from Definition \ref{Def:Univ} that $\Univ$ is the $0$-category of universal fibrations with codomain being universes.
 \end{remone}

 \begin{defone}
  Let $\E_{/ \Univ}$ be the subcategory of $\Arr(\E) \times_\E \Univ$, with objects 
  $(-1)$-connected map $A \to \U$.
 \end{defone}

 Notice, if $\E$ has sufficient universes we have an equivalence 
 $$\E_{/ \Univ} \simeq \underset{\U \in \Univ}{\colim} \E_{/ \U}$$
 but similarly we have an equivalence 
 $$\O_\E \simeq \underset{\U \in \Univ}{\colim} [(\O_\E)_{/p_\U: \U_* \to \U}]$$ 
 which immediately gives us following corollary.
 
 \begin{corone}
  Let $\E$ have sufficient universes. Then there exists an equivalence
  $$\Fr : \O_\E \to \E_{/ \Univ}$$
 \end{corone}

 Using this functor we can give a new equivalent definition of an elementary higher topos.
 
 \begin{theone} \label{The:Smallest Universe EHT}
  Let $\E$ satisfy the conditions of Remark \ref{Rem:E Conditions Free Universe Section}.
  Then the following equivalent:
  \begin{enumerate}
   \item $\E$ has sufficient universes.
   \item There exists a functor 
   $$\Fr_\U: \O_\E \xrightarrow{ \ \simeq \ } \E_{/ \Univ}$$
   that assigns to each morphism its free universe.
  \end{enumerate}
 \end{theone}

 \begin{exone}
  Let us see how this works in the context of spaces. Let $K$ be a $\kappa$-small space, for some cardinal $\kappa$. 
  Then we have pullback square 
  \begin{center}
   \pbsq{K}{\s_*^\kappa}{*}{\s^\kappa}{}{}{}{}
  \end{center}
  In this case $\Fr_{\s^\kappa}(K \to *) \simeq Aut(K)$, the space of self-equivalences of the space $K$. 
  Notice the space $Aut(K)$ is connected and so the map $* \to Aut(K)$ is in fact $(-1)$-connected.
  \par 
  More generally, for an arbitrary map $p:K \to L$, we have
  $$\Fr(p) \simeq \coprod_{l \in \pi_0(L)} Aut(p^{-1}(l))$$
 \end{exone}

 \begin{remone}
  The fact that we can make the assignment functorial can be quite helpful and in particular brings the notion of
  universes closer to our understanding of universes in homotopy type theory. 
  However, we often want our universes to be closed under certain constructions, such as finite colimits and limits
  and the free universe associated to a map is usually not closed (as can be seen in the example above).
 \end{remone}
 
 The remark above naturally leads to following question:
  
 \begin{queone}
  Can we associate to a map $f$ a free universe that is also closed under finite limits, colimits or locally Cartesian closed 
  (Definition \ref{Def:Closed Universes})?
 \end{queone}
 
\section{Propositional Resizing} \label{Sec:Propositional Resizing}
 One condition in the definition of an elementary $(\infty,1)$-topos (Definition \ref{Def:EHT}) 
 is the condition that it has a subobject classifier (Definition \ref{Def:SOC}). 
 This often seems redundant given that we are already assuming the existence of universes. 
 In this section we want to use $(-1)$-truncations and prove that the existence of a subobject classifier $\Omega$ is equivalent to 
 a property of universes (Theorem \ref{The:Prop Resize egal SOC}), which (motivated by homotopy type theory) we call {\it propositional resizing} 
 (Definition \ref{Def:Propositional Resizing}).
 Using this equivalent condition we can give an alternative way to characterize an elementary $(\infty,1)$-topos
 (Theorem \ref{The:EHT SOC egal Prop Resize}).
 
 \begin{remone}
 This section primarily depends on Section \ref{Sec:The Join Construction} and concretely the construction of $(-1)$-truncations 
 (Definition \ref{Def:Neg One Truncation}, Theorem \ref{The Neg One Adj}).
 We do, however, use one observation about universe of $n$-truncated objects 
 (Definition \ref{Def:U leq n}, Lemma \ref{Lemma:U leq n classifies n trunc}).
\end{remone}

 \begin{remone} \label{Rem:E Conditions Propositional Resizing}
   In this section $\E$ is an $(\infty,1)$-category that satisfies following conditions:
   \begin{enumerate}
    \item It has finite limits and colimits.
    \item One of the following conditions hold:
    \begin{enumerate}
     \item There exists a $(-1)$-truncation functor (Definition \ref{Def:Neg One Truncation}).
     \item It is locally Cartesian closed (Definition \ref{Def:LCCC}) and has a subobject classifier (Definition \ref{Def:SOC}), 
     (Theorem \ref{Cor:Neg One Trunc via SOC})
     \item It has sufficient universes (Definition \ref{Def:Sufficient Universes} closed under 
     finite limits and colimits (Definition \ref{Def:Closed Universes}) (Theorem \ref{The Neg One Adj})
    \end{enumerate}
     which all imply that there is a $(-1)$-truncation functor.
   \end{enumerate}
 \end{remone}
  
 In order to relate subobject classifiers and universes, recall following Lemma \ref{Lemma:U leq n classifies n trunc}:
 
 \begin{lemone}
  Let $\U$ be a universe classifying the class of maps $S$. Then $\U^{\leq -1}$ 
  classifies all $(-1)$-truncated maps that are in $S$.
 \end{lemone}
 
 This basic observation motivates following definition.
 
 \begin{defone} \label{Def:Propositional Resizing}
  We say a universe $\U$ satisfies {\it propositional resizing} if the class of morphisms classified 
  by $\U$ contains all $(-1)$-truncated maps.
 \end{defone}
 
 There is an equivalent way to determine whether a universe satisfies propositional resizing.
 
 \begin{lemone}
  Let $\U$ be a universe. Then $\U$ satisfies propositional resizing if and only if for any other morphism 
  in $\Univ$ (Definition \ref{Def:Univ}), $\U \to \V$, the induced map 
  $$\U^{\leq -1} \to \V^{\leq -1}$$
  is an equivalence.
 \end{lemone}

 \begin{proof}
  For a given universe $\V$, we will denote the class of morphism it classifies as $S_\V$.
  Moreover, recall that for any universe $\V$ and any object $B$ we have an equivalence 
  $$map(B, \V^{\leq -1}) \xrightarrow{ \ \simeq \ } (\tau_{-1}((\E_{/B})^{S_{\V}}))^{core} $$
  We will use the notation and equivalence throughout the proof.
  \par 
  Let us assume that $\U$ satisfies propositional resizing and we have a map $\U \to \V$ in $\Univ$.
  Fix an object $B$. Then we have following diagram
  \begin{center}
   \begin{tikzcd}[row sep=0.5in, column sep=0.25in]
    & (\tau_{-1}(\E_{/B}))^{core}  & \\
    (\tau_{-1}((\E_{/B})^{S_{\U}}))^{core} \arrow[rr] \arrow[ur, "\simeq"] & & 
    (\tau_{-1}((\E_{/B})^{S_{\V}}))^{core} \arrow[ul, hookrightarrow] \\
    map(B, \U^{\leq -1}) \arrow[rr] \arrow[u, "\simeq"] & & map(B, \V^{\leq -1}) \arrow[u, "\simeq"]
   \end{tikzcd}
  \end{center}
  The map $(\tau_{-1}((\E_{/B})^{S_{\V}}))^{core} \to (\tau_{-1}(\E_{/B}))^{core}$ is an equivalence as it is 
  $(-1)$-truncated and its precomposition is an equivalence, which means it is also $(-1)$-connected. 
  This implies that the top horizontal map is an equivalence, which means the bottom horizontal map is an equivalence as well.
  As $B$ was arbitrary, this gives us the desired result.
  \par 
  On the other hand, let us assume $\U$ satisfies the condition given in the lemma. We have to show that every
  $(-1)$-truncated map $i:A \to B$ is the pullback of the universal fibration $p_\U: \U_* \to \U$.
  We know there exists a universe $\V$ such that $i:A \to B$ is a pullback of $p_\V: \V_* \to \V$.
  Let $\W$ be a universe that classifies $\U \coprod \V$, which also implies that it classifies $i:A \to B$,
  and notice there is a map $\U \to \W$. Then by assumption we have an equivalence 
  $\U^{\leq -1} \to \W^{-1}$, which gives us an equivalence 
  $$map(B, \U^{\leq -1}) \to map(B,\W^{-1})$$
  which, by the equivalence stated in the beginning of this proof, gives us an equivalence 
  $$(\tau_{-1}((\E_{/B})^{S_{\U}}))^{core} \xrightarrow{ \ \simeq \ } (\tau_{-1}((\E_{/B})^{S_{\W}}))^{core}$$
  As the map $i$ is a point on the right hand side, by the equivalence it is also a point on the left hand side, 
  meaning $i$ is classified by $\U$ as well.
 \end{proof}

 \begin{remone}
  The term propositional resizing was introduced as an axiom in homotopy type theory with the goal of studying 
  subobjects (there known as {\it mere propositions} \cite[Subsection 3.5]{UF13}). 
  Notice the definition given there corresponds to the equivalent 
  condition given in the previous lemma rather than the definition we gave (Definition \ref{Def:Propositional Resizing}). 
 \end{remone}

 We can now relate propositional resizing to subobject classifiers
 
 \begin{theone} \label{The:Prop Resize egal SOC}
  $\U$ satisfies propositional resizing if and only if $\U^{\leq -1}$ is a subobject classifier.
 \end{theone}

 \begin{proof}
  If $\U$ satisfies propositional resizing, then we have an equivalence 
  $$map(X,\U^{\leq -1}) \cong \Sub(X)$$
  which proves that $\U^{\leq -1}$ satisfies the universal property of a subobject classifier. 
  \par 
  On the other hand, let us assume $\U^{\leq -1}$ is a subobject classifier and $i: A \to B$ is an arbitrary mono. 
  As, $\U^{\leq -1}$ is a subobject classifier it classifies $i$ which gives us following diagram 
  \begin{center}
   \begin{tikzcd}[row sep=0.5in, column sep=0.5in]
    A \arrow[d, "i"] \arrow[r] \arrow[dr, phantom, "\ulcorner", very near start] & 
    \U^{\leq -1}_* \arrow[r] \arrow[d] \arrow[dr, phantom, "\ulcorner", very near start] & \U_* \arrow[d] \\
    B \arrow[r] & \U^{\leq -1} \arrow[r] & \U
   \end{tikzcd}
  \end{center}
  which proves that $i$ is also classified by $\U$, which proves that $\U$ satisfies propositional resizing.
 \end{proof}
 
 Having proven the theorem we can give an alternative characterization of an elementary $(\infty,1)$-topos 
 which does mention subobject classifiers:
 
 \begin{theone} \label{The:EHT SOC egal Prop Resize}
  Let $\E$ be $(\infty,1)$-category, which satisfies all conditions of Definition \ref{Def:EHT} except for the existence of 
  subobject classifiers. Then the following are equivalent:
  \begin{enumerate}
   \item $\E$ is an elementary $(\infty,1)$-topos.
   \item There exists a universe $\U$ that satisfies propositional resizing.
  \end{enumerate}
 \end{theone}
 
 \begin{proof}
  $\E$ is a topos if and only $\E$ has a subobject classifier. 
  If $\E$ has a subobject classifier $\Omega$ then any universe that classifies 
  $1 \to \Omega$ also classifies all $(-1)$-truncated maps and thus satisfies propositional resizing.
  \par 
  On the other hand, if $\U$ satisfies propositional resizing, then $\U^{\leq -1}$ is a subobject classifier.
 \end{proof}

\appendix

\section{Higher Topos Theory} \label{Sec:HTT}
 In this section we will give a review of the main topos theoretic conditions we use throughout.
 We will only give definitions and refer the interested reader to the appropriate sources.
 \par 
 Our goal is to work with an elementary $(\infty,1)$-topos, which we define thusly:
 
 \begin{defone} \label{Def:EHT}
  Let $\E$ be an $(\infty,1)$-category. Then we say $\E$ is an {\it elementary $(\infty,1)$-topos} if it satisfies following conditions:
  \begin{enumerate}
   \item It has finite limits and colimits.
   \item It is locally Cartesian closed (Definition \ref{Def:LCCC}).
   \item It has a subobject classifier (Definition \ref{Def:SOC}).
   \item It has sufficient universes (Definition \ref{Def:Sufficient Universes}) which are closed under 
   finite limits, colimits and locally Cartesian closed (Definition \ref{Def:Closed Universes}).
  \end{enumerate}
 \end{defone}
 
 The goal of this appendix is to give a review of the conditions stated above and give some implications of these results 
 that are used throughout.

 \begin{defone} \label{Def:LCCC}
  We say an $(\infty,1)$-category $\E$ is {\it locally Cartesian closed} if it has finite limits and for every map 
  $f: x \to y$ in $\E$ the pullback back map 
  $$f^*: \E_{/y} \to \E_{/x}$$
  has a right adjoint
 \end{defone}
 
 One important implication of being a left adjoint is preservation of colimits.
 
 \begin{defone} \label{Def:Colimits Universal}
  Let $\E$ be an $(\infty,1)$-category with finite limits and colimits. 
  We say {\it colimits in $\E$ are universal} if the pullback of a colimits diagram is a colimit diagram.
 \end{defone}

 \begin{lemone} \label{Lemma:LCCC Colimits Universal}
  If $\E$ is a locally Cartesian closed $(\infty,1)$-category with finite colimits. 
  Then colimits in $\E$ are universal.
 \end{lemone}

 \begin{proof}
  This follows immediately from the fact that left adjoints commute with colimits.
 \end{proof}

 \begin{remone}
  If $\E$ is a presentable $(\infty,1)$-category then being locally Cartesian closed is equivalent to colimits being universal.
  However, as we do not make such an assumption, colimits being universal is a strictly weaker condition.
 \end{remone}

 \begin{defone} \label{Def:Core}
  We denote the {\it maximal underlying $(\infty,1)$-groupoid} of an $(\infty,1)$-category $\E$ by $\E^{core}$. 
 \end{defone}

 \begin{defone} \label{Def:Fibrations}
  Let $\E$ be an $(\infty,1)$-category with finite limits. Then 
  \begin{enumerate}
   \item $\Arr(\E)$ is the {\it $(\infty,1)$-category of morphism} in $\E$, which comes with a target map $t:\Arr(\E) \to \E$, 
   which is a Cartesian fibration.
   \item $\O_\E$ is the subcategory of $\Arr(\E)$ with the same objects, but where the morphisms are necessarily 
   pullback squares. The restricted target map $t: \O_\E \to \E$ is a right fibration.
   \item  $\Emb(\E)$ is the subcategory of $\O_\E$ where objects are mono maps. The restricted target map
   $t:\Emb(\E) \to \E$ is also a right fibration, but where the fiber over each point is a set.
  \end{enumerate}
 \end{defone}
 
 \begin{remone}
  The Cartesian fibration $t: \Arr(\E) \to \E$ classifies the functor 
  $$\E^{op} \to \cat_\infty$$
  that takes an object $x$ to the over-category $\E_{/x}$.
  \par 
  The restriction $t: \O_\E \to \E$ classifies the functor 
  $$\E^{op} \to \s$$
  that takes an object $x$ to the space $(\E_{/x})^{core}$
  \par 
  The restriction $t: \Emb(\E) \to \E$ classifies the functor 
  $$\E^{op} \to \set$$
  that takes an object $x$ to the set $\Sub(x)$.
 \end{remone}
 
 As there are three levels of functors, there are three levels of representability.
 
 \begin{defone} \label{Def:SOC}
  Let $\E$ be an $(\infty,1)$-category with finite limits. A {\it subobject classifier} $\Omega$ is an object that classifies $\Emb(\E)$.
  Concretely, $\Omega$ is a subobject classifier if there is a universal mono $1 \to \Omega$ that is final in $\Emb(\E)$.
 \end{defone}
 
 Informally, we express the subobject classifier condition as 
 $$Map(X,\Omega) \cong \Sub(X)$$
 
 We cannot just represent $\O_\E$ the way we just represented $\Emb(\E)$, as it is too large. Rather 
 we need a class of representing objects.
 
 \begin{defone} \label{Def:Univ}
  Let $\Univ$ the full subcategory of $(-1)$-truncated objects in $\O_\E$. The objects $p_\U:\U_* \to \U$ in $\Univ$ are called 
  {\it universal fibrations} and the target $\U$ of each such fibration is a {\it universe}.
 \end{defone}

 \begin{defone} \label{Def:Sufficient Universes}
  Let $\E$ be an $(\infty,1)$-category with finite limits. Then we say $\E$ has {\it sufficient universes} if there is an equivalence 
  $$\O_\E \simeq \underset{\U \in \Univ}{\colim} (\O_\E)_{/p_\U}$$
  This is equivalent to the statement that the projection maps 
  $$\{(\O_\E)_{/p_\U} \to \O_\E \}_{\U \in \Univ}$$
  are jointly surjective.
 \end{defone}

 More concretely, each morphism $f:Y \to X$ is an object in a full subcategory $(\O_\E)_{/p_\U}$.
 \par
 As we are now choosing our universes, we often need them to be closed under certain properties.
 
 \begin{defone} \label{Def:Closed Universes}
  Let $P$ be a categorical property from the following list:
  \begin{enumerate}
   \item having finite limits
   \item having finite colimits
   \item being locally Cartesian closed
  \end{enumerate}
  Then we say $\U$ is {\it closed under property} $P$, if the full subcategory $(\O_\E)_{/p_\U}$ is closed under property $P$.
 \end{defone}

 Usually representing $\O_\E$ suffices for our purposes, but sometimes we need to represent the whole Cartesian fibration 
 $\Arr(\E)$.
 
 \begin{remone}
  We use the notation $\Arr(\E)_{/\U}$ for the full subcategory of $\Arr(\E)$ whose objects are in $(\O_\E)_{/\U}$.
 \end{remone}

 \begin{theone}[\cite{Ra18b}] \label{The:CSU Exist}
  Let $\E$ be a finitely complete $(\infty,1)$-category and $\U$ a locally Cartesian closed universe.
  Then there exists a simplicial object 
  $$\U_\bullet: \Delta^{op} \to \E$$
  such that $\U_0 \simeq \U$ and there is an equivalence of Cartesian fibrations
  $$\E_{/ \U_\bullet} \simeq \Arr(\E)_{/\U}.$$
  Concretely, we have equivalences 
  $$\U_1 \simeq [id_\U \times p_\U: \U \times \U_* \to \U \times \U ]^{[p_\U \times id_\U: \U_* \times \U \to \U \times \U ]}$$
  and 
  $$\U_n \simeq \U_1 \underset{\U_0}{\times} ... \underset{\U_0}{\times} \U_1$$ 
 \end{theone}

 \begin{defone} \label{Def:CSU}
  We call any such simplicial object constructed this way a {\it complete Segal universe}.
 \end{defone}

 Complete Segal universes can be used to give an alternative characterization of elementary $(\infty,1)$-toposes.
 
 \begin{theone}[\cite{Ra18b}] \label{The:EHT via CSU}
  An $(\infty,1)$-category $\E$ is an elementary $(\infty,1)$-topos if and only if it satisfies following conditions:
  \begin{enumerate}
   \item It has finite limits and colimits.
   \item It has a subobject classifier.
   \item It has sufficient complete Segal universes closed under finite limits, colimits and locally Cartesian closed.
  \end{enumerate}
 \end{theone}

 \begin{notone} \label{Not:S morphisms classified by U}
  If $\U$ is a universe then we denote the class of maps classified by $\U$ as $S_\U$ or, if there is no possibility for confusion, 
  as $S$. In particular, we denote the full subcategory of $\Arr(\E)$ classified by $\U_\bullet$ as
  $\Arr(\E)^{S_\U}$. Similarly, we denote the fiber of $\Arr(\E)^{S_\U}$ as $(\E_{/X})^{S_\U}$. 
 \end{notone}
 
 One benefit of using complete Segal universes is that we have a Yoneda lemma for representable Cartesian fibrations \cite{Ra17}.
 
 \begin{theone} \label{The:Yoneda Lemma for RepCartFib}
  Let $\E_{/\U^1_\bullet}$ and $\E_{/\U^2_\bullet}$ be two representable Cartesian fibration. 
  Then we have an equivalence 
  $$Map_{/\E}(\E_{/\U^1_\bullet}, \E_{/\U^2_\bullet}) \simeq map_{\E^{\Delta^{op}}}(\U^1_\bullet , \U^2_\bullet)$$
 \end{theone}

 Thus in an elementary $(\infty,1)$-topos, we get following result.
 
 \begin{corone} \label{Cor:Yoneda Lemma for CSU}
  Let $\E$ be an elementary $(\infty,1)$-topos and $\U_\bullet$ a complete Segal universe classifying the collection of morphisms $S_\U$.
  Then we have an equivalence 
  $$ Map_{/\E}(\Arr(\E)^{S_\U}, \Arr(\E)^{S_\U}) \simeq map_{\E^{\Delta^{op}}}(\U_\bullet, \U_\bullet)$$
 \end{corone}

 One of the important implications of universes is descent. It was first studied in \cite{Re05} and \cite{Lu09}.
 For a helpful review see \cite{ABFJ17}.
 
 \begin{defone}\label{Def:Descent}
  Let $\E$ be a locally Cartesian closed $(\infty,1)$-category with finite limits and colimits. We say $\E$ satisfies {\it descent}
  if $\O_\E$ is closed under colimits.
 \end{defone}
 
 \begin{remone} \label{Rem:Description of Descent}
  Concretely, this means that if we have a commutative square in $\Arr(\E)$
  \begin{center}
   \begin{tikzcd}[row sep=0.5in, column sep=0.5in]
    f \arrow[r, Rightarrow, "\alpha"] \arrow[d, Rightarrow, "\beta"'] & g \arrow[d, Rightarrow, "\gamma"] \\
    h \arrow[r, Rightarrow, "\phi"'] & k
   \end{tikzcd}
  \end{center}
   such that $\alpha$ and $\beta$ are pullback squares, then $\gamma$ and $\phi$ are also pullback squares.
 \end{remone}

 We can use the fact that representable fibrations have colimits to get following result:
 
 \begin{theone}[\cite{Ra18b}] \label{The:Universes give Descent}
  Let $\E$ be a locally Cartesian closed $(\infty,1)$-category with finite limits and colimits.
  If $\E$ has universes, then it satisfies descent.
 \end{theone}

 If $\E$ is presentable then satisfying descent is equivalent to universes, as is the main result in \cite{Re05}.
 However, in general descent is a weaker condition than having universes

\section{Natural Number Objects} \label{Sec:Appendix NNO}
 We make extensive use of natural number objects in Section \ref{Sec:The Join Construction}, Section \ref{Sec Truncation Functors} 
 and Section \ref{Sec:Filter Quotients and Truncations}, but they also come up in Section \ref{Sec:Constructing Truncations} 
 and Section \ref{Sec Blakers Massey Theorem for Modalities}.
 Thus, we want to give a quick review of natural number objects. For more details on natural number objects in 
 $(\infty,1)$-toposes see \cite{Ra18c} and for more details about natural number objects in elementary $1$-toposes see 
 \cite{Jo03}.
 
 \begin{remone}
  We will not discuss here how we use natural number objects to define internal sequential colimits.
  As we only need it in Subsection \ref{Subsec:Join and Trunc} we cover it right before in Subsection \ref{Subsec:Sequential Colimits via NNO}.
 \end{remone}

 \begin{defone} \label{Def:NNO}
  Let $\E$ be an $(\infty,1)$-category. A triple $(\mathbb{N}, o: 1 \to \mathbb{N}, s: \mathbb{N} \to \mathbb{N})$ is a 
  {\it natural number object} if it is initial among such triples. More concretely, this means for any other triple $(X, b: 1 \to X, u: X \to X)$
  the space of maps $f: \mathbb{N} \to X$ making the following diagram commute 
  
  \begin{center}
   \begin{tikzcd}[row sep=0.3in, column sep=0.6in]
    & \mathbb{N} \arrow[r, "s"] \arrow[dd, "f", dashed] & \mathbb{N} \arrow[dd, "f", dashed] \\
    1 \arrow[ur, "o"] \arrow[dr, "b"] & & \\
    & X \arrow[r, "u"] & X 
   \end{tikzcd}
  \end{center}
 \end{defone}

 However, there we can define a natural number object in a way that is closer to our intuition about induction.
 
 \begin{theone} \label{The:Peano NNO}
  A triple $(\mathbb{N}, o: 1 \to \mathbb{N}, s: \mathbb{N} \to \mathbb{N})$ is a natural number object if and only if 
  the following hold:
  \begin{enumerate}
   \item $s$ is a mono.
   \item $o$ and $s$ are disjoint subobjects of $\mathbb{N}$
   \item Any subobject $\mathbb{N}'$ closed under $s$ and $o$, meaning we have a commutative diagram of the form
    \begin{center}
     \begin{tikzcd}[row sep=0.3in, column sep=0.6in]
      & \mathbb{N}' \arrow[r, "s"] \arrow[dd, hookrightarrow] & \mathbb{N}' \arrow[dd, hookrightarrow] \\
      1 \arrow[ur, "o"] \arrow[dr, "o"] & & \\
      & \mathbb{N} \arrow[r, "s"] & \mathbb{N} 
     \end{tikzcd}
    \end{center}
   is isomorphic to $\mathbb{N}$.
  \end{enumerate}
 \end{theone}

 \begin{defone} \label{Def:Natural Number}
  A {\it natural number} is a map $1 \to \mathbb{N}$. 
 \end{defone}

 Every natural number object has certain natural numbers, known as the {\it standard natural numbers}.
 
 \begin{exone}
  Let $\mathbb{N}$ be a natural number object. Then any natural number of the form 
  $$s \circ s \circ ... \circ s \circ o :1 \to \mathbb{N}$$
  is called a standard natural number.
 \end{exone}
 
 Sometimes these natural numbers recover all natural numbers in which case our natural number object is standard.
 
 \begin{defone} \label{Def:Standard NNO}
  A natural number object $\mathbb{N}$ is {\it standard} if the family of morphisms 
  $s \circ s \circ ... \circ s \circ o : 1 \to \mathbb{N}$ is jointly surjective. 
 \end{defone}

 \begin{exone}
  Let $\E$ have countable colimits. Then $\coprod_\mathbb{N} 1$ is a natural number object and is always standard.
 \end{exone}
 
 \begin{remone}
  If $\mathbb{N}$ is a standard natural number object in $\E$, then this does not imply that 
  $$Map_\E(1,\mathbb{N}) \cong \{ 0,1,2,...\}$$
  For an example take the category of $\mathbb{R}$ indexed spaces, $\s^{\mathbb{R}}$. Then we have
  $$Map_{\s^\mathbb{R}}(1,\mathbb{N}) \cong \mathbb{N}^\mathbb{R}.$$
  and so a natural number is a function $f:\mathbb{R} \to \mathbb{N}$.
  The standard natural numbers $n$ are exactly the constant sequences. The fact that this natural number object is standard 
  corresponds to the fact that every map $f: \mathbb{R} \to \mathbb{N}$ is determined by its values. 
  In other words we have 
  $$f = \underset{n \in \mathbb{N}}{\coprod} n 1_{\{ f = n\}}.$$
 \end{remone}

 Thus, we can only ever hope to have non-standard natural number objects in a topos without infinite colimits.
 For an example of such a topos with non-standard natural number object see
 Subsection \ref{Subsec:Non Standard Truncations}.

 Finally, we make extensive use of following key result in \cite{Ra18c}.
 
 \begin{theone} \label{The:EHT has NNO}
  Let $\E$ satisfy following conditions:
  \begin{enumerate}
   \item It has finite limits and colimits
   \item It is locally Cartesian closed
   \item It has a subobject classifier
   \item It has a universe closed under finite limits and colimits.
  \end{enumerate}
  Then $\E$ has a natural number object.
 \end{theone}

\end{document}